\documentclass[12pt,pdftex,a4paper]{amsart}

\selectfont

\usepackage{qtree}
\usepackage{forest}

\usepackage{caption}

\usepackage{etoolbox}
\usepackage{epic,eepic,ecltree}
\usepackage{graphicx}
\usepackage{tikz}
\usetikzlibrary{decorations.text}
\usepackage{amssymb,color, euscript, enumerate}
\usepackage{amsthm}
\usepackage{amsmath}
\usepackage{braket}
\usepackage{amscd}
\usepackage{txfonts}
\usepackage{comment}
\usepackage{bm}
\usepackage{amscd}
\numberwithin{equation}{section}

\makeatletter
\let\old@tocline\@tocline
\let\section@tocline\@tocline
\newcommand{\subsection@dotsep}{4.5}
\newcommand{\subsubsection@dotsep}{4.5}
\patchcmd{\@tocline}
  {\hfil}
  {\nobreak
     \leaders\hbox{$\m@th
        \mkern \subsection@dotsep mu\hbox{.}\mkern \subsection@dotsep mu$}\hfill
     \nobreak}{}{}
\let\subsection@tocline\@tocline
\let\@tocline\old@tocline

\patchcmd{\@tocline}
  {\hfil}
  {\nobreak
     \leaders\hbox{$\m@th
        \mkern \subsubsection@dotsep mu\hbox{.}\mkern \subsubsection@dotsep mu$}\hfill
     \nobreak}{}{}
\let\subsubsection@tocline\@tocline
\let\@tocline\old@tocline

\let\old@l@subsection\l@subsection
\let\old@l@subsubsection\l@subsubsection

\def\@tocwriteb#1#2#3{%
  \begingroup
    \@xp\def\csname #2@tocline\endcsname##1##2##3##4##5##6{%
      \ifnum##1>\c@tocdepth
      \else \sbox\z@{##5\let\indentlabel\@tochangmeasure##6}\fi}%
    \csname l@#2\endcsname{#1{\csname#2name\endcsname}{\@secnumber}{}}%
  \endgroup
  \addcontentsline{toc}{#2}%
    {\protect#1{\csname#2name\endcsname}{\@secnumber}{#3}}}%

\newlength{\@tocsectionindent}
\newlength{\@tocsubsectionindent}
\newlength{\@tocsubsubsectionindent}
\newlength{\@tocsectionnumwidth}
\newlength{\@tocsubsectionnumwidth}
\newlength{\@tocsubsubsectionnumwidth}
\newcommand{\settocsectionnumwidth}[1]{\setlength{\@tocsectionnumwidth}{#1}}
\newcommand{\settocsubsectionnumwidth}[1]{\setlength{\@tocsubsectionnumwidth}{#1}}
\newcommand{\settocsubsubsectionnumwidth}[1]{\setlength{\@tocsubsubsectionnumwidth}{#1}}
\newcommand{\settocsectionindent}[1]{\setlength{\@tocsectionindent}{#1}}
\newcommand{\settocsubsectionindent}[1]{\setlength{\@tocsubsectionindent}{#1}}
\newcommand{\settocsubsubsectionindent}[1]{\setlength{\@tocsubsubsectionindent}{#1}}

\renewcommand{\l@section}{\section@tocline{1}{\@tocsectionvskip}{\@tocsectionindent}{}{\@tocsectionformat}}%
\renewcommand{\l@subsection}{\subsection@tocline{2}{\@tocsubsectionvskip}{\@tocsubsectionindent}{}{\@tocsubsectionformat}}%
\renewcommand{\l@subsubsection}{\subsubsection@tocline{3}{\@tocsubsubsectionvskip}{\@tocsubsubsectionindent}{}{\@tocsubsubsectionformat}}%
\newcommand{\@tocsectionformat}{}
\newcommand{\@tocsubsectionformat}{}
\newcommand{\@tocsubsubsectionformat}{}
\expandafter\def\csname toc@1format\endcsname{\@tocsectionformat}
\expandafter\def\csname toc@2format\endcsname{\@tocsubsectionformat}
\expandafter\def\csname toc@3format\endcsname{\@tocsubsubsectionformat}
\newcommand{\settocsectionformat}[1]{\renewcommand{\@tocsectionformat}{#1}}
\newcommand{\settocsubsectionformat}[1]{\renewcommand{\@tocsubsectionformat}{#1}}
\newcommand{\settocsubsubsectionformat}[1]{\renewcommand{\@tocsubsubsectionformat}{#1}}
\newlength{\@tocsectionvskip}
\newcommand{\settocsectionvskip}[1]{\setlength{\@tocsectionvskip}{#1}}
\newlength{\@tocsubsectionvskip}
\newcommand{\settocsubsectionvskip}[1]{\setlength{\@tocsubsectionvskip}{#1}}
\newlength{\@tocsubsubsectionvskip}
\newcommand{\settocsubsubsectionvskip}[1]{\setlength{\@tocsubsubsectionvskip}{#1}}

\patchcmd{\tocsection}{\indentlabel}{\makebox[\@tocsectionnumwidth][l]}{}{}
\patchcmd{\tocsubsection}{\indentlabel}{\makebox[\@tocsubsectionnumwidth][l]}{}{}
\patchcmd{\tocsubsubsection}{\indentlabel}{\makebox[\@tocsubsubsectionnumwidth][l]}{}{}

\newcommand{\@sectypepnumformat}{}
\renewcommand{\contentsline}[1]{%
  \expandafter\let\expandafter\@sectypepnumformat\csname @toc#1pnumformat\endcsname%
  \csname l@#1\endcsname}
\newcommand{\@tocsectionpnumformat}{}
\newcommand{\@tocsubsectionpnumformat}{}
\newcommand{\@tocsubsubsectionpnumformat}{}
\newcommand{\setsectionpnumformat}[1]{\renewcommand{\@tocsectionpnumformat}{#1}}
\newcommand{\setsubsectionpnumformat}[1]{\renewcommand{\@tocsubsectionpnumformat}{#1}}
\newcommand{\setsubsubsectionpnumformat}[1]{\renewcommand{\@tocsubsubsectionpnumformat}{#1}}
\renewcommand{\@tocpagenum}[1]{%
  \hfill {\mdseries\@sectypepnumformat #1}}

\let\oldappendix\appendix
\renewcommand{\appendix}{%
  \leavevmode\oldappendix%
  \addtocontents{toc}{%
    \protect\settowidth{\protect\@tocsectionnumwidth}{\protect\@tocsectionformat\sectionname\space}%
    \protect\addtolength{\protect\@tocsectionnumwidth}{2em}}%
}
\makeatother

\makeatletter
\settocsectionnumwidth{2em}
\settocsubsectionnumwidth{2.5em}
\settocsubsubsectionnumwidth{3em}
\settocsectionindent{1pc}%
\settocsubsectionindent{\dimexpr\@tocsectionindent+\@tocsectionnumwidth}%
\settocsubsubsectionindent{\dimexpr\@tocsubsectionindent+\@tocsubsectionnumwidth}%
\makeatother

\settocsectionvskip{10pt}
\settocsubsectionvskip{0pt}
\settocsubsubsectionvskip{0pt}
    
\settocsectionformat{\bfseries}
\settocsubsectionformat{\mdseries}
\settocsubsubsectionformat{\mdseries}
\setsectionpnumformat{\bfseries}
\setsubsectionpnumformat{\mdseries}
\setsubsubsectionpnumformat{\mdseries}

\let\oldtableofcontents\tableofcontents
\renewcommand{\tableofcontents}{%
  \vspace*{-\linespacing}% Default gap to top of CONTENTS is \linespacing.
  \oldtableofcontents}

\newcommand{\Ell}{\mathrm{Ell}}
\newcommand{\CP}{\C P^1}

\newcommand{\im}{\mathrm{Im}\,}

\newcommand{\bJ}{\bar{J}}
\newcommand{\bc}{\bar{c}}
\newcommand{\btau}{\bar{\tau}}
\newcommand{\bQ}{\bar{Q}}
\newcommand{\bG}{\bar{G}}

\newcommand{\ad}{\mathrm{ad}}

\usetikzlibrary{cd, decorations.pathmorphing}
\newcommand{\cH}{\mathcal{H}}

\newcommand{\Hom}{{\mathrm{Hom}}}

\newcommand{\Z}{\mathbb{Z}}
\newcommand{\R}{\mathbb{R}}
\newcommand{\C}{\mathbb{C}}

\newcommand{\Om}{\Omega}

\newcommand{\std}{{\text{std}}}
\newcommand{\sft}{{\text{sf}}}
\newcommand{\ind}{{\mathrm{Ind}}}

\newcommand{\va}{\bm{1}}
\newcommand{\vac}{\bm{1}}

\newcommand{\id}{{\mathrm{id}}}

\newcommand{\z}{{\bar{z}}}

\newcommand{\h}{{\bar{h}}}
\newcommand{\p}{{\bar{p}}}

\newcommand{\uz}{\underline{z}}

\newcommand{\pa}{{\partial}}

\newcommand{\al}{\alpha}
\newcommand{\bal}{\bar{\alpha}}
\newcommand{\ep}{\epsilon}
\newcommand{\be}{\beta}
\newcommand{\ga}{\gamma}
\newcommand{\bga}{\bar{\gamma}}

\newcommand{\D}{\bar{D}}
\newcommand{\om}{\omega}
\newcommand{\la}{\lambda}

\newcommand{\omb}{{\bar{\omega}}}
\newcommand{\si}{\sigma}

\newcommand{\lat}{{\mathrm{lat}}}
\newcommand{\hY}{\hat{Y}}
\newcommand{\tY}{\tilde{Y}}

\newcommand{\g}{{\mathfrak{g}}}

\newcommand{\ft}{\frac{1}{2}}

\newcommand{\Ld}{{\overline{L}}}

\newcommand{\Aut}{\mathrm{Aut}\,}
\newcommand{\tw}{{{I\hspace{-.1em}I}}}

\newcommand{\End}{\mathrm{End}}

\newcommand{\td}{{\bar{t}}}

\newcommand{\sVect}{{\underline{\mathrm{sVect}}}}

\newcommand{\sHilb}{{\underline{\mathrm{sHilb}_f}}}
\newcommand{\Bord}{{\underline{\mathrm{Bord}}_2^{\text{ori}}}}

\newcommand{\sll}{{\mathrm{sl}}}

\newcommand{\bare}{{\bar{0}}}
\newcommand{\baro}{{\bar{1}}}

\newcommand{\op}{{\mathrm{op}}}

\newcommand{\Ind}{{\mathrm{Ind}}}

\def\top{{\text{top}}}

\makeatletter
\NewDocumentCommand{\extp}{e{^}}{%
  \mathop{\mathpalette\extp@{#1}}\nolimits
}
\NewDocumentCommand{\extp@}{mm}{%
  \bigwedge\nolimits\IfValueT{#2}{^{\extp@@{#1}#2}}%
  \IfValueT{#1}{\kern-2\scriptspace\nonscript\kern2\scriptspace}%
}
\newcommand{\extp@@}[1]{%
  \mkern
    \ifx#1\displaystyle-1.8\else
    \ifx#1\textstyle-1\else
    \ifx#1\scriptstyle-1\else
    -0.5\fi\fi\fi
  \thinmuskip
}
\makeatletter

\newtheorem{thm}{Theorem}[section]
\newtheorem{dfn}[thm]{Definition}
\newtheorem{lem}[thm]{Lemma}
\newtheorem{prop}[thm]{Proposition}
\newtheorem{cor}[thm]{Corollary}
\newtheorem{rem}[thm]{Remark}
\newtheorem{exa}[thm]{Example}

\newtheorem{conj}{Conjecture}

\begin{document}

\begin{center}
{{\LARGE \bf 
Cohomology ring of unitary $N=(2,2)$ full vertex algebra and mirror symmetry}
} \par \bigskip

\renewcommand*{\thefootnote}{\fnsymbol{footnote}}
{\normalsize
Yuto Moriwaki \footnote{email: \texttt{moriwaki.yuto (at) gmail.com}}
}
\par \bigskip
{\footnotesize 
RIKEN Center for Interdisciplinary Theoretical and
Mathematical Sciences (iTHEMS), RIKEN,\\ Wako 351-0198, Japan}

\par \bigskip
\end{center}

\noindent

\vspace{5mm}

\begin{center}
\textbf{\large Abstract}
\end{center}
We formulate two-dimensional $N=(2,2)$ supersymmetric conformal field theories in terms of unitary full vertex operator superalgebras and develop their cohomology theory. Cohomology rings, Hodge numbers, and the Witten index of a unitary $N=(2,2)$ full VOA are introduced.
Using generalized full vertex operator superalgebras, spectral flow is constructed algebraically. Its periodicities are proved to be equivalent to the existence of top-degree cohomology classes, namely volume forms and holomorphic volume forms, and these characterizations yield Poincar\'e duality, T-duality, and Frobenius algebra structures on the cohomology rings, and thus two-dimensional topological field theories.
A mirror construction for full VOAs and its relation to Hodge-theoretic mirror symmetry are also discussed. Finally, examples arising from abelian varieties, a special K3 surface, and a Landau-Ginzburg model are examined.

\vspace{10mm}

\tableofcontents

\begin{center}
\textbf{\large Introduction}
\end{center}
\vspace{5mm}

A compact Calabi-Yau manifold $X$ is expected to determine a two-dimensional
$N=(2,2)$ supersymmetric conformal field theory (SCFT), namely the supersymmetric
sigma model with target $X$ \cite{Polc,Polc2}.
%Its geometric invariants, such as Hodge
%numbers, quantum cohomology, and the elliptic genus, are in turn expected to be
%recovered as classical limits of the sigma model $F_X$ \cite{Witten87,Witten,Hori}. It is
%therefore natural to study SCFTs themselves as mathematical objects, rather than only
%the geometric invariants extracted from them.
Various invariants of $X$, such as Hodge numbers, quantum cohomology, and the elliptic genus, should appear as classical limits of the sigma model $F_X$
\cite{Witten87,Witten,Hori}. It is therefore natural to study SCFTs themselves as mathematical objects, rather than only the geometric invariants
extracted from them.

A basic difficulty is that sigma models are usually constructed by path-integral
methods and are not yet available in a mathematically rigorous form in general.
By contrast, the algebraic structure of a two-dimensional conformal field theory
can be formulated rigorously by a full vertex operator algebra \cite{HK,M1,M11}; see also \cite{FRS}. The aim of this
paper is to develop, in the setting of full VOAs, a cohomology theory for
$N=(2,2)$ SCFTs and to establish several structures predicted in physics,
including Poincar\'e-type duality, T-duality, holomorphic twists, the Witten
index, and the construction of two-dimensional topological field theories.

More precisely, we formulate an $N=(2,2)$ SCFT by a unitary $N=(2,2)$
full vertex operator superalgebra $F$. Using the notion of \emph{topological twist} in
the sense of Elliott--Safronov \cite{ES}, one obtains two differentials
\begin{align*}
d_A=\tau^+_{(0)}+\bar\tau^-_{(0)}, \qquad
d_B=\tau^+_{(0)}+\bar\tau^+_{(0)}\quad \in \End\, F
\end{align*}
corresponding to the $A$- and $B$-twists in physics. Their cohomologies
$H(F,d_A)$ and $H(F,d_B)$ are supercommutative algebras. 
A basic result of this paper is that, under the unitarity assumption, these
cohomology rings carry a natural $\mathbb{R}^2$-grading. This is the full-VOA
analogue of the Hodge decomposition in K\"ahler geometry, and it is intrinsically
a feature of the full theory rather than of its chiral part, since the two gradings
arise from the holomorphic and anti-holomorphic structures. Thus the full VOA
framework is essential already at the level of cohomology.

The cohomological grading is not necessarily integral. For Landau-Ginzburg
models, fractional gradings appear naturally \cite{VW}. In the simplest $A_2$ example, the
$B$-twisted cohomology has nonzero components in bidegrees $(0,0)$ and
$(1/3,1/3)$, whereas the $A$-twisted cohomology is trivial except in degree
$(0,0)$. Thus, in general, the $A$- and $B$-twists may behave quite differently.
For SCFTs arising from sigma models, by contrast, the two twists are expected to
be closely related, and one of the main results of this paper is the T-duality
between them.

A central technical ingredient of the paper is the algebraic construction of
spectral flow. In physics, spectral flow is a fundamental operation in
$N=(2,2)$ theories, usually described in terms of changing fermionic boundary
conditions \cite{LVW}. However, the usual Lie-algebraic formulation does not by
itself explain how spectral flow is incorporated into the operator product
expansion of conformal field theory. In this paper, using the theory of
\emph{generalized full vertex operator algebras} \cite{Li20,M1}, we construct spectral flow directly
from a unitary $N=(2,2)$ full VOA $F$ as an $\mathbb{R}^2$-graded vector space
\[
\widetilde F=\bigoplus_{\lambda,\mu\in\mathbb R} F(\lambda,\mu),
\]
where $F(0,0)\cong F$, each $F(\lambda,\mu)$ is a twisted $F$-module, and
$\widetilde F$ itself carries the structure of a generalized full vertex operator
algebra. Thus spectral flow is realized not merely as an automorphism of the
$N=2$ superconformal Lie algebra, but as a canonical family of twisted sectors
within the full operator product expansion.

A key point is that the periodicity of spectral flow is governed by the existence
of top-degree cohomology classes. Let $(c,\bar c)$ denote the central charge of
$F$. For a unitary $N=(2,2)$ full VOA, the $d_T$-cohomology carries a natural
$\mathbb{R}^2$-grading
\begin{align*}
H(F,d_T)=\bigoplus_{\substack{c/3\ge p\ge 0\\\;\bar c/3\ge q\ge 0}} H^{p,q}(F,d_T),
\qquad T=A,B,
\end{align*}
coming from the eigenvalues of the left and right $U(1)$-charges. A non-zero
element in $H^{c/3,\bar c/3}(F,d_B)$ (resp.\ $H^{c/3,0}(F,d_B)$) is called a
\emph{volume form} (resp.\ a \emph{holomorphic volume form}). We prove that the diagonal
periodicity
\[
F(\lambda,\lambda)\cong F(\lambda+n,\lambda+n), \qquad n\in\mathbb Z,\ \lambda\in\mathbb R
\]
as twisted $F$-modules is equivalent to the existence of a volume form,
$H^{c/3,\bar c/3}(F,d_B)\neq 0$. Moreover, the  two-parameter periodicity
\[
F(\lambda,\mu)\cong F(\lambda+n,\mu+m), \qquad m,n\in\mathbb Z,\la,\mu\in\R
\]
is equivalent to the simultaneous existence of a volume form and a holomorphic
volume form, that is,
$H^{c/3,\bar c/3}(F,d_B)\neq 0$ and $
H^{c/3,0}(F,d_B)\neq 0$.
%These periodicity results yield Poincar\'e duality and T-duality for the cohomology rings.
These periodicity results yield Poincar\'e duality and T-duality for the cohomology
rings (see \cite{LVW} for the physical discussion).
More precisely, if $H^{c/3,\bar c/3}(F,d_T)\neq 0$, then
\begin{align}
H^{p,q}(F,d_T)\cong H^{c/3-p,\bar c/3-q}(F,d_T),
\qquad p,q\ge 0,\ T=A,B.
\end{align}
If in addition $H^{c/3,0}(F,d_T)\neq 0$, then the gradings become integral and
\begin{align}
H^{p,q}(F,d_B)\cong H^{c/3-p,q}(F,d_A),
\qquad p,q\in\mathbb{Z}_{\geq 0}.
\label{eq_intro_Tdual}
\end{align}
Thus spectral flow is a basic ingredient in the theory.
Another consequence is the construction of \emph{two-dimensional topological field
theories} from topological twists. When $H^{c/3,\bar c/3}(F,d_T)\neq 0$,
%when the top cohomology of a
%twist is nonzero, 
the corresponding cohomology ring becomes a Frobenius
superalgebra, and hence determines a $2$d TFT in the category of super vector spaces.

The \emph{Witten index} is also defined using spectral flow \cite{Witten87}. The
half-spectral-flow sector $F(1/2,1/2)$ is naturally identified with the \emph{Ramond
sector}, and we define
\[
\Ind_F(y,q)=\mathrm{tr}_{F(1/2,1/2)}(-1)^F y^{J(0)} q^{L(0)-c/24}.
\]
For the $B$-twist, the torus partition function of the resulting $2$d TFT
is identified with
\[
Z_F(\mathrm{torus})=\lim_{q\to 0,\; y\to 1}\Ind_F(y,q).
\]

We also study the holomorphic twist \cite{ES}. It is given by the functor
$F\mapsto H(F,\bar\tau^+_{(0)})$,
from unitary $N=(2,2)$ full VOAs to $N=2$ VOAs, and the topological twists are
recovered by a further twist:
\[
H(F,d_B)\cong H\bigl(H(F,\bar\tau^+_{(0)}),\tau^+_{(0)}\bigr), \qquad
H(F,d_A)\cong H\bigl(H(F,\bar\tau^-_{(0)}),\tau^+_{(0)}\bigr).
\]
Thus the holomorphic twist gives an intermediate chiral object
between the full theory and the topological twists. This is closely related to a
conjecture of Kapustin, which predicts that the large volume limit of the
holomorphic twist of the sigma model recovers the chiral de Rham complex of
Malikov-Schechtman-Vaintrob \cite{Kap05,MSV,BHS}.

As an application of the preceding duality results, we discuss \emph{mirror symmetry} at the
level of full VOAs. For a unitary $N=(2,2)$ full VOA $F$, one defines another unitary
$N=(2,2)$ full VOA $\widehat F$, called the \emph{mirror algebra}, by reversing the
anti-holomorphic $N=2$ structure (see Definition \ref{def_mirror}). This operation is
involutive and exchanges the two twists in the form
\begin{align}
H(F,d_A)\cong H(\widehat F,d_B),\qquad
H(F,d_B)\cong H(\widehat F,d_A).
\label{eq_intro_mirror}
\end{align}

If $F_X$ is the sigma model associated with a Calabi-Yau $d$-fold $X$, physics
predicts that the $A$-twist should recover the Hodge decomposition of $X$
in the form \cite{Witten}
\[
H^{p,q}(F_X,d_A)\cong H^q(X,\Omega_X^p).
\]
Suppose that $X$ and $Y$ are Calabi-Yau $d$-folds whose sigma models $F_X$ and
$F_Y$ are mirror as unitary $N=(2,2)$ full VOAs. Then, combined with the
T-duality \eqref{eq_intro_Tdual}, \eqref{eq_intro_mirror} implies
\begin{align*}
h^{p,q}(X)=h^{d-p,q}(Y).
\end{align*}
Thus the mirror relation for Hodge numbers appears as a consequence of mirror
symmetry at the level of full conformal field theories \cite{CDGP,Witten}.

% the Hodge-theoretic
%mirror relation  (Corollary \ref{cor_mirror_CY})

This point of view is closely related to the study of the moduli space of
$N=(2,2)$ SCFTs. One expects that geometric deformations of a Calabi-Yau
manifold $X$ correspond to deformations of the algebraic structure of the
associated full VOA $F_X$. This suggests that it is important to investigate
mirror symmetry already at the level of conformal field theory itself; see also
\cite{M1} for mathematical constructions of deformation families of CFTs.

There is also important earlier work by Borisov on mirror symmetry via vertex algebras
with $N=2$ structure, closely related to the chiral de Rham complex and its variants
\cite{Borisov2,Borisov,Borisov3}. In the toric and Berglund-H\"ubsch settings, Borisov constructed
vertex algebras whose chiral rings recover the expected $A$- and $B$-rings, and showed
that mirror symmetry is reflected in a natural mirror involution on the $N=2$ structure, together with the interchange of the underlying dual combinatorial data. These works
provide a concrete algebraic approach to mirror symmetry. See also
\cite{CHKLX,He,HeKa,NY} for related approaches to
supersymmetric chiral conformal field theory.

By contrast, the present paper works directly with unitary full vertex operator
algebras. This makes it possible to formulate spectral flow, topological twists,
T-duality, the Witten index, and the mirror involution at the level of the full
conformal field theory itself.

We apply these general constructions to several examples.
For abelian varieties and for a special K3 surface, we
construct unitary $N=(2,2)$ full VOAs whose topological twists reproduce the
expected Hodge numbers. Their cohomology rings agree with the corresponding
quantum cohomology rings, and their Witten indices recover the expected elliptic
genera. For related work on abelian varieties, see \cite{KO}; for the physics of K3
sigma models, see \cite{EOTY,GHV}.
%We study several examples. For abelian varieties and for a special K3 surface, we
%construct unitary $N=(2,2)$ full VOAs whose topological twists reproduce the
%expected Hodge numbers. Their cohomology rings agree with the corresponding
%quantum cohomology rings, and their Witten indices recover the expected elliptic
%genera. 
We also discuss the simplest Landau--Ginzburg model \cite{VW}, which illustrates
the appearance of fractional gradings and lies beyond the Calabi-Yau sigma-model
setting.
% These examples illustrate the general results of the paper.

\vspace{3mm}

The paper is organized as follows. Section~1 introduces the basic setting of the
paper. We review $N=2$ superconformal algebras, $N=2$ vertex operator
superalgebras, and $N=(2,2)$ full VOAs, together with their twisted modules and
unitarity. We also recall topological twists and two-dimensional topological
field theories, and define the mirror algebra. Section~2 develops the cohomology
theory of unitary $N=(2,2)$ full VOAs. We define the $A$- and $B$-twisted
cohomology rings, study the consequences of unitarity, and formulate the
resulting Poincar\'e duality, T-duality, and Frobenius algebra structure.
Section~3 is devoted to spectral flow. Using generalized full VOAs, we construct
spectral flow and analyze its periodicity. From this we derive the duality
results stated earlier, and we also study the holomorphic twist and the Witten
index. Section~4 concerns geometric aspects and examples. After reviewing the
relevant geometric invariants and the Hodge-theoretic mirror relation, we study
examples coming from abelian varieties, a special K3 surface, and a
Landau--Ginzburg model.

\vspace{3mm}
\begin{center}
\textbf{\large Notations}
\end{center}

\vspace{3mm}
We will use the following notations:
\begin{itemize}
\item[$\g_{\la}^{N=2}$:] the superconformal algebra with twist $\la \in \R$, \S \ref{sec_Lie}
\item[$G_r^{\pm}$:] the chiral superconformal current $r \in \ft+\Z$, \S \ref{sec_Lie} and \ref{sec_def_VOA}
\item[$J_n$:] the chiral R-symmetry $n \in \Z$, \S \ref{sec_Lie} and \ref{sec_def_VOA}
\item[$\bG_r^{\pm}$:] the anti-chiral superconformal current $r \in \ft+\Z$, \S \ref{sec_def_fullVOA}
\item[$\bJ_n$:] the anti-chiral R-symmetry $n \in \Z$, \S \ref{sec_def_fullVOA}
\item[$F$:] $N=(2,2)$ full vertex operator superalgebra, \S \ref{sec_def_fullVOA}
\item[$(c,\bc)$:] the central charges of the full vertex operator superalgebra, \S \ref{sec_def_fullVOA}
\item[$\phi$:] the unitary structure on $F$, \S \ref{sec_unitary}
\item[$\langle -,-\rangle$:] the positive-definite  invariant sesquilinear form on $F$, \S \ref{sec_unitary}
\item[$d_A,d_B$:] the differential operators, A-twist and B-twist, \S \ref{sec_class_twist}
%\item[$d_l,d_r$:] holomorphic and anti-holomorphic twists, \S
\item[$H(F,d)$:] the cohomology ring of $F$, \S \ref{sec_twist_def}
\item[$C(F)_{cc,ca,ac,aa}$:] the space of c-c (c-a,a-c,a-a) primary vectors of $F$, \S \ref{sec_inequality}
\item[$\ep \in H^{J,\bJ}(F,d)$:] a volume form of $F$, \S \ref{sec_twist_star}
\item[$\eta \in H^{J,0}(F,d)$:] a holomorphic volume form of $F$, \S \ref{sec_twist_serre}
\item[$F(\la,\mu)$:] the spectral flow of $N=(2,2)$ full vertex operator superalgebra $F$, \S \ref{sec_flow}
\item[$\Om_F$:] the generalized full vertex algebra associated with $F$, \S \ref{sec_str_h}
\item[$\sVect$:] the category of super vector spaces, \S \ref{sec_tft}
\item[$L_\Om(-1),\Ld_\Om(-1)$:] translations of generalized full vertex algebra $\Om$, \S \ref{sec_str_h}
\item[$\ind_F(y,q)$:] the elliptic genus of $F$, \S \ref{sec_genus}
\item[$h_{p,q}^\bullet(F)$:] the Hodge numbers of $F$, \S \ref{sec_twist_star}
\end{itemize}

\vspace{5mm}

\section{Preliminary}\label{sec_Preliminary}
%Section \ref{sec_Lie}では $N=2$ superconformal algebra と呼ばれる 無限次元リー代数の定義や Lie algebra 的な spectral flow を振り返る。Section \ref{}では $N=2$ vertex operator algebra とその自己同型群を復習する。
%Section \ref{} では $N=(2,2)$ full vertex operator superalgebra とそのtwisted加群を定義し、twisted加群には twist された $N=(2,2)$ superconformal algebra が作用することを見る。Section \ref{sec_unitary}では、$N=(2,2)$ full vertex operator superalgebraのユニタリ性とミラー代数を定義する。またSection \ref{}では 2d topological field theory の定義を振り返る。

Section \ref{sec_Lie} reviews the definition of an infinite-dimensional Lie algebra called an $N=2$ superconformal algebra and its automorphisms. Section \ref{sec_def_VOA} recalls the definition of an $N=2$ vertex operator algebra and its automorphism group.
Section \ref{sec_def_fullVOA} defines an $N=(2,2)$ full vertex operator superalgebra and its twisted modules. In Section \ref{sec_unitary}, we define  unitarity and the mirror algebra of an $N=(2,2)$ full vertex operator superalgebra. In Section \ref{sec_class_twist}, we recall the definition of holomorphic/topological twists from \cite{ES} and classify them up to the inner automorphisms of a unitary $N=(2,2)$ full vertex operator algebra and in Section \ref{sec_tft} we review the definition of 2d topological field theory.

\subsection{$N=2$ superconformal algebra}\label{sec_Lie}
An $N=2$ superconformal algebra is an infinite-dimensional Lie algebra which is an extension of the chiral part of the super-Poincar\'e Lie algebra. In this section, we review the definition of $N=2$ superconformal algebra and its automorphisms. All the content of this section is standard (see for example \cite{LVW}).

%$N=2$ superconformal algebra は, superPoncare Lie algebra のchiral part の拡大となる infinite dimensional Lie algebra である。この章では$N=2$ superconformal algebraの定義と、いくつかの自己同型を振り返る。この章の内容は全て標準的である

\begin{dfn}\label{def_susy_Lie}
Let $\la \in \R$.
The ($\la$-twisted) $N = 2$ superconformal algebra $\g_\la^{N=2}$ is the Lie superalgebra with basis of even elements $c, L_n, J_n$, for $n$ an integer, and odd elements $G_r^+,G_s^-$ for $r \in \ft+ \la+\Z$ and $s\in \ft - \la +\Z$,
defined by the following relations:
\begin{align}
\begin{split}
[J_n,J_m]&= \frac{c}{3} n\delta_{n+m,0},\quad [J_n,G_r^+]=\pm G_{r+n}^+,\quad [L_n,G_r^+]= (\frac{n}{2}-r)G_{r+n}^+\\
[L_n,J_m]&= -mJ_{n+m},\quad [J_n,G_s^-]=\pm G_{s+n},\quad [L_n,G_s^-]= (\frac{n}{2}-s)G_{s+n}^-\\
[G_r^+,G_{r'}^+]_+&=[G_s^-,G_{s'}^-]_+=0,\quad [L_n,L_m] = (n-m)L_{n+m} + c\frac{n^3-n}{12}\delta_{n+m,0}\\
[G_r^+,G_s^-]_+&=L(r+s)+\frac{1}{2}(r-s)J(r+s)+\frac{c}{6}(r^2-\frac{1}{4})\delta_{r+s,0}
\end{split}
\label{eq_susy}
\end{align}
for any $n,m\in \Z$, $r,r' \in \ft+\la+\Z$ and $s,s' \in \ft-\la+\Z$.
\end{dfn}
Note that the algebras $\g_\la^{N=2}$ and $\g_\mu^{N=2}$ are the same if $\la - \mu \in \Z$.
$\g_\la^{N=2}$ with $\la \in \Z$ is called {\it the $N = 2$ Neveu-Schwarz algebra}, and $\g_\la^{N=2}$
with $\la \in \ft +\Z$ is called {\it the $N=2$ Ramond algebra}.

\begin{dfn}\label{def_spectral_Lie}
For $\la,\eta\in\R$,
{\it spectral flow} is a Lie superalgebra isomorphism \cite{LVW}
\begin{align*}
U_\eta: \g_\la^{N=2} \rightarrow \g_{\la+\eta}^{N=2}
\end{align*}
given by
\begin{align*}
U_{\eta} (c) =c,\quad U_\eta (L_n) &= L_n + \eta J_n + \frac{c\eta^2}{6} \delta_{n,0},\quad U_\eta (J_n) = J_n+ \frac{c\eta}{3} \delta_{n,0}\\
U_\eta (G_r^+) &= G_{r+\eta}^+, \quad U_\eta (G_s^-) = G_{s-\eta}^-.
\end{align*}
\end{dfn}
%\begin{itemize}
%\item
%Lerche-Vafa-Warner, chiral rings in N=2 superconformal theories.
%\end{itemize}

\begin{rem}
The shift of constants in spectral flow can be seen, for example, from:
\begin{align*}
&[U_\eta(G_r^+),U_\eta(G_{-r}^-)]_+ = [G_{r+\eta}^+,G_{-r-\eta}^-] = L(0)+(r+\eta) J(0)+\frac{c}{6} ((r+\eta)^2-\frac{1}{4})\\
&=(L(0)+\eta J(0)+\frac{c \eta^2}{6})+ r(J(0)+ \frac{c \eta}{3}) +\frac{c}{6} (r^2-\frac{1}{4})= U_\eta([G_r^+,G_{-r}^-]_+).
\end{align*}
\end{rem}

\begin{dfn}\label{def_anti-involution_Lie}
The linear map $\al: \g_0^{N=2} \rightarrow \g_0^{N=2}$ given below is an automorphism of the Lie superalgebra $\g_\la^{N=2}$:
\begin{align*}
\al (c) =c,\quad \al(L_n) &= L_{n},\quad \al (J_n) = -J_{n},\quad
\al (G_r^+) = G_{r}^-, \quad \al (G_s^-) = G_{s}^+
\end{align*}
for any $n\in \Z$ and $r,s \in \ft+\Z$.
\end{dfn}

\subsection{$N=2$ vertex operator superalgebra}\label{sec_def_VOA}
Vertex operators can be defined as generating functions of $N=2$ superconformal algebra.
In this section we review the definition of a vertex operator superalgebra \cite{Kac}
and an $N=2$ vertex operator superalgebra \cite{HeKa}.
A vertex operator superalgebra is an algebra describing the holomorphic part of a CFT.
A full vertex operator superalgebra, which is the algebra of the whole CFT, which is real analytic, will be reviewed in the next section.

%$N=2$ superconformal algebra の母関数として、共形場理論の量子場(頂点作用素)が定義できる。
%この章では超頂点作用素代数の定義を振り返る \cite{Kac}。
%超頂点作用素代数は共形場理論の正則部分を記述する代数であり、実解析的な部分を含む一般化は次の章で扱う。

Let $V$ be a $\ft\Z$-graded vector space $V=\bigoplus_{n\in \ft \Z} V_n$ and regard it as a $\Z_2$-graded vector space by
\begin{align*}
V^{\bare}=\bigoplus_{n\in \Z}V_n,
% \text{ and }
 \quad\quad V^\baro =\bigoplus_{n\in \ft+\Z}V_n
\end{align*}
and $V^\vee = \bigoplus_{n\in\ft\Z}V_n^*$, where $V_n^*$ is the dual vector space.
Denote the $\Z_2$-grading of a vector $a \in V= V^{\bar{i}}$ by $|a| = \bar{i} \in \Z_2$,
which is called a {\it fermion number}.

A $\Z_2$-graded vertex superalgebra is defined as follows (see \cite{Li2,Kac} for more detail):
\begin{dfn}\label{def_vertex_superalgebra}
A {\it $\ft \Z$-graded vertex superalgebra} is a $\ft \Z$-graded $\C$-vector space $V=\bigoplus_{n\in \ft \Z} V_n$ equipped with a linear map
$$Y(-,z):V \rightarrow \End (V)[[z^\pm]],\; a\mapsto Y(a,z)=\sum_{n \in \Z}a(n)z^{-n-1}$$
and an element $\va \in V_0$ satisfying the following conditions:
\begin{enumerate}
\item[V1)]
For any $i,j \in \Z_2$ and $a \in V^{\bar{i}}$ and $b\in V^{\bar{j}}$, $Y(a,z)b \in V^{\overline{i+j}}((z))$,
that is, $Y(-,z)$ is an even linear map;
\item[V2)]
For any $a \in V$, $Y(a,z)\va \in V[[z,\z]]$ and $\lim_{z \to 0}Y(a,z)\va = a(-1)\va=a$;
\item[V3)]
$Y(\va,z)=\mathrm{id}_V \in \End V$;
\item[V4)]
%convergence
For any $a,b,c \in V$ and $u \in V^\vee$, there exists $\mu(z_1,z_2) \in \C[z_1^\pm,z_2^\pm,(z_1-z_2)^\pm]$ such that
\begin{align*}
u(Y(a,z_1)Y(b,z_2)c) &= \mu(z_1,z_2)|_{|z_1|>|z_2|}, \\
u(Y(Y(a,z_0)b,z_2)c) &= \mu(z_1,z_2)|_{|z_2|>|z_1-z_2|},\\
(-1)^{|a||b|}u(Y(b,z_2)Y(a,z_1)c)&=\mu(z_1,z_2)|_{|z_2|>|z_1|},
\end{align*}
where $z_0=z_1-z_2$;
\item[V5)]
For any $a \in V$, $z\frac{d}{dz}Y(a,z) = [L(0),Y(a,z)]-Y(L(0)a,z)$.
%$V_n(r)V_m \subset V_{n+m-r-1}$ for any $n,m,r \in \Z$.
\end{enumerate}
\end{dfn}

\begin{dfn}\label{def_conformal}
A {\it vertex operator superalgebra} is a $\ft \Z$-graded vertex superalgebra $V$ equipped with
a vector $\om \in V_2$ such that
\begin{enumerate}
\item
There exist a scalar $c \in \C$ such that
$\om(3)\om=\frac{c}{2} \va$ and $\om(k)\om=0$ for any $k=2$ or $k\in \Z_{\geq 4}$.
\item
$\om(0)a=a(-2)\va$ for any $a\in V$;
\item
$\om(1)|_{V_n}=n \id_{V_n}$ for any $n\in \ft\Z$;
\item
There exists $N \in \Z$ such that $V_n=0$ for any $n <N$;
\item
$\dim V_n < \infty$ for any $n \in \ft\Z$.
\end{enumerate}
\end{dfn}
Let $L(n)$ be the coefficients of $\om$, that is,
\begin{align*}
Y(\om,z) = \sum_{n\in \Z}L(n) z^{-n-2}.
\end{align*}
By (V4), for any $a_1,a_2 \in V$ and $p,q\in \Z$,
\begin{align}
[a_1(p),a_2(q)]_{\pm} = \sum_{k \geq 0}\binom{p}{k} (a_1(k)a_2)(p+q-k)
\label{eq_commutator}
\end{align}
holds, which implies that  $\{L(n)\}_{n\in \Z}$ satisfies the Virasoro commutator relation:
\begin{align*}
[L(n),L(m)]= (n-m)L(n+m)+\frac{n^3-n}{12}c \delta_{n+m,0}.
\end{align*}

\begin{dfn}\label{def_N2}
An {\it $N=2$ vertex operator superalgebra} is a vertex operator superalgebra $(V,\om)$ equipped with
vectors $\tau^+,\tau^- \in V_{\frac{3}{2}}$ and $J \in V_1$ satisfying the following conditions:
%\begin{enumerate}
%\item
\begin{align}
\begin{split}
L(n)J&= L(n) \tau^\pm=0\\
J(n)\tau^\pm &=0,\quad\quad J(0)\tau^\pm=\pm \tau^\pm\\
J(1)J&= \frac{c}{3}\va,\quad\quad J(k)J=0\\
\tau^+(m)\tau^+ &= \tau^-(m)\tau^-=0\\
\tau^+(0)\tau^- &=\om+\ft L(-1)J,\\
\tau^+(1)\tau^- &=J,\quad\quad \tau^+(2)\tau^-=\frac{c}{3}.
\end{split}
\label{eq_susy_ope}
\end{align}
%\item
%$J(0)$ is semisimple on $V$ with integer eigenvalues and $\exp(\pi i J(0)) |_{V^{\bar{i}}} = (-1)^i$ for any $i \in \Z_2$.
%\end{enumerate}
for any $n \geq 1$, $m \geq 0$ and $k \geq 2$ or $k=0$.
\end{dfn}

Let $(V,\om,\tau^\pm,J)$ be an $N=2$ vertex operator superalgebra. Set
\begin{align*}
Y(\tau^\pm,z)=\sum_{r \in \ft+\Z} G_r^\pm z^{-\frac{3}{2}-r},\quad\quad Y(J,z) = \sum_{n\in \Z}J_n z^{-n-1},
\end{align*}
that is, $G_r^\pm = \tau^\pm(r+\ft)$.
By \eqref{eq_commutator}, \eqref{eq_susy_ope} is equivalent to the following relations:
\begin{align}
\begin{split}
[J_n,J_m]&= \frac{c}{3} n\delta_{n+m,0},\quad [J_n,G_r^\pm]=\pm G_{r+n},\quad [L_n,G_r^\pm]= (\frac{n}{2}-r)G_{r+n}^\pm\\
[L_n,J_m]&= -mJ_{n+m},\quad [G_r^+,G_s^+]_+=[G_r^-,G_s^-]_+=0,\\
[G_r^+,G_s^-]_+&=L(r+s)+\frac{1}{2}(r-s)J(r+s)+\frac{c}{6}(r^2-\frac{1}{4})\delta_{r+s,0}
\end{split}
\label{eq_susy}
\end{align}
for any $n,m\in\Z$ and $r,s\in \ft +\Z$. Hence, $\g_0^{N=2}$ naturally acts on $V$.

Put another way, let $V_c^{N=2}$ be the universal VOA generated from $J,\om,\tau^\pm$ with the relation \eqref{eq_susy_ope}. This algebra is determined only by the central charge $c \in \C$. An $N=2$ vertex operator superalgebra is a VOA that has a quotient of $V_c^{N=2}$ as a subalgebra sharing the same conformal vector.
Below, we review the automorphism groups of $V_c^{N=2}$.

%別の言い方をすると、\eqref{eq_susy_ope}の関係式で$J,\om,\tau^\pm$から生成される universal なVOAを$V_c^{N=2}$とおく。この代数は中心電化$c \in \C$のみで決まる。$N=2$ vertex operator superalgebra とは$V_c^{N=2}$を共形ベクトルを共有する部分代数として持つようなVOAのことである。

%以下, $V_c^{N=2}$ の自己同型群を振り返る。
\begin{dfn}\label{def_VOA_alpha}
Let $\al$ be the vertex operator algebra automorphisms of $V_c^{N=2}$ defined by
\begin{align*}
\al(\om)=\om,\quad &\al(\tau^\pm)=\tau^\mp,\quad \al(J)= -J
\end{align*}
and $\exp (i \la J(0))$ be the inner automorphism of $V_c^{N=2}$ defined by the exponential of the zero mode $J(0)$, i.e.,
\begin{align*}
\exp (i \la J(0))(\om)=\om,\quad &\exp (i \la J(0))(\tau^\pm) = \exp( i \la) \tau^\pm, \quad \exp (i \la J(0))(J)=J
\end{align*}
for any $\la \in \C^\times$.
\end{dfn}
Let $\Aut (V_c^{N=2})$ be the group of vertex operator algebra automorphisms of $V_c^{N=2}$.
The following proposition is obvious:
\begin{prop}\label{prop_auto_N2}
The automorphism group $\Aut(V_c^{N=2})$ is generated by $\al$ and $\{\exp (i \la J(0))\}_{\la \in \C^\times}$ and isomorphic to $\C^\times \rtimes \Z_2$.
\end{prop}

\subsection{$N=(2,2)$ full vertex operator superalgebra and twisted modules}\label{sec_def_fullVOA}

Non-super conformal field theories can be described using the notion of full vertex operator algebras introduced in \cite{M1}.
In this section, we define the super analogue of a full vertex operator algebra and
review its basic properties. Since the proofs are the same as in the non-super case,
we refer to \cite{M1}. We then formulate two-dimensional $N=(2,2)$ SCFTs in
terms of full vertex operator superalgebras (full VOAs).

%In this section, we define a super analogy, a full vertex operator superalgebra, and describe its properties. Since all these results can be proved in the same way as in the non-super case, we refer to \cite{M1} and do not give the proofs here.
%Then, we formulate $N=(2,2)$ 2d SCFT by using a full vertex operator superalgebra (full VOA).
%In order to study the topological twists, it is important to consider the unitary property, which will be treated in Section \ref{sec_}.

In this section we also introduce the notion of $\exp(i \la J(0)+i\mu \bJ(0))$-twisted modules for $N=(2,2)$ full VOAs, where the twisted superconformal algebra 
\begin{align*}
\g_\la^{N=2} \oplus \g_{\mu}^{N=2}\quad\quad\quad (\la,\mu \in \C)
\end{align*}
acts naturally. 
Barron showed that on the twisted module of $V_c^{N=2}$, the twisted superconformal algebra $\g_\la^{N=2}$ acts \cite{Barron}. 
Hence, this is a generalization of Barron's result.

%
%この章では超共形場理論を記述する代数であるfull 超頂点作用素代数を用いて$N=(2,2)$ 2d SCFTを定義する。
%superでない共形場理論は \cite{}で導入した full vertex operator algebra を用いて記述することができる。
%この章ではその super 類似を定義しその性質を述べる。こうした結果は全て non-super の場合と同様に証明できるため、\cite{}を参照し証明はここでは与えない。
%また topological twist を考えるためには, unitary 性を定義することが重要であるが、それは Section \ref{sec_}で扱う。
%Section \ref{sec_def_VOA}で見たように $N=2$ VOA には リー代数$\g_{\la=0}^{N=2}$が自然に作用する。
%この章ではさらに $N=(2,2)$ full VOA の$\exp(i \la J(0)+i\mu \bJ(0))$-twisted 加群の概念を導入し、そこに $\g_\la^{N=2} \oplus \g_{\mu}^{N=2}$ ($\la,\mu \in \C$)が自然に作用することを示す。
%これは chiral な場合に \cite{Barron} によって得られた結果の拡張である。

We will first define the notion of a full vertex operator superalgebra and review its properties (see for more detail \cite{M1}).
Let $F=\bigoplus_{h,\h\in \R}F_{h,\h}$ be an $\R^2$-graded vector space
and $L(0),\Ld(0):F\rightarrow F$ linear maps defined by
$L(0)|_{F_{h,\h}}=h \id_{F_{h,\h}}$ and 
$\Ld(0)|_{F_{h,\h}}=\h \id_{F_{h,\h}}$
for any $h,\h\in \R$.
We assume that:
\begin{enumerate}
\item[FO0)]
$F_{0,0}=\C\vac$;
\item[FO1)]
$F_{h,\h}=0$ unless $h-\h \in \ft\Z$;
\item[FO2)]
$\sum_{h+\h<H} \dim F_{h,\h} < \infty$ for any $H \in \R$;
\item[FO3)]
There exists $N\in \R$ such that $F_{h,\h}=0$ unless $h \geq N$ and $\h \geq N$.
\end{enumerate}
Set 
\begin{align*}
F^\vee =\bigoplus_{h,\h\in\R} F_{h,\h}^*,
\end{align*}
where $F_{h,\h}^*$ is the dual vector space, and regard $F$ as a $\Z_2$-graded vector space by
\begin{align*}
F^{\bare} = \bigoplus_{\substack{h,\h\in\R\\ h-\h \in \Z}}F_{h,\h}\quad\quad\text{and}\quad\quad
F^{\baro} = \bigoplus_{\substack{h,\h\in\R\\ h-\h \in \ft+ \Z}}F_{h,\h}
\end{align*}
and let $(-1)^F: F \rightarrow F$ be the linear map given by
\begin{align*}
(-1)^F |_{F^{\bar{i}}}= (-1)^{\bar{i}}
\end{align*}
for $\bar{i} \in \Z_2$. We will denote by $\uz$ the pair of formal variables $z$ and $\z$.

A  \textbf{full vertex operator} on $F$ is an even linear map
\begin{align*}
Y(\bullet, z,\z):F \rightarrow \mathrm{End}(F)[[z,\z,|z|^\R]],\; a\mapsto Y(a,\uz)=\sum_{r,s \in \R}a(r,s)z^{-r-1}\z^{-s-1}
\end{align*}
such that:
\begin{align}
\begin{split}
[L(0),Y(a,\uz)]&= \frac{d}{dz}Y(a,\uz) + Y(L(0)a,\uz),\\
[\Ld(0),Y(a,\uz)]&= \frac{d}{d\z}Y(a,\uz) + Y(\Ld(0)a,\uz).
\label{eq_L0_cov}
\end{split}
\end{align}
Then, by (FO1), (FO2) and (FO3), $Y(a,\uz)b \in F((z,\z,|z|^\R))$.
%It is possible to substitute the vector $a \in F$ by
%a formal series $\sum_{r', s'} a'_{r',s'} (z')^{r'} (z')^{s'}$ with
%\textit{another set of formal variables} $z', \z', |z'|^\R$.
%Then the result is a formal series in $z,\z,|z|^\R, z',\z',|z'|^\R$.

By \eqref{eq_L0_cov}, for $u \in F_{h_0,\h_0}^\vee$ and $a_i \in F_{h_i,\h_i}$ we have
\begin{align}
u(Y(a_1,\uz_1)Y(a_2,\uz_2)a_3) \in z_2^{h_0-h_1-h_2-h_3}\z_2^{\h_0-\h_1-\h_2-\h_3}\C\Bigl(\Bigl(\frac{z_2}{z_1},\frac{\z_2}{\z_1},\Bigl|\frac{z_2}{z_1}\Bigr|^\R\Bigr)\Bigr),\label{eq_conv_rad2}\\
u(Y(Y(a_1,\uz_0)a_2,\uz_2)a_3) \in z_2^{h_0-h_1-h_2-h_3}\z_2^{\h_0-\h_1-\h_2-\h_3}\C\Bigl(\Bigl(\frac{z_0}{z_2},\frac{\z_0}{\z_2},\Bigl|\frac{z_0}{z_2}\Bigr|^\R\Bigr)\Bigr),
\label{eq_conv_rad}
\end{align}
where the left-hand side of \eqref{eq_conv_rad} is a formal series in $z_0,\z_0,|z_0|^\R, z_2,\z_2,|z_2|^\R$
but contains only terms of the form $z_0^n z_2^{-n}, \z_0^m \z_2^{-m}, z^s\z_0^s \z_2^{-s}\z_2^{-s}$ up to the factor
in front  (for more detail, see \cite[Proposition 1.5 and Lemma 1.6]{M1}).

\begin{dfn}\label{def_full_va}
A \textbf{full vertex algebra} is an $\R^2$-graded $\C$-vector space
$F=\bigoplus_{h,\h \in \R^2} F_{h,\h}$ satisfying (FO1), (FO2) and (FO3)  equipped with a
full vertex operator $Y(\bullet,\uz):F \rightarrow \mathrm{End}(F)[[z^\pm,\z^\pm,|z|^\R]]$
and an element $\vac \in F_{0,0}$ satisfying the following conditions:
\begin{enumerate}
\item[FV1)]
For any $a \in F$, $Y(a,\uz)\vac \in F[[z,\z]]$ and $\displaystyle{\lim_{\uz \to 0}Y(a,\uz)\vac = a(-1,-1)\vac=a}$.
\item[FV2)]
$Y(\vac,\uz)=\mathrm{id}_F \in \End F$;
\item[FV3)]
For any $a_i \in F_{h_i,\h_i}$ and $u \in F_{h_0,\h_0}^*$, \eqref{eq_conv_rad2} and \eqref{eq_conv_rad} are absolutely convergent in $\{|z_1|>|z_2|\} \subset Y_2(\C)$ and $\{|z_0|<|z_2|\}\subset Y_2(\C)$, respectively, and there exists a real analytic function $\mu: Y_2(\C)\rightarrow \C$ such that:
\begin{align*}
u(Y(a,\uz_1)Y(b,\uz_2)c) &= \mu(z_1,z_2)|_{|z_1|>|z_2|}, \\
u(Y(Y(a,\uz_0)b,\uz_2)c) &= \mu(z_0+z_2,z_2)|_{|z_2|>|z_0|},\\
(-1)^{|a||b|}u(Y(b,\uz_2)Y(a,\uz_1)c)&=\mu(z_1,z_2)|_{|z_2|>|z_1|}
\end{align*}
where $Y_2(\C)=\{(z_1,z_2)\in \C^2\mid z_1\neq z_2,z_1\neq 0,z_2\neq 0\}$.
\end{enumerate}
\end{dfn}

\begin{dfn}
\label{def_Koszul_vertex}
Let $(F,Y_F,\va_F),(G,Y_G,\va_G)$ be full vertex superalgebras. Define a full vertex operator $Y_{F\otimes G}(-,\uz)$ on $F \otimes G$ by
\begin{align}
Y_{F\otimes G}(a\otimes b,\uz)c \otimes d
= (-1)^{|b||c|} Y_F(a,z)c \otimes Y_G(b,\z)d \label{eq_Koszul_vertex}
\end{align}
for any $a,c \in F$ and $b,d\in G$. Then, $(F\otimes G,Y_{F \otimes G}(-,\uz),\va_F \otimes \va_G)$ is a full vertex superalgebra.
It is important to note that without the Koszul sign rule in \eqref{eq_Koszul_vertex}, there is no superalgebra structure.
(The need for such a sign is well-known in the chiral case (e.g. \cite{Kac}).
%(When $F$ and $G$ are vertex superalgebras, 
%Such a sign is well-known in the chiral case 
%coincides with \cite{Kac}.)
%ここで \eqref{eq_Koszul_vertex}における Koszul sign rule なしには superalgebra の構造を持たないことが重要である。
%($F,G$が vertex superalgebra の場合に、こうした符号はたとえば\cite{Kac}におけるテンソル積の定義と一致する。)
\end{dfn}

Let $F$ be a full vertex superalgebra
and $D$ and $\D$ denote the endomorphism of $F$
defined by $Da=a(-2,-1)\vac$ and $\D a=a(-1,-2)\vac$ for $a\in F$,
i.e., $$Y(a,z)\vac=a+Daz+\D a\z+\dots.$$

Then, we have (see \cite[Proposition 3.7, Lemma 3.11, Lemma 3.13]{M1}):
\begin{prop}
\label{prop_translation}
For $a \in F$, the following properties hold:
\begin{enumerate}
\item
$Y(Da,\uz)=\frac{d}{dz} Y(a,\uz)$ and $Y(\D a,\uz)=\frac{d}{d\z} Y(a,\uz)$;
\item
$D\vac=\D\vac=0$;
\item
$[D,\D]=0$;
\item
$Y(a,\uz)b=(-1)^{|a||b|}\exp(zD+\z\D)Y(b,-\uz)a$;
\item
$Y(\D a,\uz)=[\D,Y(a,\uz)]$ and $Y(Da,\uz)=[D,Y(a,\uz)]$.
\item
If $\D a=0$, then 
\begin{align*}
Y(a,\uz) = \sum_{n\in \Z} a(n,-1)z^{-n-1}
\end{align*}
and for any $n\in \Z$ and $b\in F$,
\begin{align*}
[a(n,-1),Y(b,\uz)]_\pm&= \sum_{j \geq 0} \binom{n}{j} Y(a(j,-1)b,\uz)z^{n-j},\\
Y(a(n,-1)b,\uz)&= 
\sum_{j \geq 0} \binom{n}{j}(-1)^j a(n-j,-1)z^{j}Y(b,\uz) \\
&\qquad -(-1)^{|a||b|}
Y(b,\uz)\sum_{j \geq 0} \binom{n}{j}(-1)^{j+n} a(j,-1)z^{n-j}.
\end{align*}
\item
If $\D a =0$ and $D b=0$, then $[Y(a,\uz),Y(b,\uz)]=0$.
\end{enumerate}
\end{prop}

Proposition \ref{prop_translation} implies $\ker \D$ is a vertex superalgebra and $\ker D$ is also a vertex superalgebra with the formal variable $\z$.

\begin{dfn}\label{def_fullVOA}
A \textbf{full vertex operator superalgebra} (full VOA) is a pair of a full vertex superalgebra and distinguished vectors
$\om \in F_{2,0}$ and $\omb \in F_{0,2}$ such that
\begin{enumerate}[{FVOA}1{)}]
\item
$\D \om=0$ and $D \omb=0$;
\item
There exist scalars $c, \bar{c} \in \C$ such that
$\om(3,-1)\om=\frac{c}{2} \vac$,
$\omb(-1,3)\omb=\frac{\bar{c}}{2} \vac$ and
$\om(k,-1)\om=\omb(-1,k)\omb=0$
for any $k=2$ or $k\in \Z_{\geq 4}$.
\item
$\om(0,-1)=D$ and $\omb(-1,0)=\D$;
\item
$\om(1,-1)|_{F_{h,\h}}=h \id_{F_{h,\h}}$ and
$\omb(-1,1)|_{F_{h,\h}}=\h \id_{F_{h,\h}}$ for any $h,\h \in \R$.
\end{enumerate}
\end{dfn}

We remark that 
\begin{align*}
L(n)= \om(n+1)\quad\text{ and }\quad \Ld(n)=\omb(n+1)
\end{align*}
satisfy the commutation relation of the Virasoro algebras by Proposition \ref{prop_translation}.
Then, we have \cite[Proposition 3.18 and Proposition 3.19]{M1}:
\begin{prop}\label{ker_hom}
Let $(F,\om,\omb)$ be a full vertex operator algebra.
Then, $(\ker \D,\om)$ and $(\ker D,\omb)$ are vertex operator superalgebras and the linear map
\begin{align*}
\ker \D \otimes \ker D \rightarrow F,\quad a\otimes b \mapsto a(-1,-1)b
\end{align*}
is a full vertex operator superalgebra homomorphism.
Here, the tensor product of full vertex superalgebras $\ker \D \otimes \ker D$ is given in Definition \ref{def_Koszul_vertex}.
%Moreover, $F$ is a $\ker \D \otimes \ker D$-module and the vertex operator $Y(\bullet,\uz)$ is an intertwining operator of $\ker \D \otimes \ker D$-modules.
\end{prop}

\begin{dfn}\label{def_automorphism}
Let $(F,\om,\omb)$ be a full vertex operator superalgebra. An automorphism of $F$ is a linear isomorphism $f:F \rightarrow F$ such that  $f(\om)=\om$, $f(\omb)=\omb$, $f(\va)=\va$ and $f(a(r,s)b)=f(a)(r,s)f(b)$ for any $a,b \in F$ and $r,s\in \R$.
\end{dfn}

%
%Let $V$ be a vertex operator algebra and $M$ a $V$-module.
%For any $n \in \bbZ_{>0}$, set 
%\begin{align*}
%C_n(M)= \{a(-n)m\mid m\in M \text{ and }a \in\bigoplus_{k\geq 1}V_k \}.
%\end{align*}
%A $V$-module $M$ is called \textbf{$C_n$-cofinite} if $M/C_n(M)$ is a finite-dimensional vector space.
%
%Since $(L(-1)a)(-n)=na(-n-1)$ for any $a \in V$ and $n\in \bbZ_{>0}$,
%$C_{n+1}(M) \subset C_{n}(M)$. Hence, if $M$ is $C_{n+1}$-cofinite,
%then $M$ is $C_{n}$-cofinite.
%Note that any vertex operator algebra is of itself $C_1$-cofinite.

\begin{dfn}\label{def_N22}
An \textbf{$N=(2,2)$ full vertex operator superalgebra ($N=(2,2)$ full VOA)} is a full vertex operator superalgebra $(F,\om,\omb)$ equipped with vectors $\tau^+,\tau^- \in F_{\frac{3}{2},0}$, $\btau^+,\btau^- \in F_{0,\frac{3}{2}}$, $J \in F_{1,0}$ and $\bJ \in F_{0,1}$ satisfying the following conditions:
\begin{enumerate}
\item
$\tau^\pm, J \in \ker \Ld(-1)$ and $\btau^\pm, \bJ \in \ker L(-1)$.
Set
\begin{align*}
Y(\tau^\pm,\uz)=\sum_{r \in \ft+\Z} G_r^\pm z^{-\frac{3}{2}-r},\quad\quad Y(J,\uz) = \sum_{n\in \Z}J_n z^{-n-1},\\
Y(\btau^\pm,\uz)=\sum_{r \in \ft+\Z} \bG_r^\pm \z^{-\frac{3}{2}-r},\quad\quad Y(\bJ,\uz) = \sum_{n\in \Z}\bJ_n \z^{-n-1}.
\end{align*}
\item
$\{L(n),J(n),G_r^\pm \}$ and $\{\Ld(n),\bJ(n),\bG_r^\pm \}$ satisfy the commutator relation of $\g_0^{N=2} \oplus \g_0^{N=2}$.
%For any $n \geq 1$, $m \geq 0$ and $k \geq 2$ or $k=0$,
%\begin{align}
%\begin{split}
%L(n)J= L(n) \tau^\pm=0,\quad \Ld(n)J= \Ld(n) \btau^\pm&=0\\
%J(n)\tau^\pm =0, J(0)\tau^\pm=\pm \tau^\pm,\quad \bJ(n)\btau^\pm =0, \bJ(0)\btau^\pm&=\pm \btau^\pm\\
%J(1)J= \frac{c}{3}\va, J(k)J=0,\quad \bJ(1)\bJ= \frac{\bc}{3}\va, \bJ(k)\bJ&=0\\
%\tau^+(m)\tau^+ = \tau^-(m)\tau^-=0,\quad \btau^+(m)\btau^+ = \btau^-(m)\btau^-&=0\\
%\tau^+(0)\tau^- =\om+\ft L(-1)J,\quad \tau^+(1)\tau^- =J, \tau^+(2)\tau^-&=\frac{c}{3}\\
%\btau^+(0)\btau^- =\omb+\ft \Ld(-1)\bJ,\quad  \btau^+(1)\btau^- =\bJ, \btau^+(2)\btau^-&=\frac{\bc}{3}.
%\end{split}
%\label{eq_susy_ope}
%\end{align}
\item
$J(0), \bJ(0)$ is semisimple on $F$ with real eigenvalues and all the eigenvalues of $J(0)-\bJ(0)$ are integers which satisfies
\begin{align}
\exp(\pi i (J(0)-\bJ(0))) = (-1)^F
\label{eq_fermi_integer}
\end{align}
as a linear map on $F$.
\end{enumerate}
\end{dfn}

Note that (2) is equivalent to impose the OPE relation in \eqref{eq_susy_ope}.
The assumption (3) in Definition \ref{def_N22} can be found in physics \cite[(1.15)]{LVW},
which plays an important role in this paper.

Let $F$ be an $N=(2,2)$ full VOA. 
%本論文を通じて我々は次の記法を用いる。
Throughout this paper we use the following notation:
Let $H_l=\R J$ and $H_r =\R \bJ$ 
be a vector space over $\R$ equipped with the bilinear forms $(J,J)_l = \frac{c}{3}$ and $(\bJ,\bJ)_r = \frac{\bc}{3}$.
Throughout this paper we assume $c,\bc \neq 0$. Therefore, these are non-degenerate bilinear forms.
For $\al,\be \in H=H_l\oplus H_r$, $(\al,\be)_l$ and $(\al,\be)_r$ are determined by projection of $H$ onto $H_l$ and $H_r$, respectively.
%この論文を通じて $c,\bc \neq 0$と仮定する。よってこれらは非退化な bilinear formである。
%また$\al,\be \in H=H_l\oplus H_r$の元に対して $(\al,\be)_l$ and $(\al,\be)_r$をそれぞれ$H$の$H_l$ and $H_r$への projection によって定める。
For $\al \in H = H_l\oplus H_r$, set
\begin{align*}
F_{h,\h}^\al &= \{v \in F_{h,\h} \mid J(0)v =(\al,J)_l v \text{ and } \bJ(0)v =(\al,\bJ)_r v\}
\end{align*}
and $F^\al = \bigoplus_{h,\h\in \R} F_{h,\h}^\al$. Then,
\begin{align*}
F = \bigoplus_{\substack{h,\h \in \R \\ \al \in H}}F_{h,\h}^\al.
\end{align*}

By Proposition \ref{prop_translation}, for any $\la,\mu\in \R$, 
\begin{align}
\exp(i \la J(0)+ i\mu \bJ(0)):F \rightarrow F
\label{eq_exp_JJ}
\end{align}
is a full vertex operator algebra automorphism (see Definition \ref{def_automorphism}), which sends
\begin{align*}
\tau^\pm \mapsto \exp(\pm i \la) \tau^\pm \quad\text{and}\quad
\btau^\pm \mapsto \exp(\pm i \mu) \btau^\pm
\end{align*}
by
\begin{align}
\tau^\pm \in F_{\frac{3}{2},0}^{\pm \frac{3}{c}J},\quad
\btau^\pm \in F_{0,\frac{3}{2}}^{\pm \frac{3}{\bc}\bJ}.
\label{eq_wt_susy_current}
\end{align}

Next, we will introduce a notion of twisted modules of an $N=(2,2)$ full VOA $F$,
which is a direct generalization of the chiral cases \cite{FLM,DLM,Barron}.
We can define twisted modules for general automorphisms, but this is not necessary in this paper, so we consider only the twists by the automorphisms \eqref{eq_exp_JJ}
%我々は一般の自己同型写像の場合にtwisted module を定義することができるが本稿では必要ないため、
%\eqref{eq_exp_JJ}による twist のみを考える 
(for the general definition of twisted modules in the chiral case, see for example \cite{FLM,DLM}):

%Let $V$ be a vertex operator superalgebra.
%Assume that $h \in V_1$ satisfies $L(n)h=0$ for any $n\geq 1$ and $h(0) \in \End V$ is semisimple on $V$ with real eigenvalues. 
%Then, for any $t\in \R$, $\exp(2\pi i t h(0))$ is a vertex superalgebra automorphism such that $\exp(2\pi i t h(0))\om=\om$.
%Although twisted modules are defined for general automorphisms of a vertex operator superalgebra, in this paper it is sufficient to consider only those that can be given by the exponential of the derivation $h(0)$
%twisted 加群は頂点作用素代数の一般の自己同型写像に対して定義されているが、本論文では微分$h(0)$の指数関数でかけるものだけを考えれば十分である。一般のtwisted module の定義はたとえば\cite{}を参照されたし。
\begin{dfn}\label{def_twisted}
Let $\la,\mu\in \R$.
An $\exp(2\pi i (\la J(0) -\mu \bJ(0)))$-twisted module of $F$ is an $H \times \R^2$-graded vector space
$M =  \bigoplus_{\substack{h,\h \in \R \\ \be \in H}}M_{h,\h}^\be$ and a linear map
\begin{align*}
Y_M(\bullet,\uz):F \rightarrow \End M [[z^\R,\z^\R]],\quad a \mapsto Y_M(a,\uz) = \sum_{r,s\in \R} a(r,s)z^{-r-1}\z^{-s-1}
\end{align*}
such that:
\begin{enumerate}
\item
For any $a \in F_{h,\h}^\al$ and $m \in M_{h',\h'}^\be$,
\begin{align*}
a(r,s)m \in M_{h+h'-r-1,\h+\h'-s-1}^{\al+\be};
\end{align*}
\item
$Y_M(L(-1)a,\uz)=\frac{d}{dz}Y_M(a,\uz)$ and $Y_M(\Ld(-1)a,\uz)=\frac{d}{d\z}Y_M(a,\uz)$ for any $a\in F$;
\item
For any $K \in \R$ and $\be \in H$, $\sum_{h+\h <K} \dim M_{h,\h}^\be < \infty$;
\item
For any $\be\in H$, there exists $N_\be \in \R$ such that $M_{h,\h}^\be=0$ unless $h \geq N_\be$ and $\h \geq N_\be$;
\item
For any $a\in F^\al$ and $m \in M^\be$,
\begin{align*}
z^{\la(\al,J)_l -\mu(\al,\bJ)_r}Y_M(a,\uz)m \in M((z,\z,|z|^\R)).
\end{align*}
\item
For any $a_i \in F^{\al_i}$ ($i=1,2$ and $\al_i \in H$) and $m \in M$ and $u \in M^\vee = \bigoplus_{\substack{h,\h \in \R \\ \al \in H}}(M_{h,\h}^\al)^*$,
there exists $f(z_1,z_2) \in C^\om(Y_2(\C))$ such that:
\begin{align}
\begin{split}
z_1^{\la(\al_1,J)_l -\mu(\al_1,\bJ)_r}
z_2^{\la(\al_2,J)_l -\mu(\al_2,\bJ)_r}
\langle u, Y_M(a_1,\uz_1)Y_M(a_2,\uz_2) m\rangle &= f(z_1,z_2)|_{|z_1|>|z_2|}\\
(-1)^{|a_1||a_2|}
z_1^{\la(\al_1,J)_l -\mu(\al_1,\bJ)_r}
z_2^{\la(\al_2,J)_l -\mu(\al_2,\bJ)_r}
%z_1^{r_1}z_2^{r_2}
\langle u, Y_M(a_2,\uz_2)Y_M(a_1,\uz_1) m\rangle &= f(z_1,z_2)|_{|z_2|>|z_1|}\\
z_2^{\la(\al_2,J)_l -\mu(\al_2,\bJ)_r}(z_2+z_0)^{\la(\al_1,J)_l -\mu(\al_1,\bJ)_r}|_{|z_2|>|z_0|}
\langle u, Y_M( Y(a_1,\uz_0)a_2,\uz_2) m\rangle &= f(z_2+z_0,z_2)|_{|z_2|>|z_0|},
\end{split}
\label{eq_twist_borcherds}
\end{align}
where $C^\om(Y_2(\C))$ is a vector space real analytic functions on $Y_2(\C)$.
%For any $a\in V^r$ and $b\in V$,
%\begin{align}
%\begin{split}
%z_0^{-1}&\delta\left(\frac{z_1-z_2}{z_0}\right)Y_M(a,z_1)Y_M(b,z_2)
%-(-1)^{|a||b|}z_0^{-1}\delta\left(\frac{z_2-z_1}{-z_0}\right)Y_M(b,z_2)Y_M(a,z_1)\\
%&=z_2^{-1}\left(\frac{z_1-z_0}{z_2}\right)^{-r}
%\delta\left(\frac{z_1-z_0}{z_2}\right)Y_M(Y(a, z_0)b,z_2).
%\end{split}
%\label{eq_twist_jacobi}
%\end{align}
\end{enumerate}
\end{dfn}

\begin{rem}\label{rem_just_module}
Let $\la,\mu\in\R$.
If one of the following equivalent conditions are satisfied:
\begin{enumerate}
\item
$\la(\al,J)_l-\mu(\al,\bJ)_r \in \Z$ for any $\al \in H$ with $F^\al\neq 0$;
\item
$\exp(2\pi i (\la J(0)-\mu \bJ(0)))$ is the identity map on $F$,
\end{enumerate}
then an $\exp(2\pi i (\la J(0)-\mu \bJ(0)))$-twisted module is just a module in the usual sense (see \cite{M1} for the definition of modules of a full vertex algebra).
\end{rem}

From the definition of twisted module we see that:
\begin{itemize}
\item
By (2), for $a \in \ker \Ld(-1) \cap F^\al$,
\begin{align*}
Y_M(a,\uz) = \sum_{s \in \la(\al,J)_l+\Z} a(s)z^{-s-1}.
\end{align*}
\item
By \eqref{eq_wt_susy_current}, let $\{L(n),\Ld(n),G_r^+,G_s^-,\bG_p^+,\bG_q^-, J(n),\bJ(n)\}_{n\in \Z, r \in \ft+\la +\Z, s \in \ft-\la+\Z, p \in \ft+\mu +\Z,q \in \ft-\mu+\Z}$ be operators on $M$ given by
\begin{align}
\begin{split}
Y_{M}(\om,\uz) =\sum_{n\in\Z}L(n)z^{-n-2},\quad &Y_{M}(\omb,\uz) =\sum_{n\in\Z}\Ld(n)\z^{-n-2},\\
Y_{M}(J,\uz) = \sum_{n\in \Z}J_n z^{-n-1},\quad &Y_{M}(\bJ,\uz) = \sum_{n\in \Z}\bJ_n \z^{-n-1}\\
Y_{M}(\tau^+,\uz)=\sum_{r \in \ft + \la+\Z} G_r^+ z^{-\frac{3}{2}-r},\quad
&Y_{M}(\tau^-,\uz)=\sum_{s \in \ft - \la+\Z} G_s^- z^{-\frac{3}{2}-s}\\
Y_{M}(\btau^+,\uz)=\sum_{p\in \ft + \mu+\Z} \bG_p^+ \z^{-\frac{3}{2}-p},\quad
&Y_{M}(\btau^-,\uz)=\sum_{q \in \ft - \mu+\Z} \bG_q^- \z^{-\frac{3}{2}-q}.
\end{split}
\label{eq_twist_def_current}
\end{align}
%\item
%For any $r\in\R$ and $a\in V^r$,
%\begin{align}
%Y_M(a,z) = \sum_{s \in r + \Z} a(s)z^{-s-1}.
%\label{eq_twist_mode}
%\end{align}
\end{itemize}
In the chiral case, the following result is obtained in \cite{Barron}:
\begin{prop}\label{prop_twist_susy}
Let $\la,\mu \in\R$ and $F$ be an $N=(2,2)$ full vertex operator superalgebra and $M(\la,\mu)$ an $\exp( i \la J(0)- i\mu\bJ(0))$-twisted module of $F$.
%Set
%\begin{align*}
%Y_{M(t)}(\om,z) =\sum_{n\in\Z}L(n)z^{-n-2},\quad
%Y_{M(t)}(J,z) = \sum_{n\in \Z}J_n z^{-n-1},\\
%Y_{M(t)}(\tau^+,z)=\sum_{r \in \ft + t+\Z} G_r^+ z^{-\frac{3}{2}-r},\quad
%Y_{M(t)}(\tau^-,z)=\sum_{r \in \ft - t+\Z} G_r^- z^{-\frac{3}{2}-r}.
%\end{align*}
\\
Then, $\{L(n),\Ld(n),G_r^+,G_s^-,\bG_p^+,\bG_q^-, J(n),\bJ(n)\}_{n\in \Z, r \in \ft+\la +\Z, s \in \ft-\la+\Z, p \in \ft+\mu +\Z,q \in \ft-\mu+\Z}$ in \eqref{eq_twist_def_current} gives a representation of the twisted superconformal algebras $\g_{\la}^{N=2}\oplus \g_{\mu}^{N=2}$.
\end{prop}
\begin{proof}
Applying $a_1,a_2 \in \ker \Ld(-1)$ to \eqref{eq_twist_borcherds} and multiplying both side by 
\begin{align}
(z_1\z_1)^{-\la (\al_1,J)_l+\mu (\al_1,\bJ)_r}|z_2|^{-\la (\al_2,J)_l+\mu (\al_2,\bJ)_r},
\label{eq_single_factor}
\end{align}
since \eqref{eq_single_factor} is single-valued real analytic on $Y_2(\C)$, by Definition \ref{def_twisted} (6), there is a real analytic function $g(z_1,z_2) \in C^\om(Y_2(\C))$ such that:
\begin{align*}
\z_1^{-\la(\al_1,J)_l +\mu(\al_1,\bJ)_r}
\z_2^{-\la(\al_2,J)_l +\mu(\al_2,\bJ)_r}
\langle u, Y_M(a_1,\uz_1)Y_M(a_2,\uz_2) m\rangle &= g(z_1,z_2)|_{|z_1|>|z_2|}\\
(-1)^{|a_1||a_2|}
\z_1^{-\la(\al_1,J)_l +\mu(\al_1,\bJ)_r}
\z_2^{-\la(\al_2,J)_l +\mu(\al_2,\bJ)_r}
%z_1^{r_1}z_2^{r_2}
\langle u, Y_M(a_2,\uz_2)Y_M(a_1,\uz_1) m\rangle &= g(z_1,z_2)|_{|z_2|>|z_1|}\\
\z_2^{-\la(\al_2,J)_l +\mu(\al_2,\bJ)_r}(\z_2+\z_0)^{-\la(\al_1,J)_l +\mu(\al_1,\bJ)_r}|_{|z_2|>|z_0|}
\langle u, Y_M( Y(a_1,\uz_0)a_2,\uz_2) m\rangle &= g(z_2+z_0,z_2)|_{|z_2|>|z_0|}.
\end{align*}
By $a_1,a_2 \in \ker \Ld(-1)$ and Definition \ref{def_twisted} (2), $g(z_1,z_2)$ is a holomorphic function with possible singularities along $z_1 = 0, z_2=0,z_1-z_2=0$,
and thus, $g(z_1,z_2) \in \C[z_1^\pm,z_2^\pm,(z_1-z_2)^\pm]$.
Hence, By the Residue formula, we get the commutator formula:
\begin{align*}
&[a_1(t),a_2(u)]_+\\
 &= \sum_{k \geq 0}\binom{t}{k} (a_1(k)a_2)(t+u-k)
%&= (\om+\ft L(-1)J)(p+q+1) + (p+\ft) J(p+q)+ \binom{p+\ft}{2} \frac{c}{3}\va (p+q-1)\\
%&= L(p+q)+ \ft (p-q) J(p+q) +\frac{c}{6}(p^2-\frac{1}{4})\delta_{p+q,0}.
\end{align*}
for any $t \in \la(\al_1,J)_l+\Z$ and $u \in \la(\al_2,J)_l+\Z$.
Hence, the assertion holds (see \cite[Proposition 3.11]{M1} for more detail).
%By \eqref{eq_twist_commutator}, for any $p \in \ft+ t+\Z$ and $q\in \ft -t +\Z$,
\end{proof}
In this paper we will show that for an $N=(2,2)$ full VOA $F$, a family of $\exp( 2\pi i \la J(0)- 2\pi i\mu\bJ(0))$-twisted modules $F(\la,\mu)$ for all $\la,\mu \in \R$ can be canonically constructed, which is a mathematical formulation of the spectral flow in physics.

%この論文では$N=(2,2)$ full VOA $F$に対して、an $\exp( i \la J(0)- i\mu\bJ(0))$-twisted module $F(\la,\mu)$が canonical に構成できることを示す。
%\begin{rem}
%\label{rem_spectral_flow_physics}
%これは物理における spectral flow twist と呼ばれる概念の数学的定式化である。
%物理的には the family $\{F(\la,\la)\}_{\la \in \R}$ は fermion の境界条件を連続的に変えたfamily に対応する。
%定義\ref{def_N22} (3)より任意の$N \in \Z$に対して
%\begin{align*}
%\exp(2 \pi i N (J(0)-\bJ(0)))=\id_F
%\end{align*}
%であるから、$F(N,N)$は untwisted module である。
%higher genus まで consistent な full CFT は非自明な加群を持たないはずであるから、
%\begin{align*}
%F \cong F(N,N) \quad\quad \text{as $F$-modules for any $N \in \Z$}
%\end{align*}
%が成り立つべきである。我々はこれを spectral flow の periodicity と呼び、その必要十分条件や、そこから topological twist の上に Hodge-Serre duality が導かれることを見る。
%%of twisted modules is a family with modified boundary conditions of fermions,
%%hence, it is natural to assume (2), where the periodicity coming from the geometry.
%\end{rem}
%
%Therefore, on $M(\ft)$, $Y_{M(\ft)}(\tau^\pm,z) = \sum_{n\in \Z} G_n^\pm z^{-\frac{3}{2}-n}$. 
%Such a vector space $M(\ft)$ is called a {\it Ramon sector} and corresponds to imposing the peroidic boundary condition of fermions on $S^1$.
%In particular, a general twisted module $M(t)$ corresponds to the twisted boundary condition of fermions and plays an important role for the formulation of the spectral flow in this paper.

\subsection{Unitarity of $N=(2,2)$ full VOA}
\label{sec_unitary}
The unitarity of vertex operator algebras was studied in \cite{Y, AL17}. In this section, we briefly review \cite{AMT, M9}, which generalizes the definition of unitary vertex operator superalgebra \cite{Y,AL17} to the full case, and define unitary $N=(2,2)$ full VOAs and their mirror algebras.
%
%chiralな vertex operator superalgebra 場合、unitary 性は\cite{Y,AL17}で研究された。
%この章では\cite{Y}の定義を full の場合に一般化した\cite{M25}を簡潔に振り返り、$N=(2,2)$ full VOA やそのミラーを定義する。
%A bilinear form $(\bullet,\bullet)_{LN}:V\otimes V \rightarrow \C$ is called {\it invariant} if
%\begin{align}
%(u,Y(a,z)v) &=
%\begin{cases}
% (Y(\exp(L(1)z) (-1)^{L(0)+2L(0)^2} z^{-2L(0)}a, z^{-1})u,v),\quad&\text{ (Definition in \cite{Y, AL17})}\\
%  (Y(\exp(L(1)z) \exp(\pi i L(0)) z^{-2L(0)}a, z^{-1})u,v),\quad&\text{ (Definition in \cite{LY22})}
%\end{cases}
%\label{eq_bilinear_inv}
%\end{align}
%for any $a,u,v \in V$. 
%
%
%どちらの定義もeven 部分である vertex operator algebra に制限すれば一致する。
%これらの定義は\ref{sec_Koszul}章で述べたように、supervector space 上の 2通りのHermite 内積の定義に対応する。
%どちらの定義も本質的には同値であり \cite{Stehouwer}物理の多くの文献は\cite{AL17}の定義を採用しているが、このうち Koszul sign rule と compatible なのは\cite{LY22}の方であるため、圏論的な自然さから我々は \cite{LY22}の定義を full vertex operator superalgebra に拡張したものを採用する。
%
%bosonic な full vertex operator algebra 上の invariant bilinear form は\cite{}で導入された。
%またinvariant bilinear form の正定値性は \cite{AMT}で導入され、別の場の量子論の定式化 (OS axioms and Wightman axioms) との関係を調べる上で重要な役割を果たした。ここではその定義を superalgebra に拡張する。
%この章では$N=(2,2)$ full VOA のunitary 性を定義する。
The following definition can be found in \cite{AMT, M9}:
\begin{dfn}\label{def_full_bilinear}
Let $(F, Y,\om,\omb)$ be a full vertex operator superalgebra.
\begin{itemize}
\item
A bilinear form $(\bullet,\bullet):F\otimes F \rightarrow \C$ is called {\it invariant} if
\begin{align}
(u,Y(a,\uz)v) = (Y(\exp(L(1)z+\Ld(1)\z) (-1)^{(L(0)-\Ld(0))+2(L(0)-\Ld(0))^2}z^{-2L(0)}\z^{-2\Ld(0)}a, \uz^{-1})u,v)
\label{eq_bilinear_inv}
\end{align}
for any $a,u,v \in F$.
\item
An anti-linear automorphism $\phi$ of $F$ is an
anti-linear isomorphism $\phi:F \rightarrow F$ such that $\phi(\va) = \va, \phi(\om)=\om, \phi(\omb)=\omb$ and 
$\phi(a(r,s)b) = \phi(a)(r,s)\phi(b)$ for any $a,b\in F$ and $r,s \in \R$.
\end{itemize}
\end{dfn}
Throughout this paper, we normalize the invariant bilinear form by $(\vac,\vac)=1$.
Then, we have:
\begin{prop}\cite[Proposition 2.13]{M9}
Let $(F, Y,\va,\om,\omb)$ be a full vertex operator superalgebra with a non-degenerate invariant bilinear form $(\bullet,\bullet)$
and $\phi:F \rightarrow F$ be an anti-linear involution, i.e.\! an anti-linear automorphism of order 2.
Assume that $F_{0,0}=\C\va$ and $(\vac,\vac)=1$.
Then, the following properties hold:
\begin{enumerate}
\item
For any $a,b \in F$,
\begin{align*}
\overline{(\phi(a),\phi(b))} = (a,b).
\end{align*}
\item
The sesquilinear form $\langle \bullet,\bullet \rangle = (\phi(\bullet),\bullet)$ is Hermite, that is, for any $a,b\in F$
\begin{align*}
\langle a,b \rangle = \overline{\langle b, a \rangle}.
\end{align*}
\end{enumerate}
\end{prop}

\begin{dfn}\label{def_}
Let $(F, Y,\va,\om,\omb)$ be a full vertex operator superalgebra with an invariant bilinear form $(\bullet,\bullet)$
and $\phi:F \rightarrow F$ be an anti-linear involution. The pair $(F,\phi)$ is called \textbf{unitary} if the Hermite form $\langle \bullet,\bullet \rangle = (\phi(\bullet),\bullet)$ is positive-definite.
\end{dfn}
%
%\begin{prop}
%Let $F,G$ be a unitary full vertex operator superalgebras. 
%Then, the full vertex operator superalgebra $F \otimes G$ equipped with the invariant bilinear form $(\bullet,\bullet)_{F\otimes G}$ in Definition \ref{def_Hermite_tensor} and the anti-linear involution $\phi_F\otimes \phi_G$ is unitary.
%\end{prop}

\begin{dfn}\label{def_full_susy}
A unitary $N=(2,2)$ full vertex operator superalgebra is an $N=(2,2)$ full vertex operator superalgebra $(F,\om,\omb,\tau^\pm,\bar{\tau}^\pm,J,\bJ)$ which is unitary as a full vertex operator superalgebra by an invariant bilinear form $(\bullet,\bullet)$ and an anti-linear involution $\phi:F \rightarrow F$ such that:
\begin{align}
\begin{split}
\phi(J)=-J,\quad \phi(\tau^+) = \tau^-,\quad \phi(\tau^-)= \tau^+\\
\phi(\bJ)=-\bJ, \quad \phi(\bar{\tau}^+) = -\bar{\tau}^-,\quad \phi(\bar{\tau}^-)= -\bar{\tau}^+.
\end{split}
\label{eq_unitary_susy_full}
\end{align}
\end{dfn}

Let $F$ be a unitary full vertex operator superalgebra.
Note that for $a \in F_{h,\h}$ with $L(1)a = \Ld(1)a =0$,
\begin{align}
\begin{split}
\langle u,Y(a,\uz)v \rangle
&= (\phi(u),Y(a,\uz)v )\\
&=(-1)^{(h-\h)+2(h-\h)^2}(z^{-2h}\z^{-2\h}Y(a,\uz^{-1})\phi(u),v )\\
&=(-1)^{(h-\h)+2(h-\h)^2}\langle z^{-2h}\z^{-2\h}Y(\phi(a),\uz^{-1})u,v \rangle.
\end{split}
\label{eq_unitary_inv}
\end{align}

\begin{prop}\label{prop_unitary_adjoint}
Let $F$ be a unitary $N=(2,2)$ full vertex operator superalgebra with the positive-definite Hermite form $\langle -,- \rangle$. Then, 
\begin{align}
\begin{split}
\langle u,L(n)v \rangle = \langle L(-n)u,v \rangle,\quad&\langle u,\Ld(n)v \rangle = \langle \Ld(-n)u,v \rangle,\\
\langle u,J(n)v \rangle = \langle J(-n)u,v \rangle,
\quad & \langle u,\bJ(n)v \rangle = \langle \bJ(-n)u,v \rangle\\
\langle u,G^+(r)v \rangle =  \langle G^-(-r)u,v \rangle,\quad
&\langle u,\bG^+(r)v \rangle = \langle \bG^-(-r)u,v \rangle,\\
 \langle u,G^-(r)v \rangle = \langle G^+(-r)u,v \rangle,
 \quad &  \langle u,\bG^-(r)v \rangle = \langle \bG^+(-r)u,v \rangle,
\end{split}
\label{eq_adjoint_formula}
\end{align}
for any $n\in \Z$, $r\in \ft+\Z$ and $u,v \in F$. 
\end{prop}
\begin{proof}
%We will only check $\langle u,G^+(r)v \rangle = \langle G^-(-r)u,v \rangle$.
We will only check for $G_r^+$ and $\bG^+$.
By \eqref{eq_unitary_inv}, we have
\begin{align*}
\langle u, Y(\tau^+,\uz)v \rangle &= (-1)^{\frac{3}{2}+2 \frac{9}{4}}
\langle Y(\tau^-,\uz^{-1})u,v\rangle = \langle Y( \tau^-,\uz^{-1})u,v\rangle
\end{align*}
and
\begin{align*}
\langle u, Y(\btau^+,\uz)v \rangle &= (-1)^{-\frac{3}{2}+2 \frac{9}{4}}
\langle Y(-\btau^-,z^{-1})u,v\rangle = \langle Y(\btau^-,z^{-1})u,v\rangle.
\end{align*}
\end{proof}
\begin{rem}\label{rem_sign_strange}
The reader may find the difference in sign between the holomorphic and anti-holomorphic parts of the definition \ref{def_full_susy} strange, that is, $\phi(\tau^+)=\tau^-$ and $\phi(\btau^+) =- \btau^-$.
If we assume $\phi(\btau^+)=\btau^-$, then by the proof of Proposition \ref{prop_unitary_adjoint}, \eqref{eq_adjoint_formula} changes into 
\begin{align*}
\langle G_r^+ u, v\rangle =  \langle u, G_{-r}^-  v\rangle \quad \text{ and }\quad
\langle \bG_r^+ u, v\rangle = - \langle u, \bG_{-r}^-  v\rangle.
\end{align*}
Hence, the minus sign appears in the adjoint operator.
Such a difference in sign between the holomorphic and anti-holomorphic parts already appears in Definition \eqref{eq_bilinear_inv}.
Geometrically $L(0)-\Ld(0)$ generates $\mathrm{Spin}(2)$, 
which seems to be the origin of this sign (see also \cite{M9}).
%こうした正則部分と反正則部分の符号の違いは、\eqref{eq_bilinear_inv}の定義に既に現れている。
%幾何学的には$L(0)-\Ld(0)$が $\mathrm{Spin}(2)$を生成することが、この符号の由来であり不可避なものであると思われる (see \cite{M9})。
\end{rem}
By Proposition \ref{prop_unitary_adjoint},  we have:
\begin{align*}
\langle \om, \om \rangle = \langle L(-2)\vac, L(-2)\vac \rangle = \langle \vac, L(2)L(-2)\vac \rangle  = c/2.
\end{align*}
Hence, we have:
\begin{cor}
\label{cor_c_zero}
Let $F$ be a unitary full vertex operator algebra. Then, 
if $\om \neq 0$ (resp. $\omb \neq 0$), then $c >0$ (resp. $\bc > 0$) holds.
\end{cor}
If $\om =0$, then $L(-1) =0$, which implies $Y(\bullet,\uz)$ does not depend on $z$, which is trivial. Hence, we will always assume that
\begin{itemize}
\item
$c >0$ and $\bc >0$.
\end{itemize}

\begin{dfn}\label{def_mirror}
Let $(F,\om,\omb,\tau^\pm,\btau^\pm,J,\bJ,(-,-),\phi)$ be a unitary $N=(2,2)$ full vertex operator superalgebra.
Define a new $N=(2,2)$ structure on $(F,\om,\omb)$ by
\begin{align*}
\tau_{\mathrm{new}}^\pm = \tau^\pm,
\quad &\btau_{\mathrm{new}}^\pm = \btau^\mp,\\
J_{\mathrm{new}} = J,\quad &\bJ_{\mathrm{new}} = -\bJ.
\end{align*}
Then, $(F,\om,\omb,\tau_{\mathrm{new}}^\pm,\btau_{\mathrm{new}}^\pm, J_{\mathrm{new}},\bJ_{\mathrm{new}})$ satisfy \eqref{eq_unitary_susy_full} for $\phi$, and thus, a unitary $N=(2,2)$ full vertex operator algebra.
We denote it by $\tilde{F}$ and call {\it the mirror algebra}.
\end{dfn}
\begin{rem}
\label{rem_four_mirror}
We can consider four $N=(2,2)$ structures $\{F,\om,\omb, J_{\ep,\bar{\ep}},\bJ_{\ep,\bar{\ep}}, \tau_{\ep,\bar{\ep}}^\pm,\btau_{\ep,\bar{\ep}}^\pm\}$ for each $\ep, \bar{\ep} \in \Z_2$ by
\begin{align*}
\tau_{\ep,\bar{\ep}}^\pm = \tau^{\ep \pm},
\quad &\btau_{\ep,\bar{\ep}}^\pm = \btau^{\bar{\ep}\pm},\\
J_{\ep,\bar{\ep}} = \ep J,\quad &\bJ_{\ep,\bar{\ep}} = \bar{\ep}\bJ.
\end{align*}
However, the anti-linear involution $\phi:F \rightarrow F$ gives an (anti-linear) isomorphism between $(\ep,\bar{\ep})$ and $(-\ep,-\bar{\ep})$, and thus, they are essentially the same. Hence, we will consider $(1,1) $ and its mirror $(1,-1)$.
\end{rem}

\subsection{Classification of twists}\label{sec_class_twist}
%超対称性を持つ場の量子論に対しては、超対称性を用いて理論を twist することで情報を落とした別の理論を構成することができる。
%場の量子論の物理量を計算することは一般に難しいが、twist した後の物理量はしばしば数学的に定式化可能な量となる。
%こうした理由から twist は物理において非常に良く調べられてきた。
%\cite{}は factorization algebra の概念を用いて、こうした twist を数学的に定式化し、Holomorphic / topological twist の概念を導入した。この章では $N=(2,2)$ full VOA の定式化とその自己同型群 (Proposition \ref{}) を元に \cite{}の topological twist を分類する。

For quantum field theories with supersymmetry, it is possible to construct another theory by dropping information, called a twist.
Although it is generally difficult to compute physical quantities in quantum field theories, the physical quantities after twisting are often mathematically formulable and much easier.
For this reason, twists have been very well investigated in physics.
Using the concept of factorization algebras,
Elliott and Safronov formulated twists mathematically and introduced the concept of holomorphic / topological twists \cite{ES}. In this section, we classify the topological twists of unitary $N=(2,2)$ full VOAs based on the definition of \cite{ES}.

Let $S$ be a subspace of the Lie superalgebra $\g_{\la=0}^{N=2} \oplus \g_{\la=0}^{N=2}$ spanned by the super translations, i.e., 
\begin{align*}
S = \C G_{-\ft}^+ \oplus \C G_{-\ft}^- \oplus \C \bG_{-\ft}^+ \oplus \C \bG_{-\ft}^-.
\end{align*}
\begin{dfn}[\cite{ES}]\label{def_hol_top_twist}
An element $d \in S$ with $[d,d]_+=0$ is called 
\begin{itemize}
\item
a {\bf holomorphic twist} if the image of the linear map $[d,\bullet]:S \rightarrow \C L(-1) \oplus \C \Ld(-1), a\mapsto [d,a]$ is one-dimensional;
\item
 a {\bf topological twist} if $[d,\bullet]:S \rightarrow \C L(-1) \oplus \C \Ld(-1)$ is surjective.
\end{itemize}
\end{dfn}
Let $p,q,r,s \in \C$ and set 
\begin{align}
d = p G_{-\ft}^+ +q G_{-\ft}^- +r\bG_{-\ft}^+ + s\bG_{-\ft}^-. \label{eq_twist_all}
\end{align}
Then, $[d,d]_+ =0$ if and only if $pq=0$ and $rs=0$.

\begin{prop}\label{prop_classify_twist}
Any holomorphic twist is a scalar multiple of $G_{-\ft}^+,G_{-\ft}^-,\bG_{-\ft}^+$ or $\bG_{-\ft}^-$.
Any topological twist is one of the followings:
\begin{align*}
\la G_{-\ft}^+ + \mu \bG_{-\ft}^+,\quad \la G_{-\ft}^+ + \mu \bG_{-\ft}^-,\quad \la G_{-\ft}^- + \mu \bG_{-\ft}^+,
\quad\la G_{-\ft}^- + \mu \bG_{-\ft}^-,
\end{align*}
where $\la,\mu \in \C^\times$.
\end{prop}

Let $F$ be an $N=(2,2)$ unitary full vertex operator algebra.
We should then classify topological twists under the automorphism group of the full VOA. 
Among the automorphisms of Definition \ref{def_VOA_alpha}, $\al$ is not necessarily lifted to an automorphism of $F$ in general (basically, it is never lifted).
But $\exp(i\la J(0)+i\mu \bJ(0))$ is an inner automorphism, so it always defines an automorphism of $F$. Under inner automorphisms and the anti-linear involution $\phi$, there are essentially only two topological twists:
\begin{align}
\begin{split}
d_B&= G_{-\ft}^+ + \bG_{-\ft}^+\\
d_A&= G_{-\ft}^+ + \bG_{-\ft}^-,
\end{split}
\label{eq_AB_twist}
\end{align}
which are called the {\bf B-twist} and the {\bf A-twist}, respectively in physics.
The purpose of this paper is to mathematically examine these twists.

\subsection{2d topological field theory and Frobenius superalgebra}\label{sec_tft}
In this section, we review the definition of 2d topological field theory (2d TFT) and briefly discuss the relation between 2d TFTs and commutative Frobenius algebras,
and the unitary structure on 2d TFT (for more details, see \cite{Atiyah,Abrams,St1,St2}).
This section does not contain anything new, but is written to prepare some notation on Frobenius superalgebras and for the reader's convenience.

A (non-extended) $d$-dimensional topological field theory is a symmetric monoidal functor from the symmetric monoidal category of $d$-dimensional bordisms to the given symmetric monoidal category, called the target category. In this paper, we consider the case that the target category is a category of the finite dimensional supervector space $\sVect$, where the braiding is given by the Koszul sign rule.

%この章では2d 位相的場の理論(2d TFT)の定義を振り返り、2d TFTと commutative Frobenius algebra の関係を簡潔に述べる。
%またこの章の最後に Atiyah や ? に従い位相的場の量子論のunitary 構造について復習する。unitary SCFT の topological twist で構成されるTFTは Atiyahの意味で unitary に成りえないことが分かる。
%
%(拡張されていない)位相的場の理論とは Cobordism のなす symmetric monoidal category から target category と呼ばれる symmetric monoidal category への symmetric monoidal functor のことである。
%この論文では target category は 有限次元 supervector space のなす圏 $\sVect$ (Definition \ref{def_sVect})を考える。
%

Let $\Bord$ be the bordism category 
whose objects are oriented closed 1-manifold, i.e., disjoint union of oriented $S^1$ and morphisms from $Y_0$ to $Y_1$ are compact oriented 2-manifolds with boundaries $\Sigma$ together with a partition of the boundary into two parts
\begin{align*}
\pa\Sigma = (\pa\Sigma)_{\text{in}}\sqcup (\pa\Sigma)_{\text{out}}
\end{align*}
and $\mathrm{GL}_2^+(\R)$-diffeomorphism 
\begin{align}
Y_0^* \cong (\pa\Sigma)_{\text{in}},\quad 
Y_1 \cong (\pa\Sigma)_{\text{out}},
\label{eq_orientation}
\end{align}
where on $Y_i$, we consider the $\mathrm{GL}_2^+(\R)$-structure, induced from the orientation, on a once-stabilized tangent space $TY_i\oplus \R$, and $Y_0^*$ is the 1-manifold with the reverse $\mathrm{GL}_2^+(\R)$-structure.
Here, we consider the homotopy equivalence classes of  bordisms. Then, $\Bord$ inherits a symmetric monoidal category structure by the disjoint union of manifolds (for more detail, see for example \cite{St1}).

\begin{dfn}\label{def_TFT}
A {\bf 2d topological field theory} (2d TFT) with the target category $\sVect$ is a symmetric monoidal functor 
\begin{align*}
Z:\Bord \rightarrow \sVect.
\end{align*}
%A {\it unitary 2d topological field theory} (with the target category $\sHilb$) is a symmetric monoidal dagger functor 
%\begin{align*}
%Z:\Bord \rightarrow \sHilb.
%\end{align*}
\end{dfn}

It is well-known that 2d TFT can be classified by Frobenius superalgebras (see \cite{Abrams}). 
We will briefly review this result.
\begin{dfn}\label{def_Frob}
A \textbf{Frobenius superalgebra} is a finite-dimensional $\Z_2$-graded vector space $A=A_0\oplus A_1$ equipped with a  linear map 
$m:A \otimes A \rightarrow A$ and a unit $\eta:\C \rightarrow A$ and a counit $\ep:A \rightarrow \C$ such that:
\begin{enumerate}
\item
All maps $m, \eta,\ep$ are even;
\item
$m$ and $\eta$ define a unital associative algebra structure on the vector space $A$;
\item
The product is supercommutative, that is, $m(a,b)= (-1)^{|a||b|} m(b,a)$ for any $a,b\in A$;
\item
A bilinear form $(-,-)_\ep:A\otimes A \rightarrow \C$ given by
\begin{align*}
(a,b)_\ep = \ep(m(a,b)),\quad\quad(\text{ for }a,b \in A)
\end{align*}
is non-degenerate.
\end{enumerate}
\end{dfn}

Let $A$ be a Frobenius superalgebra.
Let $\{e_i \in A\}_{i \in I}$ be a basis of $A$ such that 
$I=I_0 \sqcup I_1$ and $e_i$ is in $A_0$ (resp. $A_1$) if $i\in I_0$ (resp. $i\in I_1$) and $\{e^i\}_{i\in I}$ the right dual basis with respect to the bilinear form $(-,-)_\ep$, that is, they satisfy:
\begin{align}
(e_i,e^j)_\ep = \delta_{i,j}. \label{eq_right_dual}
\end{align}
Note that since $(-,-)_\ep$ is not symmetric on $A_1$ the order of \eqref{eq_right_dual} is important.
Then, define a linear map $\delta: A \rightarrow A\otimes A$ by
\begin{align}
\delta(x) = \sum_{i \in I} (x\cdot e^i) \otimes e_i.
\label{eq_delta}
\end{align}

\begin{prop}\label{prop_delta_super}
The following properties hold:
\begin{enumerate}
\item
$ \sum_{i \in I} (x\cdot e^i) \otimes e_i=  \sum_{i \in I} e^i \otimes (e_i \cdot x)$ for any $x\in A$;
\item
$\tau \circ \delta =\delta$, where $\tau:A\otimes A \rightarrow A\otimes A$ is the braiding of $\sVect$;
\item
$(\ep \otimes \id_A) \circ \delta = \id_A = (\id_A \otimes \ep) \circ \delta$;
\item
$(m\otimes \id_A) \circ (\id_A\otimes \delta) = \delta \circ m = (\id_A\otimes m) \circ (\delta \otimes \id_A)$.
\end{enumerate}
\end{prop}
\begin{proof}
To show (1), evaluate the left-hand side tensor product by $(a,-)_\ep \in A^*$ for $a\in A$. Then,
\begin{align*}
\sum_{i \in I} (a, x e^i)_\ep e_i = \sum_{i \in I} (ax, e^i)_\ep e_i = ax
\end{align*}
and
\begin{align*}
\sum_{i \in I} (a,e^i)_\ep e_i \cdot x= ax.
\end{align*}
Hence, (1) holds. (3) and (4) follows from (1). (2) follows from the definition.
\end{proof}

(2) and (3) mean that $(A,\delta,\ep)$ is a commutative comonad in $\sVect$.
(4) is called the {\it Frobenius relation}.

\begin{figure}[h]
  \begin{minipage}[l]{0.4\linewidth}
    \includegraphics[width=5cm]{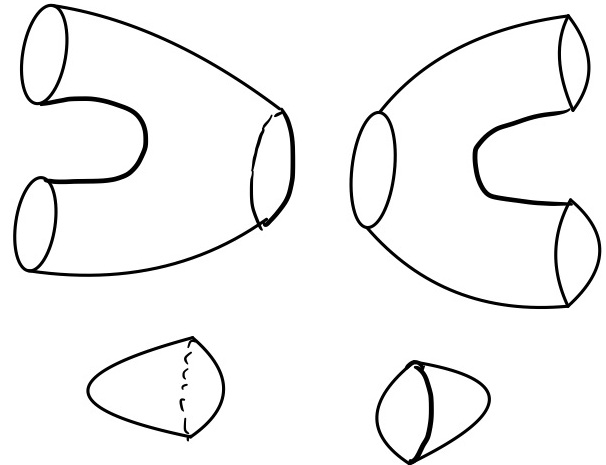}
\caption{}
\label{fig_TFT1}
  \end{minipage}
    \begin{minipage}[h]{0.4\linewidth}
    \vspace{-2mm}
        \flushright
    \includegraphics[width=5.0cm]{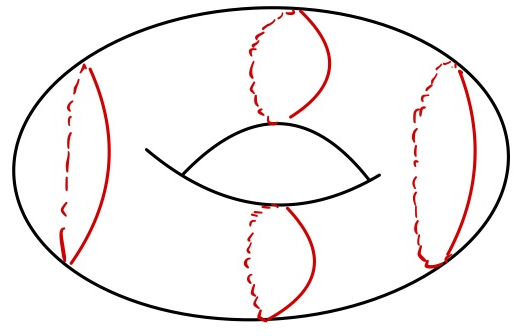}
    \caption{}\label{fig_TFT2}
  \end{minipage}
\end{figure}

By the Morse theory, any two-dimensional manifold with boundaries decomposes into the four fundamental 2-manifolds in Figure \ref{fig_TFT1}.
These four manifolds correspond to the four operations on a Frobenius algebra $A$:
\begin{align*}
m:A \otimes A \rightarrow A,\quad\quad \delta:A \rightarrow A\otimes A,\quad \eta:\C \rightarrow A,\quad\quad \ep:A \rightarrow \C.
\end{align*}
For example, the partition function of $\mathrm{Torus}$ is from Figure \ref{fig_TFT2}
\begin{align*}
Z(\mathrm{Torus}) =\ep\circ m\circ \delta \circ \eta(1) \in \C.
\end{align*}
It can be seen from the definition of a Frobenius algebra that such partition functions do not depend on how a manifold is decomposed into fundamental pieces in Figure \ref{fig_TFT1} \cite{Atiyah, DJ}.
%
%モース理論によって任意の境界付き二次元多様体は、図1の四つの基本的な境界付き二次元多様体に分解する。この四つが Frobenius algebra $A$上の 四つのoperations に対応する。
%\begin{align*}
%m:A\otimes A \rightarrow A,\quad&\quad \delta:A \rightarrow A\otimes A,\\
%\eta:\C \rightarrow A, \quad&\quad \ep:A \rightarrow \C.
%\end{align*}
%たとえば$\mathrm{Torus}$の分配関数は 図2から
%\begin{align*}
%Z(\mathrm{Torus}) =\ep\circ m\circ \delta \circ \eta(1)
%\end{align*}
%となる。こうした分配関数は多様体の図1のピースへの分解の仕方によらないことが、Frobenius relation や associativity から分かる。
\begin{prop}
\label{prop_TFT_frob}
There is a one-to-one correspondence between 2d topological field theory with the target category $\sVect$ and Frobenius superalgebras.
\end{prop}

By definition, we have:
\begin{prop}\label{prop_partition}
Let $A$ be a Frobenius superalgebra and $Z_A:\Bord \rightarrow \sVect$ the associated 2d TFT. Then, we have:
\begin{enumerate}
\item
$Z_A(\text{2-sphere}) = \ep\circ \eta(1)$;
\item
$Z_A(\text{torus}) = \dim A_0 - \dim A_1$;
%\item
%If $A$ is a $*$-Frobenius superalgebra, then $\ep(1)$ is a non-zero positive real number.
\end{enumerate}
\end{prop}
\begin{proof}
(1) is clear. By taking the basis \eqref{eq_right_dual}, by \ref{fig_TFT2}, we have:
\begin{align*}
Z_A(\text{torus}) = \sum_{i \in I_0 \sqcup I_1} \ep(e^i, e_i) = \sum_{i\in I_1}1 - \sum_{i \in I_1} = \dim A_0-\dim A_1.
\end{align*}
\end{proof}

This section ends with a brief description of unitary 2d TFT. For more details, see \cite{Atiyah, St1}. Note that this paper does not examine unitary 2d TFTs themselves, but only to show that topological twists of $N=(2,2)$ SCFTs do not inherit any unitary structure. Therefore, readers who are not interested may skip this part.
%この章の最後に2d TFTのユニタリ性を簡潔に説明する。詳細は\cite{}を参照されたい。なお本論文では topological twist が unitary TFT にならないことを見るため、unitary 2d TFT 自身を調べるわけではない。そのため興味のない読者はこの部分をスキップしても良い。

A {\bf dagger structure} on a symmetric monoidal category $C$ is a functor from the opposite category $C^\op$ of $C$ to $C$ itself
\begin{align*}
\dagger:C^\op \rightarrow C
\end{align*}
which is the identity on objects and involution $\dagger^\op\circ \dagger =\mathrm{id}_C$
and some compatibility with the symmetric monoidal structure.
$\Bord$ inherits a dagger structure by the orientation reverse. Also $\sHilb$ (the category of finite dimensional super Hilbert spaces) has a dagger structure by the Hermite conjugate.
Then, a {\bf unitary 2d topological field theory} with the target category $\sHilb$ is a symmetric monoidal functor $Z:\Bord \rightarrow \sHilb$ which preserves the dagger structure (see \cite{St1,St2} for more detail).
From the definition it is easy to see that (see \cite{DJ}):
\begin{lem}\label{lem_non_zero_star}
Let $Z:\Bord \rightarrow \sHilb$ be a unitary 2d TFT.
Then, $\ep\circ \eta(1)$ is a non-zero positive real number.
\end{lem}

Lemma \ref{lem_non_zero_star} says that if 2d TFT is unitary, then the partition function on $S^2$ is non-zero by Proposition \ref{prop_partition}.
We will see that the topological twist of any non-trivial unitary $N=(2,2)$ SCFT does not satisfy this property.

\section{Cohomology ring of unitary $N=(2,2)$ full vertex operator superalgebra}\label{sec_ring}

In this section, we introduce the cohomology ring of an $N=(2,2)$ full vertex operator algebra with topological twists $d_A,d_B$ in Section \ref{sec_class_twist}.
Section \ref{sec_twist_def} shows that for an $N=(2,2)$ full VOA $F$ which is not necessarily unitary, $H(F,d_A)$ and $H(F,d_B)$ are $\Z_2$-graded algebras.
In Section \ref{sec_inequality}, we see that by assuming unitarity, these cohomology rings become $\R^2$-graded algebras. This grading corresponds to the Hodge decomposition and is very important.
In Section \ref{sec_twist_star} and Section \ref{sec_twist_serre}, we examine the dualities between $H(F,d_A)$ and $H(F,d_B)$. We also see that the cohomology rings give 2d topological field theories.
We emphasize that Sections \ref{sec_twist_def} and \ref{sec_inequality} are independent of the spectral flow developed later, and depend only on the properties of the differentials and on unitarity. The proofs of the results in Sections \ref{sec_twist_star} and \ref{sec_twist_serre}, however, require the spectral flow construction established in Section \ref{sec_sec_flow}.
%
%Proofs of the results in Section \ref{sec_twist_star} and Section \ref{sec_twist_serre} will be given in Section \ref{sec_Hodge_Serre}.

%この章では Section \ref{sec_class_twist}における topological twists $d_A,d_B$を用いて、$N=(2,2)$ full vertex operator algebra の cohomology ring を定義する。
%
%Section \ref{sec_def_twist}では unitary とは限らない $N=(2,2)$ full VOA $F$ に対して、 $H(F,d_A)$ and $H(F,d_B)$が$\Z_2\times \Z_2$-graded algebra になることを示す。
%Section \ref{sec_inequality}では、unitary 性を仮定することで、これらの cohomology ring が $\R^2$-graded algebra になることを見る。この次数付けは Hodge decomposition と対応しており、非常に重要である。
%Section \ref{sec_twist_star} and Section \ref{sec_twist_serre}では、$H(F,d_A)$ and $H(F,d_B)$の間の dualities を調べる。またこれらが 2d topological field theory を与えることを見る。

\subsection{Cohomology ring of $N=(2,2)$ full VOA}
\label{sec_twist_def}
%この章では $N=(2,2)$ full VOA のコホモロジー環を定義する。
%これは物理では chiral ring や topological twist と呼ばれている操作である。まず初めにVOAの場合に、コホモロジー環を定義し、その後それをfull VOA の場合に一般化する。

In this section, we define cohomology rings of $N=(2,2)$ full VOA.
This operation is also called a {\it topological twist} and the cohomology ring is called a {\it chiral ring} in physics.
We first define the cohomology ring in the case of an $N=2$ VOA, and then generalize it to the full VOA cases (see Remark \ref{rem_Borisov} for the relation with the chiral ring in \cite{Borisov}).

%Note that the cohomology ring of an $N=2$ vertex operator superalgebra introduced by \cite[Section 2]{Borisov} is different from ours, however,  unitarity

Let $(V,\om,\tau^\pm,J)$ be an $N=2$ vertex operator superalgebra. Let $H=\R J$ be the subspace of $V_1$ equipped with bilinear form $(J,J)= \frac{c}{3}$. For $\al \in H$, set
\begin{align*}
V_n^\al &= \{v \in V_n \mid J(0)v =(\al,J) v \}
\end{align*}
and 
\begin{align*}
Q &= \tau^+(0) = G_{-\ft}^+.
\end{align*}
By \eqref{eq_commutator}, for any $a,b \in V$ and $n\in \Z$,
\begin{align*}
Q(a(n)b) = (Qa)(n) b + (-1)^{|a|} a(n)Qb.
\end{align*}
Hence, $Q$ is a super-derivation, and thus, $\ker Q$ is a subalgebra.
Moreover, by $Q^2=0$, $\im Q \subset \ker Q$ is an ideal of the vertex superalgebra $\ker Q$.

By an easy computation, we have:
\begin{lem}\label{lem_conf}
$\tilde{\om} = \om + \ft L(-1)J = Q \tau^-$ is a conformal vector of $\ker Q$ with central charge $0$, that is, it satisfies 
\begin{enumerate}
\item
$\tilde{\om}(k)\tilde{\om}=0$ for any integer $k\geq 2$.
\item
$\tilde{\om}(0)a=a(-2)\va$ for any $a\in V$;
\item
$\tilde{\om}(1) = L(0) - \ft J(0)$ is semisimple on $V$.
\end{enumerate}
\end{lem}
%\begin{rem}
%Note that $\tilde{\om}(1)$
%\end{rem}

Set
\begin{align*}
H(V,Q) = \ker Q /\im Q.
\end{align*}
By Lemma \ref{lem_conf}, $H(V,Q)$ is a conformal vertex superalgebra with the trivial conformal vector.
Hence, for any $a\in H(V,Q)$, the vertex operator consists of only the constant term, $Y(a,z) = a(-1)$,
and $H(V,Q)$ inherits a structure of a unital supercommutative algebra by $(-1)$-th product:
\begin{align*}
H(V,Q) \otimes H(V,Q) \rightarrow H(V,Q),\quad a\otimes b \mapsto a(-1)b
\end{align*}
and the unit $\va \in H(V,Q)$. In addition, $H(V,Q)$ has a $\Z_2$-grading by the $\Z_2$-grading on $V$.
%さらに$H(V,Q)$には$J(0)$の固有値によって$\Z$次数付けが入る。
\begin{prop}
\label{prop_chiral_ring}
$H(V,Q)$ is a $\Z_2$-graded unital supercommutative algebra by (-1)-th product.
\end{prop}
%This algebra is called {\it a chiral ring} in physics.

Note that there is another direct way to look at the triviality of the vertex operator of the chiral ring as follows:
\begin{prop}\label{prop_constant_chiral}
Let $a_1,\dots,a_{r+1} \in V$ satisfy $Q a_i =0$ and $u \in V^\vee$ satisfy $Q^*u=u(Q -)=0$.
Then,
\begin{align*}
\langle u, Y(a_1,z_1) \cdots Y(a_r,z_{r})a_{r+1}\rangle \in \C,
\end{align*}
that is, all non-constant coefficients of $z_i$ vanish.
\end{prop}
\begin{proof}
For any $i \in \{1,\dots,r \}$,
\begin{align*}
&\frac{d}{dz_i} \langle u, Y(a_1,z_1) \cdots Y(a_r,z_r)a_{r+1}\rangle\\
&=\langle u, Y(a_1,z_1) \cdots Y(L(-1)a_i,z_i) \cdots Y(a_r,z_r)a_{r+1} \rangle\\
&=\langle u, Y(a_1,z_1) \cdots Y([G_{-\ft}^+,G_{-\ft}^-]_+ a_i,z_i) \cdots Y(a_r,z_r)a_{r+1} \rangle\\
&=\langle u, Y(a_1,z_1) \cdots Y(G_{-\ft}^+G_{-\ft}^- a_i,z_i) \cdots Y(a_r,z_r)a_{r+1} \rangle\\
&=\langle u, Y(a_1,z_1) \cdots Y(\tau^+(0)\tau^-(0)a_i,z_i) \cdots Y(a_r,z_r)a_{r+1}\rangle\\
&=\langle u, Y(a_1,z_1) \cdots [\tau^+(0),Y(\tau^-(0)a_i,z_i)]_\pm \cdots Y(a_r,z_r) a_{r+1} \rangle\\
&=\pm\langle u,\tau^+(0) Y(a_1,z_1) \cdots Y(\tau^-(0)a_i,z_i) \cdots Y(a_r,z_r)a_{r+1} \rangle\\
&+\sum_{k \neq i}^{r}\pm  \langle u, Y(a_1,z_1) \cdots [\tau^+(0),Y(a_k,z_k)]_\pm \cdots Y(\tau^-(0)a_i,z_i) \cdots Y(a_r,z_r)a_{r+1}\rangle\\
&\pm \langle u,Y(a_1,z_1) \cdots Y(\tau^-(0)a_i,z_i) \cdots Y(a_r,z_r)\tau^+(0) a_{r+1}\rangle\\
&=0.
\end{align*}
\end{proof}

\begin{lem}\label{lem_chiral_concentrate}
Let $v \in \ker Q \cap V_n^\al$ be a non-zero vector. If $n \neq \frac{(J,\al)}{2}$, then $v \in \im Q$.
\end{lem}
\begin{proof}
Since
$G_{-\ft}^+ G_{\ft}^- v = [G_{-\ft}^+, G_{\ft}^-]_+ v = (n- \frac{(J,\al)}{2})v$,
if $n \neq \frac{(J,\al)}{2}$,
then
$v = \frac{2}{2n-(J,\al)} Q G_{\ft}^- v \in \im Q$.
\end{proof}
%Hence, the chiral ring is a sub-quotient of the vector space
%\begin{align*}
%C(V) = 
%\end{align*}
By Lemma \ref{lem_chiral_concentrate}, $H(V,Q)$ can be identified as a quotient of 
\begin{align}
\bigoplus_{\al \in H} V_{\frac{(J,\al)}{2}}^\al \subset V.
\label{eq_Borisov}
\end{align}

\begin{rem}\label{rem_Borisov}
Borisov directly introduced an algebra structure on \eqref{eq_Borisov} under some assumption \cite[Definition 2.4]{Borisov}, which is called a chiral ring \cite{Borisov}. Although Borisov's definition differs from our definition, for a unitary $N=2$ vertex operator algebra, Borisov's assumption, which is used to define the algebra structure, automatically holds, so they coincide (see also Proposition \ref{prop_isom_ring}).
%Borisov は (full でなく chiralな) $N=2$ vertex operator algebra に対して、chiral ring を導入した。Borisov の定義は cohomology としての構成でなく、我々の cohomology ring と定義が異なるが、unitary $N=2$ vertex operator algebra に対しては、Borisov が algebra 構造を定義するために用いた 仮定 Definition 2.4 が自動的に成り立つため両者は一致する。
\end{rem}

% the following subspace of $V$:
%\begin{align}
%\bigoplus_{\al \in H} V_{\frac{(J,\al)}{2}}^\al.
%\label{eq_subspace_grading}
%\end{align}
%By imposing unitarity on the vertex operator superalgebra, we will later see that the cohomology ring exactly coincides with 
%\eqref{eq_subspace_grading} and inherits a $\Z$-graded Frobenius algebra structure.

%次に上記の構成をfullの場合に一般化し、$N=(2,2)$ full VOAの cohomology ring を定義する。
Next, we generalize the above construction to the full case and define a cohomology ring of $N=(2,2)$ full VOA.
Let $(F,\om,\omb,\tau^\pm,\btau^\pm,J,\bJ)$ be an $N=(2,2)$ vertex operator superalgebra.
Let $H_l=\R J$ and $H_r =\R J_r$ equipped with bilinear forms $(J,J)_l = \frac{c}{3}$ and $(\bJ,\bJ)_r = \frac{\bc}{3}$.
For $\al \in H = H_l\oplus H_r$, set
\begin{align*}
F_{h,\h}^\al &= \{v \in F_{h,\h} \mid J(0)v =(\al,J)_l v \text{ and } \bJ(0)v =(\al,\bJ)_r v\}
\end{align*}
and 
\begin{align*}
d_B=Q+\bQ,\quad Q =\tau^+(0),\quad\bar{Q}=\bar{\tau}^+(0),
\end{align*}
which is a topological twist in Section \ref{sec_class_twist}, called the B-twist.
There is another topological twist, namely the $A$-twist \eqref{eq_AB_twist}.
However, since the $B$-twist of the mirror full vertex operator superalgebra $\widetilde F$
coincides with the $A$-twist of $F$, every statement and proof for the $B$-twist applies
formally to the $A$-twist as well, although the two are not isomorphic in general; see
Section \ref{sec_LG}. Accordingly, we work mainly with the $B$-twist and write $d_B$ simply as
$d$ unless otherwise stated.
%
%Note that an $N=(2,2)$ theory has the other topological twist, the A-twist \eqref{eq_AB_twist}. However, since the B-twist of the mirror $N=(2,2)$ full vertex operator superalgebra $\tilde{F}$ is the A-twist of $F$, 
%all statements and proofs for the $B$-twist is automatically applicable for the $A$-twist,
%while they are not isomorphic in general (see Section \ref{sec_LG}).
%Hence, we shall consider the B-twist and write $d_B$ simply as $d$ unless otherwise noted.
%$N=(2,2)$ theory には、up to automorphism でもう一つのtopological twistである A-twist \eqref{eq_AB_twist} がある。しかしミラーの$N=(2,2)$ full vertex operator superalgebra $\tilde{F}$ のB-twist は下の代数$F$のA-twist であるため、断りがない場合 B-twist を考えることにし、$d_B$を単に$d$と書くことにする。

%\begin{rem}
%\label{rem_B_twist}
%The superderivation $d = \tau^+(0)+\btau^+(0)$ is the B-twist in Section \ref{sec_}.
%\end{rem}

Since $Q$ and $\bQ$ are superderivations, 
\begin{align}
d(a(r,s)b) = (da)(r,s) b + (-1)^{|a|} a(r,s)db
\label{eq_super_derivation_d}
\end{align}
holds for any $a,b\in F$ and $r,s\in \R$. Hence, $\ker d$ is a subalgebra of $F$.
By $[Q,Q]_+ =[\bQ,\bQ]_+=[Q,\bQ]_+=0$, we have $d^2=0$.
Hence, $\im d$ is an ideal of $F$. Since for any $a,b \in \ker d$,
\begin{align}
\begin{split}
\frac{d}{dz}Y(a,\uz)b &= [L(-1),Y(a,\uz)]b = [[G_{-\ft}^+,G_{-\ft}^-]_+,Y(a,\uz)]b\\
&=[[d,G_{-\ft}^-]_+,Y(a,\uz)]b = [[d,Y(a,\uz)] ,G_{-\ft}^-]_+b +[d, [G_{-\ft}^-,Y(a,\uz)]_+]b\\
&= [Y(da,\uz) ,G_{-\ft}^-]_+b +d [G_{-\ft}^-,Y(a,\uz)]_+b\\
&=d [G_{-\ft}^-,Y(a,\uz)]_+b \in (\im d) ((z,\z,|z|^\R)),
\end{split}
\label{eq_ker_dz}
\end{align}
and similarly $\frac{d}{d\z}Y(a,\uz)b \in (\im d)((z,\z,|z|^\R))$,
the induced linear map $\bar{Y}(\bullet,\uz):H(F,d) \rightarrow \End H(F,d)[[z,\z,|z|^\R]]$ on
\begin{align*}
H(F,d) = \ker d  /\im d
\end{align*}
consists of only constant term, which is the $(-1)$-th product:
\begin{align*}
H(F)\otimes H(F) \rightarrow H(F),\quad a\otimes b \mapsto a(-1,-1)b.
\end{align*}
Since $Y(L(-1)a,\uz) = Y([d,G_{-\ft}^-]_+ a,\uz)= Y(dG_{-\ft}^- a,\uz) = [d,Y(G_{-\ft}^- a,\uz)]_+$,
similarly to the proof of Proposition \ref{prop_constant_chiral},
we have
\begin{align*}
\langle u, Y(a,\uz_1)Y(b,\uz_2)c,\quad\quad  \langle u, Y(Y(a,\uz_0)b,\uz_2)c \in \C
\end{align*}
for any $a,b,c\in \ker d$ and $u \in F^\vee$ with $u(d\bullet)=0$. Hence, the $(-1)$-th product on $H(F,d)$ is supercommutative and associative by (FV3).
%Note that since $Q: F_{h,\h}^\al \rightarrow F_{h+\ft,\h}^{\al+\frac{c}{3}J}$ and $\bQ: F_{h,\h} \rightarrow F_{h,\ft+\h}^{\al+\frac{\bc}{3}\bJ}$, we could only define $\Z_2$-grading on $H(F,d)$ (not $\R^2$-grading).
%This difficulty is solved by considering unitary full VOAs later.
Defining a $\Z_2$-grading on $H(F,d)$ by the fermion number $(-1)^F$. Then, we have:
%the linear maps $Q,\bQ$ do not change the grading. Hence, $\ker d/\im d$ has $\R^2$-graded.
%Hence, we have:
\begin{prop}\label{prop_full_ring}
$H(F,d)$ is a $\Z_2$-graded unital supercommutative algebra by
\begin{align*}
H(F,d)\otimes H(F,d) \rightarrow H(F,d),\quad a\otimes b \mapsto a(-1,-1)b
\end{align*}
and the unit $\va \in H(F,d)$.
\end{prop}

%The following lemma shows that the cohomology exists only for degree $(0,0)$ (see also \ref{lem_chiral_concentrate}).
%$L(0)-\frac{J(0)}{2}$と$L(0)-\frac{J(0)}{2}$によって$F$上の$\R^2$次数付けを定義すると、$Q,\bQ$は次数を変えないため、$\ker d/\im d$は$\R^2$次数付けを持つ。次の補題によりコホモロジーが次数$(0,0)$のみに存在していることが分かる(see also \ref{lem_chiral_concentrate})。
\begin{lem}\label{lem_full_concentrate}
Let $r,s \in \R$ and $v=\sum_{\al \in H} v_\al \in \bigoplus_{\al\in H}F_{r+\frac{(\al,J)_l}{2},s+\frac{(\al,\bJ)_r}{2}}^\al$ be a non-zero vector.
If $dv=0$ and $(r,s)\neq (0,0)$, then $v \in \im d$.
\end{lem}
\begin{proof}
Assume $dv =0$ and $r\neq 0$.
Then, $Q v_\al \in F_{r+\frac{(\al,J)}{2}+\ft,s+\frac{(\al,\bJ)}{2}}^{\al+\frac{3}{c}J}$
and $\bQ v_\al \in F_{r+\frac{(\al,J)}{2},s+\frac{(\al,\bJ)}{2}+\ft}^{\al+\frac{3}{\bc}\bJ}$.
Hence, $dv=0$ implies
$Q v_\al = \bQ v_{\al-\frac{3}{\bc}\bJ+\frac{3}{c}J}$ for any $\al\in H$.

Take $\al_0 \in H$ such that $(\al,J)$ is maximal in $\al \in H$ satisfying $v_\al \neq 0$.
Then, $Q v_{\al_0} =0$.
Since $G_{-\ft}^+ G_{\ft}^- v_{\al_0} = [G_{-\ft}^+, G_{\ft}^-]_+ v_{\al_0} = rv$.
Then,
$v' = v - \frac{1}{r} d G_\ft^- v_{\al_0}$ satisfies $dv'=0$.
By repeating this, we may assume that $Qv = \bQ v = 0$ with $v \in F_{r+\frac{(\al,J)_l}{2},s+\frac{(\al,\bJ)_r}{2}}^\al$ for some $\al\in H$.
Then,
\begin{align*}
d G_\ft^- v = (G_{-\ft}^+ + \bG_{-\ft}^+) G_\ft^- v =rv - G_\ft^- \bG_{-\ft}^+ v=rv.
\end{align*}
Thus, $v \in \im d$.
\end{proof}

\subsection{Inequalities from unitarity and supersymmetry}\label{sec_inequality}
In this section we will show some inequalities as consequences from the unitarity and supersymmetry. 
All the results (maybe except for Lemma \ref{lem_chiral_volume} and the second half of Lemma \ref{lem_primary_inequality}) in this section are well-known in physics \cite{LVW}, but for the reader's convenience, we will give all the proofs.

Let $F$ be a unitary $N=(2,2)$ full vertex operator superalgebra.
\begin{dfn}\label{def_full_primary}
Let $\al \in H=H_l\oplus H_r$ and $h,\h \in \R$.
A vector $v \in F_{h,\h}^\al$ is called a lowest weight vector of weight $\al$ if it satisfies
\begin{align*}
\begin{cases}
L_n v= \Ld_n v=0 & (n\geq 1),\\
J_n v= \bJ_n v=0 & (n\geq 1),\\
G_r^\pm v=\bG_r^\pm v= 0 & (r \geq \ft \text{ and }r \in \ft+\Z).
%L(0) v = hv, \quad \Ld(0) v = \h v\\
%J(0)v = (\al,J)_l v, \quad \bJ(0)v = (\al,\bJ)_rv.
\end{cases}
\end{align*}
For a lowest weight vector $v$, we consider the following four conditions:
\begin{align}
\begin{cases}
G_{-\ft}^+ v=\bG_{-\ft}^+ v=0,\quad\quad &(\text{c-c primary})\\
G_{-\ft}^- v=\bG_{-\ft}^+ v=0,\quad\quad &(\text{a-c primary})\\
G_{-\ft}^+ v=\bG_{-\ft}^^ v=0,\quad\quad &(\text{c-a primary})\\
G_{-\ft}^- v=\bG_{-\ft}^- v=0,\quad\quad &(\text{a-a primary}).
\end{cases}
\end{align}
A vector which satisfies the first condition is called a \textbf{chiral-chiral primary} (c-c primary) vector
and the second condition is called a \textbf{antichiral-chiral primary} (a-c primary) vector, and so on, in physics \cite{LVW}. Denote by $C_{h,\h}^\al(F)_{cc}$ (resp. $C_{h,\h}^\al(F)_{ac,ca,aa}$) be the subspace of $F$ consisting of c-c primary (resp. a-c, c-a, a-a primary) vectors of weight $\al$.
\end{dfn}
Since we mainly consider c-c primary vectors, we will hereafter write $C_{h,\h}^\al(F)_{cc}$ simply as $C_{h,\h}^\al(F)$ when there is no confusion.
%我々は主にc-c primary vector を考えるので、簡単のため混乱の恐れがない場合は以下、$C_{h,\h}^\al(F)_{cc}$を単に$C_{h,\h}^\al(F)$と書くことにする。
%
%Let $V$ be a unitary $N=2$ vertex operator superalgebra. 
%For any $\al \in H$ and $h\in\Z$, let $C_h^\al(V)$ be the subspace of $V$ consisting of vectors $v$ such that:
%\begin{align*}
%\begin{cases}
%L_n v=0 & (n\geq 1),\\
%J_n v=0 & (n\geq 1),\\
%G_r^\pm v=0 & (r \geq \ft \text{ and }r \in \ft+\Z),\\
%G_{-\ft}^+ v=0,\\
%L(0) v = hv,\\
%J(0)v = (\al,J)v,
%\end{cases}
%\end{align*}

%次の命題は物理において良く知られている標準的な結果である。

\begin{lem}
\label{lem_NS_inequality}
%Let $F$ be a unitary $N=(2,2)$ full vertex operator superalgebra. 
Let $\al \in H$
and $v \in F_{h,\h}^\al$ be a non-zero vector.
Then, the following properties hold:
\begin{enumerate}
\item
$h \geq \frac{|(\al,J)_l|}{2}$ and $\h \geq \frac{|(\al,\bJ)_r|}{2}$.
\item
$h = \frac{(\al,J)}{2}$ if and only if $G_{-\ft}^+ v =0$ and $G_\ft^- v =0$.
\item
$h = -\frac{(\al,J)}{2}$ if and only if $G_{-\ft}^- v =0$ and $G_\ft^+ v =0$.
\end{enumerate}
\end{lem}
\begin{proof}
By Proposition \ref{prop_unitary_adjoint}, if $G_\ft^- v=0$, then
\begin{align}
\begin{split}
0 &\leq  \langle G_{-\ft}^+ v, G_{-\ft}^+ v \rangle = \langle v,G_{\ft}^-G_{-\ft}^+ v \rangle = \langle v,[G_{\ft}^-,G_{-\ft}^+]_+ v \rangle \\
&= \langle v, (L(0)-\ft J(0)) v \rangle = (h-\ft(\al,J)) \langle v,v\rangle
\end{split}
\label{eq_ine_1}
\end{align}
and similarly, if $G_\ft^+ v=0$, then
\begin{align*}
0 \leq  \langle G_{-\ft}^- v, G_{-\ft}^- v \rangle = (h+\ft(\al,J)) \langle v,v\rangle.
\end{align*}
Suppose $G_\ft^- v \neq 0$.
Since $G_\ft^- G_\ft^- v = 0$ and $G_\ft^- v \in F_{h-\ft,\h}^{\al- \frac{3}{c}J}$, we have
$h-\ft \geq \frac{(\al-\frac{3}{c}J,J)}{2} = \frac{(\al,J)}{2}-\ft$.
Hence, $h \geq \frac{(\al,J)}{2}$ always holds.
%Assume $h= \frac{(\al,J)}{2}$.
%Then, $G_\ft^- v$ also satisfies the same equality.
%Thus, $G_{-\ft}^+ G_\ft^- v =0$.
%the assertion holds.
Suppose that $G_{-\ft}^+ v =0$ and $G_\ft^- v =0$. Then, by \eqref{eq_ine_1}, $h= \frac{(\al,J)}{2}$.
Conversely, assume $h= \frac{(\al,J)}{2}$.
Then,
\begin{align*}
\langle G_{\ft}^- v, G_{\ft}^- v \rangle = \langle v, G_{-\ft}^+ G_\ft^- v\rangle.
\end{align*}
Since $G_\ft^- v \in F_{h-\ft,\h}^{\al- \frac{3}{c}J}$ satisfies the equality and $(G_\ft^-)^2 v=0$,
$G_{-\ft}^+G_\ft^- v =0$ by \eqref{eq_ine_1}.
Hence, $G_{\ft}^- v=0$ and, by \eqref{eq_ine_1} again, $G_{-\ft}^+ v=0$.
\end{proof}

Hence, we have:
\begin{prop}\label{prop_NS_inequality}
Let $v \in F_{h,\h}^\al$ be a non-zero vector. Then, $v$ is c-c primary if and only if $h = \frac{(\al,J)_l}{2}$ and $\h=\frac{(\al,\bJ)_r}{2}$.
\end{prop}
\begin{proof}
If $v$ is c-c primary, then $h = \frac{(\al,J)_l}{2}$ and $\h=\frac{(\al,\bJ)_r}{2}$ by Lemma \ref{lem_NS_inequality}.
Assume that $h = \frac{(\al,J)_l}{2}$ and $\h=\frac{(\al,\bJ)_r}{2}$.
By Lemma \ref{lem_NS_inequality} again, $G_{-\ft}^+ v = 0$.
Since
\begin{align*}
L(n) v \in F_{h-n,\h}^\al \quad\quad J(n)v \in F_{h-n,\h}^\al,
\end{align*}
by the inequality (1) in Lemma \ref{lem_NS_inequality}, $L(n)v = J(n)v =0$ for any $n \geq 1$.
Similarly,
\begin{align*}
G_{r}^+ v \in F_{h-r,\h}^{\al + \frac{3}{c}J},\quad,\quad G_{s}^- v \in F_{h-s,\h}^{\al - \frac{3}{c}J},
\end{align*}
by the inequality, $G_r^+ v =G_s^-v =0$ for any $r \geq \ft$ and $s \geq \frac{3}{2}$.
Hence, the assertion holds.
\end{proof}

Hence, the inclusion map $C(F)_{\frac{(\al,J)_l}{2},\frac{(\al,\bJ)_r}{2}}^{\al} \rightarrow F_{\frac{(\al,J)_l}{2},\frac{(\al,\bJ)_r}{2}}^{\al}$ gives an isomorphism of vector spaces.
Moreover, for any $v \in F_{\frac{(\al,J)_l}{2},\frac{(\al,\bJ)_r}{2}}^{\al}$, by $d v=(G_{-\ft}^+ + \bG_{-\ft}^+) v =0$,
we can regard $v$ as an element of $H(F,d)$.
If $v =dw$ for some $w \in \bigoplus_{k\in \Z} F_{\frac{(\al,J)_l}{2}-\frac{k+1}{2},\frac{(\al,\bJ)_r}{2}+\frac{k}{2}}^{\al-(k+1)\frac{c}{3}J +k\frac{\bc}{3}\bJ}$, then by Lemma \ref{lem_NS_inequality}, $d w =0$.
Hence, by Lemma \ref{lem_full_concentrate}, we have:
\begin{prop}
\label{prop_cohomology_primary}
Let $F$ be an $N=(2,2)$ unitary full vertex operator superalgebra.
Then, the cohomology ring $H(F,d)$
is isomorphic to $\bigoplus_{\al \in H} C(F)_{\frac{(\al,J)_l}{2},\frac{(\al,\bJ)_r}{2}}^{\al}$ as a vector space by the inclusion.
%the $\R^2 \times H$-grading on $F$ induces an $\R^2\times H$-grading on $H(F,d)$ and all the maps
%\begin{align*}
%C(F)_{\frac{(\al,J)_l}{2},\frac{(\al,\bJ)_r}{2}}^{\al} \rightarrow F_{\frac{(\al,J)_l}{2},\frac{(\al,\bJ)_r}{2}}^{\al} \rightarrow H(F,d)_{\frac{(\al,J)_l}{2},\frac{(\al,\bJ)_r}{2}}^{\al}
%\end{align*}
%are isomorphism of vector spaces. 
%In particular,
%\begin{align*}
%H(F,d) \cong \bigoplus_{\al \in H} F_{\frac{(\al,J)_l}{2},\frac{(\al,\bJ)_r}{2}}^{\al}
%\end{align*}
%as vector spaces.
\end{prop}

Proposition \ref{prop_cohomology_primary} allows us to put the algebraic structure of the cohomology ring $H(F,d)$ onto $C(F) \subset F$, which is often more convenient.
Denote the isomorphism in Proposition \ref{prop_cohomology_primary} by
\begin{align}
R: C(F) = \bigoplus_{\al \in H}C(F)_{\frac{(\al,J)}{2},\frac{(\al,\bJ)}{2}}^{\al} \rightarrow H(F,d).
\label{eq_R_isom}
\end{align}

\begin{lem}\label{lem_primary_product}
Let $v \in C(F)_{\frac{(\al,J)}{2}, \frac{(\al,\bJ)}{2}}^{\al}$ and $w\in C(V)_{\frac{(\be,J)}{2},\frac{(\be,\bJ)}{2}}^{\be}$.
Then, $v(r,s)w =0$ for any $r>-1$ or $s >-1$ and
$v(-1,-1)w \in C(F)_{\frac{(\al+\be,J)}{2},\frac{(\al+\be,\bJ)}{2}}^{\al+\be}$.
\end{lem}
\begin{proof}
Since $v(r,s)w \in F_{\frac{(\al+\be,J)}{2}-r-1, \frac{(\al+\be,\bJ)}{2}-s-1}^{\al+\be}$ for any $r,s \in \R$, the assertion follows from Lemma \ref{lem_NS_inequality} and Proposition \ref{prop_NS_inequality}.
\end{proof}
Hence, we can define a product $C(F) \otimes C(F) \rightarrow C(F)$ by the $(-1,-1)$-th product,
which is denoted by $\cdot$.
An important consequence of this construction is that $C(F)$ is not only a $\Z_2$-graded algebra, but also an $\R^2$-graded algebra.
%この構成の重要な帰結は$C(F)$が$\Z_2$-graded algebra ではなく、$\R^2$-graded algebra になる点にある。
Then, by construction, we have:
\begin{prop}\label{prop_isom_ring}
The linear isomorphism $R: C(F) \rightarrow H(F,d)$ \eqref{eq_R_isom} is an isomorphism of $\Z_2$-graded supercommutative algebras. Moreover, they have the $\R^2$-graded algebra structures by the action of $J(0)$ and $\bJ(0)$.
\end{prop}

%The following lemma is very important for the proof of the non-degeneracy of the bilinear form on $H(V,Q)$.
The following lemma is important:
\begin{lem}\label{lem_primary_inequality}
Let $v \in C(F)_{\frac{(\al,J)}{2},\frac{(\al,\bJ)}{2}}^{\al}$ be a non-zero vector.
Then, the following properties hold:
\begin{enumerate}
\item
$\frac{c}{3} \geq (\al,J)$ and $\frac{\bc}{3} \geq (\al,\bJ)$;
\item
If $\frac{c}{3} = (\al,J)$, then $L(-1)v = J(-1)v$;
\item
If $\frac{\bc}{3} = (\al,\bJ)$, then $\Ld(-1)v = \bJ(-1)v$.
\end{enumerate}
\end{lem}
\begin{proof}
%$h = \frac{(\al,J)}{2}$ follows from Lemma \ref{lem_NS_inequality}.
By Proposition \ref{prop_unitary_adjoint}, we have
\begin{align}
\begin{split}
0 &\leq  \langle  G_{-\frac{3}{2}}^+ v, G_{-\frac{3}{2}}^+ v \rangle =  \langle v, G_{\frac{3}{2}}^- G_{-\frac{3}{2}}^+ v \rangle
=\langle v, [G_{-\frac{3}{2}}^+,G_{\frac{3}{2}}^-]_+ v \rangle\\
&=\langle v, (L(0)-\frac{3}{2} J(0) + \frac{c}{3}) v \rangle = (\frac{c}{3}-(\al,J))\langle v, v \rangle.
\end{split}
\label{eq_top_primary}
\end{align}
Assume $\frac{c}{3}=(J,\al)$. Then, by \eqref{eq_top_primary}, $G_{-\frac{3}{2}}^+ v=0$. Hence, we have
\begin{align*}
0&= G_{\ft}^- G_{-\frac{3}{2}}^+ v = [G_{\ft}^-, G_{-\frac{3}{2}}^+]_+ v =(L(-1) - J(-1))v.
\end{align*}
\end{proof}
The property $L(-1)v = J(-1)v$ in the above lemma is the condition used to define a notion of a {\bf lattice-like vector} in \cite{M7} and plays an important role in this paper. By Lemma \ref{lem_primary_inequality},
the possible degrees of the cohomology $H(F,d)= \bigoplus_{\al \in H}H^\al(F,d)$ are bounded by $\al = J+\bJ$. 
\begin{dfn}\label{def_full_vol}
%Set
%\begin{align*}
%H^{\text{top}}(F,d) = H^{J,\bJ}(F,d) = F_{\frac{c}{6},\frac{\bc}{6}}^{J,\bJ}.
%\end{align*}
We call a non-zero vector $\ep \in C^{J+\bJ}(F)_{cc} =F_{\frac{c}{6},\frac{\bc}{6}}^{J+\bJ}$
(resp. $\ep \in C^{J-\bJ}(F)_{ca} =F_{\frac{c}{6},\frac{\bc}{6}}^{J-\bJ}$)
 a \textbf{B-volume form} (resp. a \textbf{A-volume form}) of $F$.
\end{dfn}
Note that a $B$-volume form of $\tilde{F}$ is an $A$-volume form of $F$. 
We later see that the $B$-volume form corresponds to
\begin{itemize}
\item
periodicity of spectral flow;
\item
a Hodge star operation on $H^\bullet(F,d_B)$;
\item
a Frobenius algebra structure on $H(F,d_B)$.
\end{itemize}
%We later see that the condition for the existence of a volume form is equivalent to the condition for the spectral flow  as a (full) vertex operator algebra to be well-defined, and plays an important role in this paper.

\begin{lem}
\label{lem_chiral_volume}
Let $v\in F_{h,0}$. Then, $\Ld(-1)v=0$.
\end{lem}
\begin{proof}
By Lemma \ref{lem_NS_inequality}, $\Ld(1)v \in F_{h,-1}$ implies $\Ld(1)v=0$.
Hence,
\begin{align*}
\langle \Ld(-1)v, \Ld(-1)v \rangle =\langle v,  \Ld(1) \Ld(-1) v \rangle=\langle v, 2\Ld(0)v \rangle =0,
\end{align*}
which implies $\Ld(-1)v=0$.
\end{proof}

Similarly to Proposition \ref{prop_cohomology_primary}, we have:
\begin{prop}
\label{prop_all_top_twist}
Let $F$ be an $N=(2,2)$ unitary full vertex operator superalgebra. Then, the inclusion map induces isomorphisms of $\R^2$-graded supercommutative algebras:
\begin{align*}
C(F)_{cc} &\cong H(F,G_{-\ft}^+ + \bG_{-\ft}^+)\\
C(F)_{ca} &\cong H(F,G_{-\ft}^+ + \bG_{-\ft}^-)\\
C(F)_{ac} &\cong H(F,G_{-\ft}^- + \bG_{-\ft}^+)\\
C(F)_{aa} &\cong H(F,G_{-\ft}^- + \bG_{-\ft}^-).
\end{align*}
Moreover, $\phi$ induces anti-linear $\C$-algebra isomorphisms $C(F)_{cc} \cong C(F)_{aa}$ and $C(F)_{ca} \cong C(F)_{ac}$.
\end{prop}

\subsection{Volume form, Poincar\'e duality, and Frobenius algebra}
\label{sec_twist_star}

\begin{comment}
Section \ref{}でより詳しく議論をするが、$F$が複素$d$次元のCalab-Yau manifold $X$上の supersymmetric sigma model $F_X$の場合、
\begin{align}
H^{\frac{p}{d}J,\frac{q}{d}\bJ}(F_X,d_B) \cong H^{d-q}(X,\Om^p)
\label{eq_conj_cohomology}
\end{align}
が成立することが物理で予想されている\cite{}。このとき、$F_X$の中心電荷は$(c,\bc)=(3d,3d)$である。
Hodge 分解から
\begin{align*}
H_{\mathrm{deRham}}^{n}(X,\C) = \bigoplus_{p+q=n} H^q(X,\Om^p)
\end{align*}
となることにより volume form $\ep \in H^{J,\bJ}(F_X)$ は $H_{\mathrm{deRham}}^{2d}(X,\C)$とみなせる。
\end{comment}

\begin{dfn}\label{def_hodge_number}
For any $p,q \geq 0$, set
\begin{align*}
H^{p,q}(F,d_B) =H^{\frac{3p}{c}J+\frac{3q}{\bc}\bJ}(F,d_B)\quad&\text{ and }\quad H^{p,q}(F,d_A) =H^{\frac{3p}{c}J-\frac{3q}{\bc}\bJ}(F,d_A)
\\
h_{p,q}^B(F) = \dim H^{p,q}(F,d_B)\quad&\text{ and }\quad h_{p,q}^A(F) = \dim H^{p,q}(F,d_A)
\end{align*}
which we call \textbf{Hodge numbers} of $N=(2,2)$ SCFT, and for $T=A,B$
\begin{align*}
\chi_y^T(F) = \sum_{p,q \in \R} (-1)^{p-q} h_{p,q}^T(F) y^p
\end{align*}
which we call {\bf Hirzebruch genera} of $F$.
Note that $h_{p,q}^T(F)=0$
unless
\begin{align*}
p-q \in \Z,\quad 0 \leq p \leq \frac{c}{3},\quad 0 \leq q \leq \frac{\bc}{3},
\end{align*}
and $h_{p,q}^T(F)$ can be nonzero even when $p,q$ are not integers (see Section \ref{sec_example_K3} and \ref{sec_LG}).
\end{dfn}
\begin{rem}\label{rem_h00}
Note that $h_{0,0}^B=1=h_{0,0}^A$ since $C^{0,0}(F) \subset \ker L(-1) \cap \ker \Ld(-1) =F_{0,0}=\C\va$ by Lemma \ref{lem_chiral_volume}.
\end{rem}

We assume that 
\begin{itemize}
\item
$h_{\frac{c}{3},\frac{\bc}{3}}^B \neq 0$, that is, there is a B-volume form $\ep \in F_{\frac{c}{6},\frac{c}{6}}^{J+\bJ}$, and normalize it by $\langle \ep,\ep \rangle =1$.
\end{itemize}

%\begin{rem}
%\label{rem_non_sigma_model}
%全ての unitary $N=(2,2)$ full vertex operator superalgebra が、sigma 模型として構成できるわけでは全くない。
%シグマ模型の場合、$p,q \in \Z$であり $c \in 3\Z$ が成り立つはずであるが、unitary $N=(2,2)$ minimal 模型
%(or Landau-Ginzburg 模型)はこれを満たさない。またシグマ模型の$N=(2,2)$超対称共形場理論としての orbifold は必ずしもシグマ模型ではない (三章の具体例を参照)。
%しかしこうした模型であっても、volume form の存在するという仮定は成り立つはずである。というのは volume form の存在が spectral flow の well-defined性と同値であるからである (定理\ref{thm_unitary_period}).
%\end{rem}

Define a sesquilinear form $\langle-,-\rangle :C(F)\otimes C(F) \rightarrow \C$ by the restriction of the sesquilinear form on $F$, which is positive definite, and thus, non-degenerate.
Define a linear map $\ep: C(F) \rightarrow \C$ by
\begin{align*}
\ep(a) = \langle \ep, a\rangle
\end{align*}
and a blinear form $(-,-)_\ep: C(F)\otimes C(F) \rightarrow \C$ by
\begin{align*}
(a,b)_\ep = \ep(a \cdot b)
\end{align*}
for $a,b \in C(F)$, and 
an anti-linear map $*:F \rightarrow F$ by
\begin{align}
*a = (-1)^{s(\al)+2s(\al)^2}  \phi(a)((\al,J)_l-1,(\al,\bJ)_r-1)\ep,
\label{eq_full_star_def}
\end{align}
where $s(\al)=\frac{(\al,J)_l}{2}-\frac{(\al,\bJ)_r}{2} \in \ft\Z$.
We will prove the following theorem in Section \ref{sec_Hodge_Serre}:
\begin{thm}\label{thm_full_frob}
Let $F$ be an $N=(2,2)$ unitary full vertex operator superalgebra equipped with a volume form $\ep \in C^{J+\bJ}(F)$.
Then, the following properties hold:
\begin{enumerate}
\item
$(a,b)_\ep=0$ for any $a\in C^\al(F)$ and $b\in C^\be(F)$ unless $\al +\be =J+\bJ$;
\item
$(a,b)_\ep =(-1)^{|a||b|}(b,a)_\ep$ for any $a,b \in C(F)$;
\item
$(a,b\cdot c)_\ep =(a\cdot b,c)_\ep$ for any $a,b,c \in C(F)$;
\item
$*$ send $C^\al(F)$ onto $C^{J+\bJ-\al}(F)$;
\item
For any $a\in C^\al(F)$,
\begin{align*}
**a = (-1)^{|a|+(\frac{c}{6}-\frac{\bc}{6})|a|}a.
\end{align*}
In particular, if $c-\bc \in 12\Z$, then $**a=(-1)^{|a|} a$;
%\item
%$*(a\cdot b) = *b \cdot *a$ for any $a,b\in C(F)$.
%\begin{itemize}
%\item
%これは成り立たない！！
%基本ほぼ可換なので $\phi$ を使ったらこういうの定義できるかも。
%\item
%$g:a \mapsto \exp(\frac{\pi i}{2}(J,\al)^2) a$ とすると $g(ab)=g(b)g(a)$を満たす。
%\end{itemize}
\item
$(a,b)_\ep=\langle *a , b \rangle$ for any $a,b \in C(F)$;
\item
$C^0(F)=\C\va$ and $C^{J+\bJ}(F)=\C \ep$.
\end{enumerate}
In particular, $(C(F),\cdot,\va,\ep)$ is an $\R^2$-graded Frobenius superalgebra in $\sVect$
and defines a 2d topological field theory.
Moreover, $\dim C^\al(F) = \dim C^{J+\bJ-\al}(F)$ for any $\al \in H$.
%the Hodge duality $h_{p,q} = h_{\frac{c}{3}-p,\frac{\bc}{3}-q}$ and $h_{\frac{c}{3},\frac{\bc}{3}} =1$ hold.
\end{thm}

Assume $c$ or $\bc$ is non-zero, then $\ep(1) =0$. Hence, by Lemma \ref{lem_non_zero_star} and Corollary \ref{cor_c_zero}, we have:
\begin{cor}\label{cor_non_unitary}
If $F$ is not isomorphic to the trivial full VOA $\C$, then the $2d$ topological field theory $C(F)$ could not be extended to a unitary  $2d$ topological field theory.
\end{cor}

\subsection{A- and B-twists, mirror symmetry and T-duality}\label{sec_twist_serre}

In physics, there is a well-known duality between the type IIA and the type IIB superstring theories, called a \textbf{T-duality} \cite{Polc2}, which is related with a geometrically very nontrivial phenomenon (see Introduction and Remark \ref{rem_mirror}).
In this section, we will mathematically study this duality as a duality between the A-twist and the B-twist.

%この章ではA-twist と B-twist の比較を行う。
%物理では type IIA と type IIB の superstring theory の間の duality (T-duality) として知られ、幾何学的にきわめて非自明な現象と結びついている (see Introduction)。

Note that in general $H^\bullet(F,d_A)$ and $H^\bullet(F,d_B)$ are completely different (see Introduction and Section \ref{sec_LG}).
%For example, for a unitary supersymmetric minimal model of type A $M_{p}^{N=2}$(central charge $(\frac{3p}{p+2},\frac{3p}{p+2})$)
%\begin{align*}
%H^\bullet(M_{p}^{N=2},d_B)\cong \C[z] / z^{p+1},\quad\quad H^\bullet(M_{p}^{N=2},d_A)=\C
%\end{align*}
%(see Section \ref{sec_LG}).
However, if $F$ is the supersymmetric sigma model of a Calabi-Yau $d$-fold, physics predicts that there is an isomorphism between $H^{p,q}(F,d_A)$ and $H^{d-p,q}(F,d_B)$ as $\R^2$-graded vector spaces (see \cite{Hori} and Section \ref{sec_conj}).
In this section, we formulate this claim mathematically, and give its proof in Section \ref{sec_Hodge_Serre}.

%一般に $H^\bullet(F,d_A)$と$H^\bullet(F,d_B)$は全く異なる。
%たとえばA型のユニタリ超対称ミニマル模型 $M_{p}^{N=2}$(中心電荷$(\frac{3p}{p+2},\frac{3p}{p+2})$)の場合
%\begin{align*}
%H^\bullet(M_{p}^{N=2},d_B)\cong \C[z] / z^{p+1},\quad\quad H^\bullet(M_{p}^{N=2},d_A)=\C
%\end{align*}
%である (see Section \ref{sec_LG}).
%しかし、$F$が Calabi-Yau manifold の 超対称sigma模型の場合、$H^\bullet(F,d_A)$と$H^\bullet(F,d_B)$の間には次数付きベクトル空間としての同型があることが物理で予想されている \cite{}。
%我々はこの主張を数学的に定式化する。またその証明はSection \ref{sec_Hodge_Serre} で行う。
%

%\eqref{eq_conj_cohomology} より、$X$が Calabi-Yau manifold の場合、
%\begin{align*}
%H^{J,0}(F_X) \cong H^0(X,\Om_X^d) = H^0(X,K_X) = \C
%\end{align*}
%であるから、$\dim H^{J,0}(F_X) = \dim H^{0,J}(F_X) \neq 0$ であるはずである。
%In this section, we assume that 
%\begin{itemize}
%\item
%$h_{\frac{c}{3},\frac{\bc}{3}} \neq 0$ and $h_{\frac{c}{3},0} \neq 0$.
%\end{itemize}

%次の定理は一般に $H^{J,0}(F) \neq 0$ならば、A-twist と B-twist が次数付きベクトル空間として(次数を反転させると)同型であることを主張する。
\begin{thm}\label{thm_chiral_star}
Assume that $H^{J,0}(F,d_B) \neq 0$.
Let $\eta \in F_{\frac{c}{6},0}^{J}$ such that $\langle \eta,\eta \rangle=1$. Set $d=\frac{c}{3}$.
%and $F$ has a mirror automorphism $\mu:F \rightarrow F$.
Then, the following conditions hold:
\begin{enumerate}
\item
The eigenvalues of $J(0)$ are integers and $c \in {3}\Z$. In particular, $H(F,d_B)$ is a $\Z^2$-graded algebra;
\item
The linear map $*_c: C^{\al}(F)_{cc} \rightarrow C^{J-\al}(F)_{ca}$ defined by
\begin{align*}
*_c a = \phi(a)((\al,J)_l-1,-1)\eta
\end{align*}
is a linear isomorphism. In particular, $h_{p,q}^B(F) = h_{d-p,q}^A(F)$ for any $p,q\in\Z$.
%\item
%If $\mu^2=\id_F$, then for any $a\in C^\al(F)$
%\begin{align*}
%*_c*_c (a) = (-1)^{|a|+??}a.
%\end{align*}
\item
$h_{p,0}^B(F) = h_{d-p,0}^B(F)$ for any $p \in\Z$.
\end{enumerate}
%In particular, $H^{p,q}(F,d_B) \cong H^{d-p,q}(F,d_A)$.
%In particular, the Hodge numbers satisfy $h_{p,q} = h_{\frac{c}{3}-p,q}$.
\end{thm}

\begin{cor}
\label{cor_mirror}
Let $F$ be an $N=(2,2)$ unitary  full VOA and $\tilde{F}$ the mirror full VOA (see Definition \ref{def_mirror}).
Assume that $H^{J+\bJ}(F,d_B)\neq 0$ and $H^{J}(F,d_B)\neq 0$. Then, the Hodge numbers of $F$ are equal to the flipped Hodge numbers of $\tilde{F}$,
\begin{align*}
h_{p,q}^B(F) = h_{d-p,q}^B(\tilde{F}).
\end{align*}
\end{cor}
\begin{proof}
%Let $\tilde{F}$ be the mirror algebra of the unitary $N=(2,2)$ full vertex operator superalgebra $F$ .
By definition of the mirror algebra and Theorem \ref{thm_chiral_star},
\begin{align*}
H^{p,q}(\tilde{F},d_B) = H^{p,q}(F,d_A) \cong H^{d-p,q}(F,d_B),
\end{align*}
the assertion holds.
\end{proof}

\begin{cor}\label{cor_3fold}
Let $F$ be a unitary $N=(2,2)$ full VOA of central charge $(9,9)$.
Assume that both $H^{3,3}(F,d_B)$ and $H^{3,0}(F,d_B)$ are non-zero
and $H^{1,0}(F,d_B) =H^{0,1}(F,d_B)=0$. Then,
its Hodge numbers $h_{p,q}=h_{p,q}^B(F)$ are determined only by $h_{1,1}$ and $h_{1,2}$, that is,
\begin{align}
\begin{array}{ccccccc}
&&& h_{0,0} \\
&& h_{1,0} && h_{0,1}\\
&h_{2,0} & & h_{1,1} & & h_{0,2}\\
h_{3,0}&& h_{2,1} & & h_{1,2} &&h_{0,3}\\
&h_{3,1}&& h_{2,2}&&h_{1,3}\\
&&h_{3,2}&&h_{2,3}\\
&&&h_{3,3}
\end{array}
\quad =\quad
\begin{array}{ccccccc}
&&& 1 \\
&& 0 && 0\\
&0 & & h_{1,1} & & 0\\
1&& h_{1,2} & & h_{1,2} &&1\\
&0&& h_{1,1}&&0\\
&&0&&0\\
&&&1.
\end{array}
\end{align}
In particular, $h_{1,1}(F) = h_{1,2}(\tilde{F})$ and $h_{2,1}(F) = h_{1,1}(\tilde{F})$ for the mirror full VOA $\tilde{F}$.
\end{cor}
\begin{proof}
By the Hodge duality, Theorem \ref{thm_full_frob}, $h_{0,3} =h_{3,0}=1$ and $h_{1,1}=h_{2,2}$, $h_{1,2}=h_{2,1}$.
By the T-duality, Theorem \ref{thm_chiral_star}, and the Hodge duality, $0=h_{1,0}=h_{2,0}=h_{2,3}=h_{1,3}$ and $0=h_{0,1}=h_{0,2} = h_{3,2}=h_{3,1}$.
Hence, the assertion holds.
\end{proof}

\begin{rem}\label{rem_mirror}
In superstring theory, Calabi-Yau manifolds $X$ and $\widetilde X$ are mirror if the
associated supersymmetric sigma models $F_X$ and $F_{\widetilde X}$ satisfy
$\widehat F_X \cong F_{\widetilde X}$ as unitary $N=(2,2)$ full VOAs. Physics further
predicts that the Hodge numbers of $F_X$ coincide with the Hodge numbers of $X$ in the
sense of K\"ahler geometry; see Section \ref{sec_conj}. Hence Corollary \ref{cor_mirror} implies formally that,
when $X$ and $\widetilde X$ are mirror in the sigma-model sense, their Hodge numbers
satisfy the usual mirror relation. This is what is called Hodge-theoretic, or classical,
mirror symmetry in mathematics; see \cite{Kont,Yau}.
%In superstring theory,
%Calabi-Yau manifolds $X$ and $\tilde{X}$ are mirror if the associated supersymmetric sigma models $F_X$ and $F_{\tilde{X}}$ satisfy $\hat{F_X} \cong F_{\hat{X}}$ as unitary $N=(2,2)$ full VOAs.
%It is conjectured in physics that the Hodge numbers of $F_X$ and the Hodge number of $X$ in the sense of K\"{a}hler geometry coincide (see Section \ref{sec_conj}). Hence, from Corollary \ref{cor_mirror}, when $X$ and $\tilde{X}$ are mirrors in the sense of the sigma model, the Hodge numbers of $X$ and $\tilde{X}$ are ``mirror'', which is called a \textbf{Hodge mirror} or a classical mirror, in the mathematical study of mirror symmetry \cite{Kont,Yau}.
%物理的には Calabi-Yau manifolds $X$ and $\tilde{X}$が(シグマ模型の意味で)ミラーであるとは、$X$と$\tilde{X}$の超対称シグマ模型 $F_X$ and $F_{\tilde{X}}$がミラーの$N=(2,2)$ full VOA になっていることを意味する。
%後で見るように $F_X$のhodge数と, $X$の Kahler 幾何の意味での hodge 数は一致すると良そうされている。そこでCorollary \ref{cor_mirror}から、$X$と$\tilde{X}$がシグマ模型の意味でミラーのとき、$X$と$\tilde{X}$の hodge 数は入れ替わる (古典的ミラー)。
\end{rem}

\section{Spectral flow twist of full vertex operator superalgebra}\label{sec_sec_flow}

In this section, we construct a generalized full vertex operator superalgebra from an $N=(2,2)$ full VOA based on \cite{M1}.
It consists of a family of vector spaces $F(\la,\mu)$ with $(\la,\mu)\in \mathbb{R}^2$, on which the twisted $N=(2,2)$ superconformal algebra 
$\g_{\la}^{N=2} \oplus g_{\mu}^{N=2}$ acts.
This construction serves as a mathematical formulation of the spectral flow in physics.

In \cite{M1}, an equivalence of categories between
full $\mathcal{H}$-vertex algebras (full VOAs with Heisenberg full vertex subalgebras) and generalized full vertex algebras was shown. This result is very important in our construction of the spectral flow, and we will review it in Section \ref{sec_str_h}.
Section \ref{sec_const} examines the structure of a generalized full vertex algebra.
In Section \ref{sec_flow}, the spectral flow is formulated. Using the results in Section \ref{sec_const}, we prove the equivalence between the spectral flow being periodic and the existence of the volume form. Section \ref{sec_Lie_flow} compares the spectral flows in the sense of Lie algebra (Definition \ref{def_spectral_Lie}) and in the sense of full vertex algebra in Section \ref{sec_flow}, and shows that they coincide on the subspace of primary vectors.
This implies the dualities in Section \ref{sec_Hodge_Serre}.
In Section \ref{sec_genus}, the Witten genus of an $N=(2,2)$ full VOA is defined and its properties are investigated.

%この章では\cite{M1}の結果を応用することで、$N=(2,2)$ full VOA から canonical に generalized full vertex operator superalgebra が構成できることを見る。これは$(\la,\mu) \in \mathbb{R}^2$で parametrize された vector space $F(\la,\mu)$の族であり、それぞれに $N=(2,2)$ superconformal algebra $\g_{\la}^{N=2}\oplus \g_{\mu}^{N=2}$が作用する。
%これは物理における spectral flow の数学的定式化になっている。

%\cite{M1}では、Heisenberg full vertex algebra を部分代数に持つ full VOA の圏と、generalized full vertex algebra の圏の間の圏同値が示された。Section \ref{sec_str_h}ではこの結果をレビューする。
%Section \ref{sec_const}では generalized full vertex algebra の構造を調べる。
%Section \ref{sec_flow}では, spectral flow を定式化し構成する。また Volume form とSection \ref{sec_const}の結果を用いて 、spectral flow が periodic であることとVolume form の存在の同値性を証明する。
%Section \ref{sec_Lie_flow}では、Lie algebra の意味での flow (Definition \ref{def_spectral_Lie})と vertex algebra の意味での flow を比較し、highest weight vector の上では両者が一致することを示す。この事実から spectral flow は c-c primary vector を a-a primary vector に流すことが分かり、Section \ref{sec_Hodge_Serre}でHodge-Serre duality が示される.
%Section \ref{sec_genus}では、$N=(2,2)$ full VOA の elliptic genus を定義しその性質を調べる。
%

\subsection{Structure of full $\cH$-vertex operator algebra}\label{sec_str_h}
%full Heisenberg vertex algebra を部分代数に持つfull vertex algebra をfull H-vertex algebraという。
%$N=(2,2)$ ful VOsA は$J,\bJ$によって自然な full $H$-vertex algebras 構造を持つ。この章ではfull H-vertex algebra の構造についての結果を文献3から振り返る。まず初めにカイラルの場合に説明し、手短にfullの場合を振り返る。

A full vertex operator algebra with a subalgebra which is isomorphic to a Heisenberg full vertex algebra is called a full $\cH$-vertex operator algebra.
The $N=(2,2)$ full VOA has a natural full $\cH$-vertex algebra structure by the affine Heisenberg full vertex algebra generated by $J$ and $\bJ$.

In \cite{M1}, we showed that the category of full $\cH$-vertex operator algebras is equivalent to the category of generalized full vertex operator algebras with charge structure. 
%これはchiralの場合のLiによる結果のfullへの一般化になっている\cite{Li20}。\cite{M1}ではこれを用いて、full VOA の変形族を構成した。この論文ではその応用として、$N=(2,2)$ full VOA の spectral flow を構成する。この章では\cite{M1}を復習する。
This is a generalization of Li's result in the chiral case \cite{Li20}. In \cite{M1} we used this equivalence to construct deformation families of full VOAs. In this paper, we use this for the construction of the spectral flow. In this section, we review \cite{M1}.

Let $H_l,H_r$ be real finite-dimensional vector spaces equipped with
non-degenerate symmetric bilinear forms $(-,-)_l:H_l \times H_l \rightarrow \R$ and $(-,-)_l:H_r \times H_r \rightarrow \R$.
Let $M_{H_l}(0)$ and $M_{H_r}(0)$ be affine Heisenberg vertex algebras
associated with $(H_l,(-,-)_l)$ and $(H_r,(-,-)_r)$.
Set $H=H_l \oplus H_r$ and let $p, \overline{p}:H \rightarrow H$
 be the projections onto $H_l$ and $H_r$
and 
$$
M_{H,p}(0)=M_{H_l}(0)\otimes \overline{M_{H_r}(0)},
$$
the tensor product of the vertex algebra $M_{H_l}(0)$
and the conjugate vertex algebra $\overline{M_{H_r}(0)}$, whose formal variable of the vertex operator is $\z$.
Hence, $M_{H,p}(0)$ is a full vertex operator algebra (see \cite[Proposition 3.10 and Corollary 3.17]{M1}).

\begin{dfn}\label{def_H_VA}
A {\bf full $\cH$-vertex operator superalgebra} is a full vertex operator superalgebra $F$ together with a full vertex superalgebra homomorphism $i:M_{H,p}(0) \rightarrow F$ such that:
\begin{enumerate}
\item[FH1)]
For any $h_l \in H_l$ and $h_r \in H_r$, $L(n)i(h_l) = 0=\Ld(n)i(h_r)$ and $L(0)i(h_l) =i(h_l)$,
$\Ld(0)i(h_r) =i(h_r)$ for any $n \geq 1$;
\item[FH2)]
For any $h_l \in H_l$ and $h_r\in H_r$, $i(h_l)(0,-1)$ and $i(h_r)(-1,0)$ semisimply act on $F$ with real eigenvalues.
\end{enumerate}
\end{dfn}
Since $i$ sends holomorphic vectors (resp. anti-holomorphic vectors) onto holomorphic vectors (resp. anti-holomorphic vectors), $L(n)i(h_r) =0$ and $\Ld(n)i(h_l)=0$ automatically holds for all $n \geq -1$.
Hereafter, we regard $M_{H,p}(0)$ as a subalgebra of $F$ and omit writing $i$.

Let $F$ be a full $\cH$-vertex operator superalgebra.
For $\al \in H$, let $\Omega_{F}^\al$ be the set of all vectors $a \in F$ satisfying the following conditions:
For any $h_l\in H_l$ and $h_r\in H_r$,
\begin{itemize}
\item
$h_l(n,-1)a=0$, $h_r(-1,n)a=0$ for any $n \geq 1$;
\item
$h_l(0,-1)a=(h_l,p\al)_l a$, $h_r(-1,0)a=(h_r,\p\al)_r a$.
\end{itemize}
Set $$\Omega_{F}=\bigoplus_{\al \in H} \Omega_{F}^\al$$
and
$$(\Om_{F})_{h-\frac{(p\al,p\al)_l}{2},\h-\frac{(\p\al,\p\al)_r}{2}}^{\al}
= F_{h,\h} \cap \Om_{F}^\al
$$ for $h,\h\in \R$ and $\al \in H$.

For $\al \in H$, set
\begin{align}
\begin{split}
E^-(\al,\uz)&= \exp\Bigl(\sum_{n \geq 1} \frac{p\al(-n,-1)}{n}z^n+\frac{\p\al(-1,-n)}{n}\z^n\Bigr)\\
E^+(\al,\uz)&= \exp\Bigl(\sum_{n \geq 1} \frac{p\al(n,-1)}{-n} z^{-n}+\frac{\p\al(-1,n)}{-n}\z^{-n} \Bigr).
\end{split}
\label{eq_lattice_vertex1}
\end{align}
For $\al \in H$,
define $z^{(p\al)(0)}\z^{(\p\al)(0)} \in \End\, \Om_{F,H}[z^\R,\z^\R]$ by
\begin{align}
z^{(p\al)(0)}\z^{(\p\al)(0)}v=z^{(p\al,p\be)_l}\z^{(\p\al,\p\be)_r}v
\label{eq_lattice_vertex2}
\end{align}
for $v \in \Omega_{F}^\be$.
Define a vertex operator 
\begin{align*}
\hY_\Om(\bullet,\uz): \Om_F \rightarrow \Om_F[[z^\R,\z^\R]]
\end{align*}
by 
\begin{align}
\hY_\Om(v,\uz)= 
E^-(-\al,\uz)Y(v,\uz)E^+(-\al,\uz)z^{(-p\al)(0)}\z^{(-\p\al)(0)}\label{eq_modified_vertex}
\end{align}
for $v \in \Omega_{F}^\al$.
For $u,v \in \Om_F$, it is nontrivial that the coefficients of $\hY(v,\uz)u$ are in $\Om_F$, but it can be proved from the commutator relation of the affine Heisenberg Lie algebra (see \cite{M1}).

Thus $\Om_F$ is an algebra with vertex operator $\hY_\Om(\bullet,\uz)$.
But $\hY_\Om(\bullet,\uz)$ contains terms like $z^r\z^s$ and is not single-valued around $z=0$.
%The $E^\pm(\al,\uz)$ and $Y(\bullet,\uz)$ are univalent.
Note that for $v \in \Om_F^\al$ and $u \in \Om_F^\be$,
\begin{align}
z^{-(\al,\be)_l}\z^{-(\al,\be)_r} &= 
z^{-(\al,\be)_l+(\al,\be)_r}(z\z)^{-(\al,\be)_r}
\label{eq_monodromy_source}
\end{align}
and $(z\z)^{-(\al,\be)_r}= \exp(-(\al,\be)_r\log(z\z))$ is single-valued around $z=0$.
Hence, the monodromy of $\hY_\Om(\bullet,\uz)$ around $z=0$ is controlled by the bilinear form $(-,-)_c$,
where 
%For $u,v \in \Om_F $\hY(v,\uz)u$の係数が$\Om_F$に入ることは非自明であるが、定義\eqref{eq_lattice_vertex1}と Heisenbergの交換関係から証明できる (see \cite{M1})。
%
%よって$\Om_F$は頂点作用素$\hY_\Om(\bullet,\uz)$を持った代数になる。
%しかし $\hY_\Om(\bullet,\uz)$は一般に $z^r\z^s$ のような項を含んでおり、$z=0$周りで一価にならない。
%$E^\pm(\al,\uz)$および$Y(\bullet,\uz)$は一価である。
%よって $v \in \Om_F^\al$ and $u \in \Om_F^\be$に対して、
%\begin{align}
%z^{-(\al,\be)_l}\z^{-(\al,\be)_r} &= 
%z^{-(\al,\be)_l+(\al,\be)_r}(z\z)^{-(\al,\be)_r}
%\label{eq_monodromy_source}
%\end{align}
%であり$(z\z)^{-(\al,\be)_r}= \exp(-(\al,\be)_r\log(z\z))$は一価であることより、$\hY_\Om(\bullet,\uz)$ のモノドロミーはbilinear form $(-,-)_c: H\otimes H \rightarrow H$ でコントロールされる。Here, $(-,-)_c$ is defined by
\begin{align}
(\al,\be)_c =(p\al,p\be)_l - (\p\al,\p\be)_r.
\label{eq_charge_lattice}
\end{align}

A generalized full vertex superalgebra is an algebra with a monodromy controlled by a real vector space $H$ equipped with a symmetric bilinear form $(-,-)_c$.
%
%generalized full vertex superalgebra とはある実ベクトル空間$H$とその上の symmetric bilinear form $(-,-)_c$ でコントロールされるモノドロミーを持つ代数である。

\begin{dfn}\cite[Section 4.2]{M1}
\label{def_fGVA}
A {\bf generalized full vertex superalgebra} is a real finite-dimensional vector space
$H$ equipped with a non-degenerate symmetric bilinear form 
$$(-,-)_c:H\times H \rightarrow \R$$
and an $\R^2\times H$-graded $\C$-vector space
$\Om=\bigoplus_{t,\td \in \R,\al \in H} \Om_{t,\td}^\al$ equipped with a linear map 
$$\hY(-,\uz):\Om \rightarrow \End \Om[[z^\R,\z^\R]],\; a\mapsto \hY(a,\uz)=\sum_{r,s \in \R}a(r,s)z^{-r-1}\z^{-s-1}$$
and an element $\va \in \Om_{0,0}^0$ %and a $\R^2$-graded subspace $F^\vee=\bigoplus_{r,s} F_{r,s}^\vee$ of the dual vector space $Hom_\C(F,\C)$ 
satisfying the following conditions:
%
% F^\vee is not nesesary?
\begin{enumerate}
\item[GFV1)]
For any $\al,\be \in H$ and $a \in \Om^\al$, $b \in \Om^\be$,
$z^{(\al,\be)}\hY(a,z)b \in \Om((z,\z,|z|^\R))$;
%For any $a,b \in \Om$,
%there exists $N\in \R$ such that $a(r,s)b=0$ for any $r \geq N$ or $s \geq N$;
\item[GFV2)]
$\Om_{t,\td}^\al=0$ unless $(\al,\al)/2+t-\td \in \ft \Z$,
which defines a $\Z_2$-grading on $\Om$;
\item[GFV3)]
For any $a \in \Om$, $\hY(a,\uz)\va \in \Om[[z,\z]]$ and $\lim_{z \to 0}\hY(a,\uz)\va = a(-1,-1)\va=a$;
\item[GFV4)]
$\hY(\va,\uz)=\mathrm{id}_\Om \in \End\, \Om$;
\item[GFV5)]
For any $\al_i \in H$ and  $a_i \in \Omega^{\al_i}$ ($i=1,2,3$) and $u \in \Omega^\vee=\bigoplus_{t,\td\in \R,\al\in H}(\Om_{t,\td}^\al)^*$,  there exists $\mu(z_1,z_2) \in C^\om(Y_2(\C))$ such that
\begin{align*}
(z_1-z_2)^{(\al_1,\al_2)}z_1^{(\al_1,\al_3)}z_2^{(\al_2,\al_3)}|_{|z_1|>|z_2|}
u(\hY(a_1,\uz_1)\hY(a_2,\uz_2)a_3) &= \mu(z_1,z_2)|_{|z_1|>|z_2|},\\
z_0^{(\al_1,\al_2)}(z_2+z_0)^{(\al_1,\al_3)}z_2^{(\al_2,\al_3)}|_{|z_2|>|z_0|}
u(\hY(\hY(a_1,\uz_0)a_2,\uz_2)a_3) &= \mu(z_0+z_2,z_2)|_{|z_2|>|z_0|},\\
(-1)^{|a_1||a_2|}(z_2-z_1)^{(\al_1,\al_2)}z_1^{(\al_1,\al_3)}z_2^{(\al_2,\al_3)}|_{|z_2|>|z_1|}
u(\hY(a_2,\uz_2)\hY(a_1,\uz_1)a_3) &= \mu(z_1,z_2)|_{|z_2|>|z_1|},
\end{align*}
where ${|a_i|}$ is the $\Z_2$-grading in (GFV2).
\item[GFV6)]
$\Om_{t,\td}^\al(r,s)\Om_{t',\td'}^\be \subset \Om_{t+t'-r-1,\td+\td'-s-1}^{\al+\be}$ for any $r,s,t,\td,t',\td' \in \R$
and $\al,\be \in H$;
\item[GFV7)]
For any $\al \in H$, there exists $N_\al \in \R$ such that $\Om_{t,\td}^\al=0$ for any $t \leq N_\al$ or
$\td \leq N_\al$.
\end{enumerate}
\end{dfn}

Let $\Om$ be a generalized full vertex superalgebra and $L_\Om(-1),\Ld_\Om(-1)$ denote the endomorphism of $\Om$ defined by
$L_\Om(-1) a =a(-2,-1)\va$ and $\Ld_\Om(-1) a = a(-1,-2)\va$.
Then, similar propositions as in the proposition\ref{prop_translation} hold (see \cite[Proposition 4.4]{M1}).
\begin{dfn}\label{def_fGVA_conformal}
A generalized full vertex operator superalgebra is a generalized full vertex superalgebra equipped with 
vectors $\om \in \Om_{2,0}^0$ and $\omb \in \Om_{0,2}^0$ such that:
\begin{enumerate}
\item
$\Ld_\Om(-1) \om=0$ and $L_\Om(-1) \omb=0$;
\item
There exist scalars $c, \bar{c} \in \C$ such that
$\om(3,-1)\om=\frac{c}{2} \va$,
$\omb(-1,3)\omb=\frac{\bar{c}}{2} \va$ and
$\om(k,-1)\om=\omb(-1,k)\omb=0$
for any $k=2$ or $k\geq 4$.
\item
$\om(0,-1)=L_\Om(-1)$ and $\omb(-1,0)=\Ld_\Om(-1)$;
\item
$\om(1,-1)|_{\Om_{t,\td}}=t$ and
$\omb(-1,1)|_{\Om_{t,\td}}=\td$ for any $t,\td \in \R$.
\end{enumerate}
\end{dfn}

%Then, it is easy to show that $\om_\Om = \om-\om_{H_l} \in (\Om_F)_{2,0}^0$ and $\omb_\Om = \omb-\omb_{H_r} \in (\Om_F)_{0,2}^0$
%and they define a generalized full vertex operator superalgebra structure on $\Om_F$.
Then, we have \cite[Theorem 5.3]{M1}:
\begin{thm}\label{thm_vacuum}
Let $(F,H_l,H_r,\om,\omb)$ be a full $\cH$-vertex operator superalgebra. Then,\\ $(\Om_F,\hY_\Om,\va,H,(-,-)_c,\om_\Om, \omb_\Om)$ is a generalized full vertex operator superalgebra, where
\begin{align}
\begin{split}
\om_{H_l} &= \ft \sum_{i\in I_l} h_i^l(-1)h_i^l,\\
\omb_{H_r} &= \ft \sum_{i\in I_r} h_i^r(-1)h_i^r,
\end{split}
\label{eq_om_H}
\end{align}
and $\{h_i^l\}_{i\in I_l}$ and $\{h_i^r\}_{i\in I_r}$ are orthonormal basis of $(H_l,(-,-)_l)$ and $(H_r,(-,-)_r)$, respectively.
\end{thm}

\begin{rem}\label{rem_history_generalized}
The notion of a generalized vertex algebra is introduced by Dong and Lepowsky \cite{DL}.
The method of obtaining a generalized vertex algebra from an $\mathcal{H}$-vertex algebra was discovered by \cite{Li20}.
%$\mathcal{H}$-vertex algebra から generalized vertex algebra を得る手法は\cite{Li20}が発見した。
We reinterpreted Li's result as giving the equivalence of categories in both the chiral and full settings.
%我々はLiの結果を functorial に捉えなおし、chiral / fullの場合にこの構成が圏同値を与えているものと理解し直した \cite{M1}。
Note that the definition of a generalized vertex algebra can be varied depending on the treatment of degrees (charges). Here we conform to \cite{M1}. Note also that \cite{DL, Li20, M1} only considered the bosonic case, not vertex superalgebra, but the generalization to vertex superalgebra is straightforward as above.
\end{rem}

Theorem \ref{thm_vacuum} yields a functor from the category of full $\cH$-vertex operator superalgebras to the category of generalized full vertex operator algebras. Hereafter in this section, we will review the opposite functor. First, let us summarize \cite{M1}:
%上記の定理により full $\cH$-vertex operator superalgebra の圏から、generalized full vertex operator algebras の圏への関手が得られる。この後この章では逆向きの関手を説明する。 full VOA から GVAを構成を振り返ると:
\begin{enumerate}
\item
As above, any full $\mathcal{H}$-VOA is a direct sum of the affine Heisenberg full VOA $M_{H,p}(0)$-modules:
\begin{align}
F= \bigoplus_{\al \in H}M_{H,p}(\al)\otimes \Om_F^\al,
\label{eq_F_bigsum}
\end{align}
where $M_{H,p}(\al)$ is the highest weight module of $M_{H,p}(0)$ with the highest weight $\al \in H$ (see below).
\item
Then, $\Om_F$ on the right-hand side of \eqref{eq_F_bigsum} is no longer commutative and associative, but is an algebra with monodoromy, which is axiomatized by generalized full vertex operator algebra.
\item
In fact, the left-hand-side of \eqref{eq_F_bigsum}
\begin{align*}
G_{H,p}=\bigoplus_{\al\in H} M_{H,p}(\al)
\end{align*}
also inherits a generalized full vertex operator algebra structure, and $F$ is actually a subalgebra of the large generalized full vertex operator algebra $G_{H,p}\otimes \Om_F$:
\begin{align}
F= \bigoplus_{\al \in H}M_{H,p}(\al)\otimes \Om_F^\al \subset  \bigoplus_{\al,\be \in H}M_{H,p}(\al)\otimes \Om_F^\be=G_{H,p}\otimes \Om_F.
\label{eq_F_tilde}
\end{align}
\item
Here $F$ is the diagonal sum in equation \eqref{eq_F_tilde} for $(\al,\be)\in H^2$, which causes a cancellation of the right and left monodromies. Therefore, the axiom of a full VOA is satisfied.
%ここで $F$は$(\al,\be)\in H^2$について式 \eqref{eq_F_tilde}における対角和をとっており、それにより右側と左側の代数の monodoromy のキャンセルが起きている。よって full VOA の公理を満たす。
\end{enumerate}

We formulate the spectral flow in Section \ref{sec_flow} as a 2-parameter family in this large generalized full vertex algebra \eqref{eq_F_tilde}. Hereafter, we will describe the generalized full vertex algebra structure on  $G_{H,p}$ based on \cite{M1},
which is a generalization of the generalized lattice vertex algebra introduced by Dong-Lepowsky in the chiral case \cite{DL}.
%
%我々はSection \ref{sec_flow}で spectral flow を、この大きな GVA \eqref{eq_F_tilde}の中の 1-parameter family として定式化する. この章ではこれより後\cite{M1}に基づき、$G_{H,p}$に入る generalized full vertex algebra の構造について説明する。
%

Let $H_l$ and $H_r$ be finite dimensional real vector spaces equipped with symmetric bilinear forms $(-,-)_l, (-,-)_r$. For $\al,\be \in H=H_l\oplus H_r$, set
\begin{align*}
(\al,\be)_p&= (p\al,p\be)+(\p\al,\p\be)\\
(\al,\be)_c&= (p\al,p\be)-(\p\al,\p\be),
\end{align*}
which are bilinear forms on $H$.
Let $\hat{H}^p=\bigoplus_{n \in \Z} H\otimes t^n \oplus \C c$ be the affine Heisenberg Lie algebra associated with $({H},(-,-)_p)$
and $\hat{H}_{\geq 0}^p=\bigoplus_{n \geq 0} H\otimes t^n \oplus \C c$ a subalgebra of $\hat{H}^p$.
Define the action of $\hat{H}_{\geq 0}^p$ on the group algebra of $H$, 
$\C[H]=\bigoplus_{\al \in H} \C e_\al$, by
\begin{align*}
c e_\alpha&=e_\al \\
h\otimes t^n e_\alpha &=
\begin{cases}
0, &n \geq 1,\cr
(h,\al)_p e_\alpha, &n = 0
\end{cases}
\end{align*}
for $\al \in H$.
Let $G_{H,p}$ be the $\hat{H}^p$-module induced from $\C[H]$.
Denote by $h(n)$ the action of $h \otimes t^n$ on $G_{H,p}$ for $n \in \Z$. Then,
\begin{align*}
G_{H,p}&= \bigoplus_{\al \in H}M_{H,p}(\al)\\
M_{H,p}(\al)&=M_{H_l}(p\al) \otimes M_{H_r}(\p \al).
\end{align*}
Let $\al \in H$.
Define $E^\pm(\al,\uz)$ and $z^{p\al(0)}\z^{\p\al(0)}$ similar to \eqref{eq_lattice_vertex1} and \eqref{eq_lattice_vertex2}.
Denote by $l_{e_\al} \in \End \C[H]$ the left multiplication by ${e_\al}$.
Then, we can define a unique generalized full vertex algebra structure on $G_{H,p}$ such that
the vertex operator $\hY_\std(\bullet,\uz):G_{H,p} \rightarrow \End G_{H,p}[[z^\R,\z^\R]]$
satisfies
\begin{align}
\hY_\std(e_\al,\uz) &=E^-(\al,\uz)E^+(\al,\uz)l_{e_\al} z^{p \alpha} \z^{\p \alpha}
\label{eq_standard_vertex}
\end{align}
(for more detail see \cite{M1}).

Set
\begin{align*} 
\va=1\otimes e_0,\quad \omega_{H_l}  = \frac{1}{2}\sum_{i=1}^{\dim H_l} h_l^i (-1)h_l^i,\quad \bar{\om}_{H_r} = \frac{1}{2} \sum_{j=1}^{\dim H_r} h_r^j (-1)h_r^j,
\end{align*}
where $h_l^i$ and $h_r^j$ is an orthonormal basis of $H_l\otimes_{\R} \C$ and $H_r\otimes_{\R} \C$ with respect to the bilinear form $(-,-)_p$.

\begin{prop}\cite[Proposition 4.10]{M1}\label{standard}
$(G_{H,p},\hY_\std,\va,H, -(-,-)_c, \omega_{H_l},\omb_{H_r})$ is a generalized full vertex operator algebra.
\end{prop}

Note that
\begin{itemize}
\item
In \eqref{eq_modified_vertex} and \eqref{eq_standard_vertex}, the sign of $z^{\pm p\al(0)}\z^{\pm \p\al(0)}$ is opposite. Therefore
\eqref{eq_monodromy_source} is the opposite sign for $\hY_{\mathrm{std}}(e_\al,\uz)$. This is why the bilinear form on $H$ is $-(-,-)_c$ in the above proposition, and thanks to that, \eqref{eq_F_bigsum} has no monodromy.
\item
The $\R^2\times H$-grading of $e_\al$ is 
$$e_\al \in (G_{H,p})_{\frac{1}{2}(p\al,p\al)_p,\frac{1}{2}(\p\al,\p\al)_p}^\al.
$$
Therefore, $\Z_2$-grading in (GFV2) is always
\begin{align*}
\frac{1}{2}(p\al,p\al)_p-\frac{1}{2}(\p\al,\p\al)_p 
- \ft (\al,\al)_c=0.
\end{align*}
Especially $G_{H,p}$ does not contain fermion.
%
%
%\eqref{eq_modified_vertex} and \eqref{eq_standard_vertex}では$z^{\pm p\al(0)}\z^{\pm \p\al(0)}$の項の符号が逆になっている。よって
%\eqref{eq_monodromy_source} は$\hY_{\mathrm{std}}(e_\al,\uz)$に対しては逆の sign になる。上記の命題で$H$上のbilinear form が $-(-,-)_c$ になる所以であり、そのおかげで\eqref{eq_F_bigsum}がmonodromy を持たない。
%\item
%The $\R^2\times H$-grading of $e_\al$ is 
%$$e_\al \in (G_{H,p})_{\frac{1}{2}(p\al,p\al)_p,\frac{1}{2}(\p\al,\p\al)_p}^\al.
%$$
%よって(GFV2)の$\Z_2$-grading は常に
%\begin{align*}
%\frac{1}{2}(p\al,p\al)_p-\frac{1}{2}(\p\al,\p\al)_p 
%- (\al,\al)_c/2=0.
%\end{align*}
%とくに$G_{H,p}$に fermion は含まれていない。
\end{itemize}

\subsection{Constant subalgebra of generalized full vertex algebra}\label{sec_const}
An important observation in this paper is that the volume form of a unitary $N=(2,2)$ full VOA satisfies the following condition by Lemma \ref{lem_primary_inequality}:
\begin{dfn}\cite[Section 3.1]{M7}\label{def_lattice_like}
Let $(F,H)$ be a full $\mathcal{H}$-vertex operator superalgebra. A vector $v \in \Om_F^\al$ is called a \textbf{lattice-like vector} if it satisfies
\begin{align}
\begin{split}
L(-1)v &= (p\al)(-1,-1)v,\\
\Ld(-1)v &= (\p\al)(-1,-1)v.
\end{split}
\label{eq_lattice_like}
\end{align}
\end{dfn}

Note that $v \in \Om_F^\al$ is a lattice-like vector if and only if
\begin{align}
v \in \ker L_\Om(-1) \cap \Ld_\Om(-1)
\end{align}
in the generalized full vertex operator algebra $\Om_F$ (see \cite[Lemma 3.2]{M7}).
In the case of a vertex operator algebra $V$, 
$A=\ker L(-1)$ is a commutative associative algebra by $-1$-th product $(a,b) \mapsto a(-1)b$ and the vertex operator $Y(\bullet,z):V \rightarrow \End V$ is automatically $A$-module homomorphism. That is, $A$ gives the coefficient ring of $V$.
In this section, we show that for a generalized full vertex operator superalgebra, $\ker L_\Om(-1)\cap \ker \Ld_\Om(-1)$ satisfies the similar properties.
%
%
%頂点作用素代数$V$の場合 $A=\ker L(-1)$は可換結合代数になり、頂点作用素$Y(\bullet,z):V \rightarrow \End V$は自動的に$A$-module hom になる。すなわち$A$は$V$の係数環を与える。
%この章では generalized full vertex operator superalgebra に対して、$\ker L_\Om(-1)\cap \ker \Ld_\Om(-1)$が同様の性質を満たすことを示す。

Let $(\Om,\hY(\bullet,\uz),\va,\om_\Om,\omb_\Om,H,(-,-)_c)$ be a generalized full vertex superalgebra.
Set
\begin{align}
\begin{split}
A_\Om &= \ker \Ld_\Om(-1) \cap \ker L_\Om(-1) \subset \Om\\
A_\Om^\al&= \Om^\al \cap A_\Om
\end{split}
\label{eq_A_om}
\end{align}
for $\al \in H$.

By definition, for $e_\al \in A_\Om$, $\hY(e_\al,\uz) = e_\al(-1,-1)$, that is, the vertex operator consisting of only the constant term. We denote $e_\al(-1,-1)$ by $e_\al \cdot$ for short.
In \cite[Proposition 5.12]{M1}, it was shown that $A_\Om$ is an associative algebra by the product $\cdot$.
\begin{lem}\label{lem_vacuum_like}
Let $e_\ga \in A_\Om^\ga$ and $a \in \Om^{\al}$. If $e_\ga \cdot a \neq 0$, then $(\al,\ga)_c \in \Z$ and
\begin{align*}
(-1)^{|a||e_\ga|+(\al,\ga)_c}\hY(a,\uz)e_\ga = \exp(L_\Om(-1)z+\Ld_\Om(-1)\z)(e_\ga\cdot a) = e_\ga \cdot \exp(L_\Om(-1)z+\Ld_\Om(-1)\z) a.
\end{align*}
\end{lem}
\begin{proof}
By \cite[Proposition 4.4 (4)]{M1}, for any $a \in \Om^\al$ and $b\in \Om^\be$,
\begin{align}
z^{(\al,\be)_c}\hY(a,\uz)b = (-1)^{|a||b|}\exp(L_\Om(-1)z+\Ld_\Om(-1)\z)
\lim_{\uz\to-\uz}z^{(\al,\be)_c}\hY(b,\uz)a
\label{eq_skew_general}
\end{align}
holds. Applying \eqref{eq_skew_general} to $b=e_\ga$, we have
\begin{align*}
z^{(\al,\ga)_c}\hY(a,\uz)e_\ga &=(-1)^{|a||e_\ga|} \exp(L_\Om(-1)z+\Ld_\Om(-1)\z)
\lim_{\uz\to-\uz}z^{(\al,\ga)_c}\hY(e_\ga,\uz)a\\
&=(-1)^{|a||e_\ga|}\exp(L_\Om(-1)z+\Ld_\Om(-1)\z)
\lim_{\uz\to-\uz}z^{(\al,\ga)_c}(e_\ga \cdot a).
\end{align*}
Since by definition $z^{(\al,\ga)_c}\hY(a,\uz)e_\ga \in \Om((z,\z,|z|^\R))$,
$(\al,\ga)_c \in \Z$ if $e_\ga \cdot a \neq 0$.
Assume that $e_\ga \cdot a \neq 0$.
Then, since $\lim_{\uz \to -\uz} z^{(\al,\ga)_c} = (-1)^{(\al,\ga)_c}z^{(\al,\ga)_c}$, we have
\begin{align*}
\hY(a,\uz)e_\ga &=(-1)^{|a||e_\ga|+(\al,\ga)_c}
\exp(L_\Om(-1)z+\Ld_\Om(-1)\z)(e_\ga \cdot a)\\
&=(-1)^{|a||e_\ga|+(\al,\ga)_c}
e_\ga \cdot \exp(L_\Om(-1)z+\Ld_\Om(-1)\z) a,
\end{align*}
where we used $[L_\Om(-1),\hY(e_\ga,\uz)] = \hY(L_\Om(-1)e_\ga,\uz) =0$ in the last line \cite[Proposition 4.4 (1) and (5)]{M1}.
\begin{comment}

For any $u\in \Om^\vee$, there exists $f(z_1,z_2) \in C^\om(Y_2(\C))$ such that:
\begin{align*}
(z_1-z_2)^{(\al,\ga)_c}|_{|z_1|>|z_2|}
\langle u, \hY(a,\uz_1) \hY(e_\ga,\uz_2)\va \rangle&= f(z_1,z_2)|_{|z_1|>|z_2|}\\
%(z_1-z_2)^{(\al,\ga)_c}\langle u, \hY(a,\uz_1) e_\ga \rangle 
(-1)^{|a||e_\ga|}(z_2-z_1)^{(\al,\ga)_c}|_{|z_2|>|z_1|}
\langle u, \hY(e_\ga,\uz_2) \hY(a,\uz_1) \va \rangle
%\langle u, e_\ga \cdot \hY(a,\uz_1) \va \rangle 
&= f(z_1,z_2)|_{|z_1|>|z_2|}.
%\langle u, \hY(\hY(a,\uz_0)e_\ga,\uz_2) \va \rangle &= f(z_1,z_2)|_{|z_1|>|z_2|}.
\end{align*}
By the standard argument of vertex algebra , $\hY(a,\uz)\va=\exp(L_\Om(-1)z+\Ld_\Om(-1)\z) a$ holds.
Since $[L_\Om(-1),\hY(e_\ga,\uz)] = \hY(L_\Om(-1)e_\ga,\uz) =0$,
\begin{align*}
\langle u, \hY(e_\ga,\uz_2) \hY(a,\uz_1) \va \rangle &= \langle u, e_\ga \cdot \hY(a,\uz_1) \va \rangle\\
 &=\langle u, \exp(L_\Om(-1)z_1+\Ld_\Om(-1)\z_1) e_\ga \cdot a \rangle \in \C[z_1,\z_1].
\end{align*}
Hence, if $e_\ga \cdot a \neq 0$, then $(z_2-z_1)^{(\al,\ga)_c}$ is a single-valued real analytic function on $Y_2$.
Thus, $(\al,\ga)_c \in \Z$. In this case, we have:
\begin{align*}
\langle u, \hY(a,\uz_1) e_\ga \rangle = (-1)^{|a||e_\ga|+(\al,\ga)_c} \langle u, \exp(L_\Om(-1)z_1+\Ld_\Om(-1)\z_1) e_\ga \cdot a \rangle.
\end{align*}
\end{comment}
\end{proof}
%
%\begin{prop}\label{prop_vac_right}
%For any $e_\ga \in A_\Om^\ga$, $a_i \in \Om^{\al_i}$ and $r,s \in \R$, if $e_\ga \cdot a_1 \neq 0$ and $e_\ga \cdot a_2\neq 0$, then
%\begin{align*}
%\hY(a_1,\uz_1) e_\ga \cdot a_2 =  (-1)^{|a_1||e_\ga|} (-1)^{(\al_1,\ga)_c} e_\ga \cdot \hY(a_1,\uz_1)a_2.
%\end{align*}
%\end{prop}
%\begin{proof}
%For any $u\in \Om^\vee$, there exists $f(z_1,z_2) \in C^\om(Y_2)$ such that:
%\begin{align*}
%z_1^{(\al_1,\al_2)_c}z_2^{(\ga,\al_2)_c}(z_1-z_2)^{(\al_1,\ga)_c}|_{|z_1|>|z_2|}
%\langle u, \hY(a_1,\uz_1) \hY(e_\ga,\uz_2)a_2 \rangle&= f(z_1,z_2)|_{|z_1|>|z_2|}\\
%(-1)^{|a_1||e_\ga|}z_1^{(\al_1,\al_2)_c}z_2^{(\ga,\al_2)_c}(z_1-z_2)^{(\al_1,\ga)_c}|_{|z_2|>|z_1|}
%\langle u, \hY(e_\ga,\uz_2) \hY(a_1,\uz_1) a_2 \rangle&= f(z_1,z_2)|_{|z_1|>|z_2|}.
%%\langle u, \hY(\hY(a,\uz_0)e_\ga,\uz_2) \va \rangle &= f(z_1,z_2)|_{|z_1|>|z_2|}.
%\end{align*}
%Since $\langle u, \hY(a_1,\uz_1) \hY(e_\ga,\uz_2)a_2 \rangle = \langle u, \hY(a_1,\uz_1) e_\ga\cdot a_2 \rangle$ is in $z_1^{-(\al_1,\al_2)_c}\C[z_1,\z_1,|z_1|^\R]$, we have:
%\begin{align*}
%\hY(a_1,\uz_1) e_\ga \cdot a_2 =  (-1)^{|a_1||e_\ga|} (-1)^{(\al_1,\ga)_c} e_\ga \cdot \hY(a_1,\uz_1)a_2.
%\end{align*}
%\end{proof}

\begin{prop}\label{prop_vac_right}
For any $e_\ga \in A_\Om^\ga$, $a_i \in \Om^{\al_i}$, if $e_\ga \cdot a_1 \neq 0$ and $e_\ga \cdot a_2\neq 0$, then
\begin{align*}
\hY(a_1,\uz_1) \Bigl(a_2(-1,-1) e_\ga\Bigr) =  \Bigl(\hY(a_1,\uz_1)a_2\Bigr) (-1,-1)e_\ga,
\end{align*}
that is, $a_1(r,s) \Bigl(a_2(-1,-1) e_\ga\Bigr)= \Bigl(a_1(r,s)a_2\Bigr) (-1,-1)e_\ga$ for any $r,s\in \R$.
\end{prop}
\begin{proof}
For any $u\in \Om^\vee$, there exists $f(z_1,z_2) \in C^\om(Y_2(\C))$ such that:
\begin{align}
\begin{split}
(z_1-z_2)^{(\al_1,\al_2)_c}|_{|z_1|>|z_2|}
\langle u, \hY(a_1,\uz_1) \hY(a_2,\uz_2)e_\ga \rangle&= f(z_1,z_2)|_{|z_1|>|z_2|}\\
z_0^{(\al_1,\al_2)_c}
\langle u, \hY(\hY(a_1,\uz_0)a_2,\uz_2)e_\ga \rangle&= f(z_2+z_0,z_2)|_{|z_2|>|z_0|},
\end{split}
\label{eq_ap_limit}
%
%(-1)^{|a_1||e_\ga|}z_1^{(\al_1,\al_2)_c}z_2^{(\ga,\al_2)_c}(z_1-z_2)^{(\al_1,\ga)_c}|_{|z_2|>|z_1|}
%\langle u, \hY(e_\ga,\uz_2) \hY(a_1,\uz_1) a_2 \rangle&= f(z_1,z_2)|_{|z_1|>|z_2|}.
%\langle u, \hY(\hY(a,\uz_0)e_\ga,\uz_2) \va \rangle &= f(z_1,z_2)|_{|z_1|>|z_2|}.
\end{align}
where we drop $z_1^{(\al_1,\ga)_c}$ and $z_2^{(\al_2,\ga)_c}$ by $(\al_1,\ga)_c, (\al_2,\ga)_c \in \Z$.
By Lemma \ref{lem_vacuum_like}, $f(z_1,z_2) \in \C[z_2,\z_2,(z_1-z_2)^\pm,(\z_1-\z_2)^\pm,|z_1-z_2|^\R]$.
Combining this with $\langle u, \hY(\hY(a_1,\uz_0)a_2,\uz_2)e_\ga \rangle=(-1)^{|e_\ga|(|a_1|+|a_2|)+(\ga,\al_1+\al_2)_c}
\langle u, e_\ga\cdot  \exp(L_\Om(-1)z_2+\Ld_\Om(-1)\z_2) \hY(a_1,\uz_0)a_2\rangle$, which consisting of finite polynomials of $z_2,\z_2$,
we can take the limit $z_2 \to 0$ of \eqref{eq_ap_limit}.
Hence, 
\begin{align*}
z_0^{(\al_1,\al_2)_c}\langle u, \Bigl(\hY(a_1,\uz_0)a_2\Bigr)(-1,-1)e_\ga \rangle
&=
\lim_{z_2 \to 0}
z_0^{(\al_1,\al_2)_c}\langle u, \hY(\hY(a_1,\uz_0)a_2,\uz_2)e_\ga \rangle\\
&=\lim_{z_2 \to 0} f(z_2+z_0,z_2) = f(z_0,0)
\end{align*}
and
\begin{align*}
z_1^{(\al_1,\al_2)_c} \langle u, \hY(a_1,\uz_1) \Bigl(a_2(-1,-1)e_\ga\Bigr) \rangle &=
\lim_{z_2 \to 0} (z_1-z_2)^{(\al_1,\al_2)_c}|_{|z_1|>|z_2|} \langle u, \hY(a_1,\uz_1) \hY(a_2,\uz_2)e_\ga \rangle\\
&=\lim_{z_2 \to 0} f(z_1,z_2) = f(z_1,0),
\end{align*}
which implies the assertion.
\end{proof}

Proposition \ref{prop_vac_right} says that the linear map 
\begin{align}
\Psi_\ga: \Om \rightarrow \Om,\quad a \mapsto a(-1,-1)e_\ga
\label{eq_app_Psi}
\end{align}
is a left $\Om$-module homomorphism.

\begin{prop}\label{prop_app_invertible}
Let $e_\ga \in A_\Om^\ga$ and $e_{-\ga} \in A_\Om^{-\ga}$ satisfy $e_\ga \cdot e_{-\ga}=\va$.
Then, $(\ga,\ga)_c +|e_\ga| \in 2\Z$ and, for any $a\in \Om$, $e_\ga \cdot a = 0$ if and only if $a=0$.
Moreover, $\Psi_\ga$ and $\Psi_{-\ga}$ in \eqref{eq_app_Psi} are mutually inverse.
\end{prop}
\begin{proof}
Note that by Lemma \ref{lem_vacuum_like} $e_\ga \cdot e_{-\ga} = (-1)^{|e_\ga||e_{-\ga}|} (-1)^{(\ga,\ga)_c} e_{-\ga} \cdot e_\ga$.
Assume that $e_\ga \cdot a = 0$ for $a \in \Om$.
Then, $\hY(e_{-\ga},\uz_1) \hY(e_\ga,\uz_2) a =0$. Thus,
\begin{align*}
0= \hY(\hY(e_{-\ga},\uz_0)e_\ga,\uz_2) a = \hY(e_{-\ga}\cdot e_\ga,\uz_2) a =
 (-1)^{|e_\ga||e_{-\ga}|+(\ga,\ga)_c} a.
\end{align*}
Since $e_\ga \cdot e_\ga = (-1)^{(\ga,\ga)_c +|e_\ga|} e_\ga \cdot e_\ga$, $(\ga,\ga)_c +|e_\ga|$ must be an even integer. Thus, $\va=e_\ga \cdot e_{-\ga} = e_{-\ga} \cdot e_\ga$.
By Proposition \ref{prop_vac_right}, we have
\begin{align*}
\Psi_{\ga}\circ \Psi_{-\ga} (a) &=\Bigl( a(-1,-1)e_{-\ga}\Bigr)(-1,-1)e_\ga \\
&=a(-1,-1) \Bigl( e_{-\ga}(-1,-1)e_\ga \Bigr)= a(-1,-1)\va =a.
\end{align*}
\end{proof}

\subsection{Periodicity of spectral flow}\label{sec_flow}
Let $F$ be an $N=(2,2)$ full vertex operator superalgebra with $c,\bc \neq 0$.
Let $H_l=\R J$ and $H_r =\R J_r$ equipped with bilinear forms $(J,J)_l = \frac{c}{3}$ and $(\bJ,\bJ)_r = \frac{\bc}{3}$.

Then, $(F,H=H_l\oplus H_r)$ is a full $\cH$-vertex operator algebra and, by Theorem \ref{thm_vacuum}, $\Om_F = \bigoplus_{\al \in H}\Om_{F}^{\al}$ is a generalized full vertex operator superalgebra with the charge lattice $(H,(-,-)_c)$.
Let $G_{H,p} = \bigoplus_{\al \in H}M_{H,p}(\al)$ be the generalized full vertex operator algebra in Proposition \ref{standard}. Then, $\Om_F \otimes G_{H,p}$ is a generalized full vertex operator superalgebra with the vertex operator (for the tensor product of generalized full vertex operator algebras see \cite[Section 4.5]{M1}):
\begin{align*}
\tY_\sft(\bullet,\uz) = \tY_\Om(\bullet,\uz)\otimes \tY_\std(\bullet,\uz) :\Om_F \otimes G_{H,p}
\rightarrow \End (\Om_F \otimes G_{H,p})[[z^\R,\z^\R]].
\end{align*}
For $\la,\mu \in \R$, let $F(\la,\mu)$ be the subspace of $\Om_F \otimes G_{H,p}$ defined by
\begin{align}
\begin{split}
F^\al(\la,\mu)&= \Om_F^{\al} \otimes M_{H,p}(\al-\la J-\mu \bJ),\\
F(\la,\mu) &= \bigoplus_{\al \in H}F^\al (\la,\mu).
\end{split}
\label{eq_spectral_twist}
\end{align}
The restriction of $\tY_\sft(\bullet,\uz)$ gives linear maps
\begin{align*}
\tY_\sft(\bullet,\uz): F(\la_1,\mu_1) \rightarrow \mathrm{Hom}_\C(F(\la_2,\mu_2),F(\la_1+\la_2,\mu_1+\mu_2))[[z^\R,\z^\R]].
\end{align*}
For $v_i \otimes e_{\al_i - \la_i J -\mu_i \bJ} \in F^{\al_i}(\la_i,\mu_i)$ ($i=1,2$), we have
\begin{align*}
\tY_\sft &(v_1 \otimes e_{\al_1 - \la_1 J-\mu_1 \bJ)},\uz)(v_2 \otimes e_{\al_2 - \la_2 J-\mu_2 \bJ}) \\
&= \hat{Y}_\Om(v_1,\uz)v_2 \otimes \tY_\std(e_{\al_1 - \la_1 J- \mu_1 \bJ},\uz)e_{\al_2- \la_2 J-\mu_2 \bJ}\\
 &\in z^{-(\al_1,\al_2)_c+(\al_1-\la_1 J-\mu_1 \bJ, \al_2-\la_2 J -\mu_2 \bJ )_c}
 F^{\al_1+\al_2}(\la_1+\la_2,\mu_1,\mu_2)((z,\z,|z|^\R)).
\end{align*}
%we have $\tY(v_1 \otimes e_{r_1 + p\la_1},z)v_2 \otimes e_{r_2 + p\la_2} \in z^{}V(\la_1+\la_2)((z))$
If $\la_1=\mu_1=0$, then $z^{-(\al_1,\al_2)_c+(\al_1-\la_1 J-\mu_1 \bJ, \al_2-\la_2 J -\mu_2 \bJ )_c} =z^{-(\al_1,\la_2 J+\mu_2 \bJ))_c}=
z^{-\la_2(\al_1,J)_l+\mu_2 (\al_1,\bJ)_r)}$. Hence,
\begin{align*}
\tY_\sft(\bullet,\uz): F^\al(0,0) \rightarrow z^{-\la_2(\al_1,J)_l+\mu_2 (\al_1,\bJ)_r)}\Hom(F^{\be}(\la,\mu),F^{\al+\be}(\la,\mu))[[z,\z,|z|^\R]].
\end{align*}
Note that $(F(0,0),\tY_\sft(\bullet,\uz))$ is a full vertex operator algebra which is isomorphic to $(F,Y(\bullet,\uz))$ by Theorem \ref{thm_vacuum}.

\begin{prop}\label{prop_chiral_twist}
For any $\la,\mu \in \R$, the restriction $\tY_\sft(\bullet,\uz): F(0,0) \rightarrow \End(F(\la,\mu))[[z^\R,\z^\R]]$ defines an $\exp(2\pi i (\la J(0)- \mu \bJ(0)))$-twisted $F(0,0)$-module structure on $F(\la,\mu)$.
\end{prop}
\begin{proof}
It suffices to show the condition (6) in Definition \ref{def_twisted}.
Let $\la,\mu \in \R$.
Since $\Om_F \otimes G_{H,p}$ is a generalized full vertex superalgebra with the charge lattice $(H \oplus H,(-,-)_c\oplus -(-,-)_c)$,
for any $a_i \in F^{\al_i}(0,0)$ ($i=1,2$) and $v \in F^{\al_3}(\la,\mu)$ and
$u \in (F^{\al_1+\al_2+\al_3}(\la,\mu))^*$,
there exists a real analytic function $f(z_1,z_2) \in C^\om(Y_2(\C),\C)$ such that:
\begin{align}
\begin{split}
z_1^{(\al_1,\la J+\mu \bJ)_c} z_2^{(\al_2,\la J+\mu \bJ)_c}
\langle u, \tY_\sft(a_1,z_1)\tY_\sft(a_2,z_2)v \rangle
&= f(z_1,z_2)|_{|z_1|>|z_2|},\\
(-1)^{|a_1||a_2|}z_1^{(\al_1,\la J+\mu \bJ)_c} z_2^{(\al_2,\la J+\mu \bJ)_c}
\langle u, \tY_\sft(a_2,z_2)\tY_\sft(a_1,z_1) v \rangle
&= f(z_1,z_2)|_{|z_2|>|z_1|},\\
(z_2+z_0)^{(\al_1,\la J+\mu\bJ)_c}|_{|z_2|>|z_0|} z_2^{(\al_2,\la J+\mu \bJ)_c}
\langle u, \tY_\sft(\tY_\sft(a_1,z_0)a_2,z_2)v \rangle
&= f(z_2+z_0,z_2)|_{|z_2|>|z_0|}.
\end{split}
\end{align}
Since $(\al,\la J+\mu \bJ)_c= \la(\al,J)_l - \mu(\al,\bJ)_r$ for any $\al\in H$,
$F(\la,\mu)$ is an $\exp(2\pi i \la J(0) -\mu\bJ(0)))$-twisted module.
\end{proof}

By Proposition \ref{prop_chiral_twist}, $L(0),\Ld(0)$ and $J(0),\bJ(0)$ act on $F(\la,\mu)$. Set
\begin{align}
F_{h,\h}^\al\left(\la,\mu \right) = \{v\in F(\la,\mu)\mid L(0)v = hv,\Ld(0)v =\h v, J_0v=(\al,J)_lv, \bJ_0v =(\al,\bJ)_r v\}.
\end{align}
Then, we have:

\begin{prop}\label{prop_flow_dim}
For any $h,\h\in \R$, $\al \in H$ and $\la,\mu \in \R$,
\begin{align*}
\dim F_{h,\h}^\al = \dim F_{h-\la(\al,J)+\frac{\la^2 c}{6},\h-\mu(\al,\bJ)+\frac{\mu^2 \bc}{6}}^{\al-\la J-\mu\bJ}\left(\la,\mu\right).
\end{align*}
\end{prop}
\begin{proof}
By definition, we have
\begin{align*}
F_{h,\h}^\al = \bigoplus_{k, l \geq 0} (\Om_F)_{h-\frac{(\al,\al)_l}{2}-k,\h-\frac{(\al,\al)_r}{2}-l}^\al \otimes (M_{H,p}(\al))_{k+\frac{(\al,\al)_l}{2},l+\frac{(\al,\al)_r}{2}}
\end{align*}
and
\begin{align*}
&F_{h-\la(\al,J)+\frac{\la^2 c}{6},\h-\mu(\al,\bJ)+\frac{\mu^2 \bc}{6}}^{\al-\la J-\mu\bJ}\left(\la,\mu\right)\\
& = \bigoplus_{k, l \geq 0} (\Om_F)_{h-\la(\al,J)+\frac{\la^2 c}{6}-\frac{(\al-\la J,\al-\la J)_l}{2}-k,\h-\mu(\al,\bJ)+\frac{\mu^2 \bc}{6}-\frac{(\al-\mu \bJ,\al-\mu \bJ)_r}{2}-l}^{\al} \otimes (M_{H,p}(\al-\la J-\mu \bJ))_{k+\frac{(\al-\la J,\al-\la J)_l}{2},l+\frac{(\al-\mu \bJ,\al-\mu \bJ)_r}{2}}\\
& = \bigoplus_{k, l \geq 0} (\Om_F)_{h-\frac{(\al,\al)_l}{2}-k,\h-\frac{(\al,\al)_r}{2}-l}^{\al} \otimes (M_{H,p}(\al-\la J-\mu \bJ))_{k+\frac{(\al-\la J,\al-\la J)_l}{2},l+\frac{(\al-\mu \bJ,\al-\mu \bJ)_r}{2}}.
\end{align*}
Since the dimension of $(M_{H,p}(\al))_{k+\frac{(\al,\al)_l}{2},l+\frac{(\al,\al)_r}{2}}$ is the number of partitions and only depends on $k,l$, the assertion holds.
\end{proof}

%\begin{rem}
%Our shifting convention is chosen so that the module is twisted by $\exp(2\pi i \la(J(0)-\bJ(0)))$.
%If we consider the opposite shift
%\begin{align*}
%\bigoplus_{\al \in H} \Om_F^{\al} \otimes M_{H,p}(\al {\bf +} \la(J+\bJ)),
%\end{align*}
%then it is $\exp(-2\pi i \la (J(0)-\bJ(0))) $-twisted module.
%\end{rem}

Set
\begin{align*}
F(\la)=F(\la,\la),
\end{align*}
which we call the \textbf{spectral flow} of $F$.
By Proposition \ref{prop_chiral_twist}, $F(\la)$ is a $\exp(2\pi i \la (J(0)-\bJ(0)))$-twisted module .
Since by \eqref{eq_fermi_integer} for any $n\in \Z$
$\exp(2 \pi i n (J(0)-\bJ(0)) )=\mathrm{id}_F$,
$F(n)$ is not twisted but a usual $F(0)$-module by Remark \ref{rem_just_module}.

\begin{rem}
\label{rem_phys_spectral}
The parameter $\la\in\R$ of $\{F(\la)\}_{\la\in\R}$
corresponds to the boundary condition of fermions $\psi$ on $S^1=\R/\Z$ (in the path-integral):
\begin{align*}
\psi(\theta+1)= \exp(2 \pi i \la) \psi(\theta)
\end{align*}
in physics.
%corresponds to the vector space formed by the quantum field changed to
%Such an operation is called spectral flow. 
From such a geometric point of view, the spectral flow should be periodic.
%物理的には族$F(\la)$は$S^1=\R/\Z$上のフェルミオン$\psi$の境界条件を
%\begin{align*}
%\psi(\theta+1)= \exp(2 \pi i \la) \psi(\theta)
%\end{align*}
%と変えた量子場のなすベクトル空間に対応している。
%このような操作は spectral flow と呼ばれる。幾何学的な意味から、spectral flow は周期性を持つべきであり、そのような必要十分条件は次の命題で述べられる。
\end{rem}

%The following proposition gives a necessary and sufficient condition for when $V(\pm)$ are isomorphic to $V(0)$ as $V(0)$-modules. However, this proposition is not necessary for the main result of this paper.

\begin{thm}\label{thm_chiral_period}
Let $F$ be an $N=(2,2)$ full vertex operator superalgebra.
Then, the following conditions are equivalent:
\begin{enumerate}
\item
For any $n \in \Z$, $F(n)$ is isomorphic to $F(0)$ as $F(0)$-modules;
\item
$F(1)$ and $F(-1)$ are isomorphic to $F(0)$ as $F(0)$-modules;
%\item
%There are non-zero vectors $v_{\pm M} \in \Om_V^{\pm \frac{M}{p}}$ such that $L(-1)v_{\pm M} =\pm M h(-1)v$
%and $v_M(M^2-1)v_{-M} = \va.$
\item
There are vectors $v_+ \in \Om_F^{J+\bJ}$ and $v_- \in \Om_F^{-J-\bJ}$ such that
$L(-1)v_\pm= \pm J(-1)v_\pm$, $\Ld(-1)v_\pm= \pm \bJ(-1)v_\pm$ and $v_+(\frac{c}{3}-1,\frac{\bc}{3}-1)v_- = \va$.
\end{enumerate}
Furthermore, if one of the above conditions are satisfied, then
$(\Om_F)_{\frac{(J,J)_l}{2},\frac{(\bJ,\bJ)_r}{2}}^{J+\bJ} =1$ and $F(\la)$ is isomorphic to $F(\la+n)$ as $\exp(2\pi i \la ( J(0)-\bJ(0)))$-twisted modules and
\begin{align*}
\dim F_{h,\h}^\al = \dim F_{h-n(\al,J)+\frac{n^2 c}{6},\h-n(\al,\bJ)+\frac{n^2 \bc}{6}}^{\al-n(J+\bJ)}
\end{align*}
 for any $\la\in\R$ and $n\in \Z$.
\end{thm}
\begin{proof}
Assume that (2) holds and let $f_\pm:F(\pm 1) \rightarrow F$ be $F(0)$-module isomorphisms.
Set
\begin{align*}
v_{\pm} = f_\pm(\va \otimes e_{\pm (J+\bJ)}).
\end{align*}
Then, $v_{\pm} \in \Om_F^{\pm (J+\bJ)}$, $L(-1)v_{\pm} =\pm J(-1)v_\pm$ and $\Ld(-1)v_{\pm} =\pm \bJ(-1)v_\pm$.

By regarding $v_+ \otimes e_{J+\bJ}$ as a vector of $\Om_F^{J+\bJ} \otimes \C e_{J+\bJ} \subset F(0)$,
the action of $F(0)$ on the modules $F(1)$ is,
\begin{align*}
\tY_\sft(v_+ \otimes e_{J+\bJ},\uz) (\va \otimes e_{-{J+\bJ}}) 
&=\hY_\Om(v_+,\uz)\va \otimes \hY_{\std}(e_{{J+\bJ}},\uz)e_{-{J+\bJ}}\\
&=(v_+  + O(z,\z)) \otimes z^{-\frac{c}{3}}\z^{-\frac{\bc}{3}}(1+J(-1)z +\bJ(-1)\z+O(z^2))e_0,
\end{align*}
which implies $v_+\otimes e_0 = \left(v_+ \otimes e_{J+\bJ}\right) (\frac{c}{3}-1,\frac{\bc}{3}-1)\left( \va \otimes e_{-J-\bJ} \right)$.
Hence, we have
\begin{align*}
f_-(v_+ \otimes e_{0})&=f_-\left(\left(v_+ \otimes e_{J+\bJ}\right) (\frac{c}{3}-1,\frac{\bc}{3}-1)\left( \va \otimes e_{-J-\bJ} \right)\right)\\
&=v_+ (\frac{c}{3}-1,\frac{\bc}{3}-1) f_-\left( \va \otimes e_{-J-\bJ} \right)\\
&=v_+ (\frac{c}{3}-1,\frac{\bc}{3}-1) v_{-}.
\end{align*}
Since $f_-$ is isomorphism and $v_+ \otimes e_0 \neq 0$, $v_+ (\frac{c}{3}-1,\frac{\bc}{3}-1) v_{-} \neq 0$.
Since
\begin{align*}
v_\pm \in \Bigl(\ker(\om(0)-\om_{H_l}(0))\Bigr) \cap \Bigl(\ker(\omb(0)-\omb_{H_r}(0))\Bigr),
\end{align*}
$v_+ (\frac{c}{3}-1,\frac{\bc}{3}-1) v_{-} \in \Om_F^0 \cap \ker L(-1) \cap \ker \Ld(-1)=F_{0,0}=\C\va$.
Hence, $v_+ (\frac{c}{3}-1,\frac{\bc}{3}-1) v_{-} = A \va$ with some non-zero $A \in \C$.
%which implies that $\{v_+,v_{-}\}$ generate an $\mathcal{H}$-vertex subalgebra $W$ which is isomorphic to the lattice VOA $V_{\sqrt{\frac{c}{3}}M \Z}$ \cite[Lemma 3.14]{M7}.

Assume (3) and let $n >0$ be an integer. By regarding $v_\pm$ as elements of the generalized full vertex algebra $\Om_F$,
$v_\pm$ are elements of $A_{\Om_F}$ in \eqref{eq_A_om}.
Set
\begin{align*}
v_{n} = \Bigl(v_+(-1,-1)\Bigr)^{n-1} v_+,\quad \quad v_{-n} = \Bigl(v_-(-1,-1)\Bigr)^{n-1} v_-.
\end{align*}
Since $v_+ (-1,-1)v_- =\va$, where the product is taken as the generalized full vertex algebra,
by Proposition \ref{prop_app_invertible}, $v_{\pm n} \in \Om_F^{\pm n(J+\bJ)} \cap A_{\Om_F}$ are non-zero vectors,
which satisfies $v_{n} (-1,-1)v_{-n}=\pm \va$. 
If necessary, multiply $v_n$ by $-1$ and normalize it as $v_n(-1,-1)v_{-n}=\va$.
% Then, there are vectors $v_{\pm k} \in \Om_V^{ \pm MkJ}$ such that
%$L(-1)v_{\pm k} =\pm M k J(-1)v_\pm$ and $v_{+k}(\frac{cM^2k^2}{3}-1)v_{-k} = \va$.
%By regarding $v_{\pm k}$ as elements of the generalized vertex algebra $\Om_V$,
%by $L(-1)v_{\pm k} =\pm {Mk} J(-1)v_{\pm k}$, $v_{\pm k} \in \ker L_\Om(-1)$.
Let $\Psi_{\pm n}$ be the linear map defined by
\begin{align}
\Psi_{\pm n}:\Om_F \rightarrow \Om_F,\quad a \mapsto a(-1,-1)v_{\pm n},
\label{eq_psin}
\end{align}
where $a(-1,-1)$ is the product of the generalized full vertex algebra (not the full vertex algebra).
Then, by Proposition \ref{prop_vac_right} and Proposition \ref{prop_app_invertible}, the following properties hold:
\begin{itemize}
\item
For any $a,b\in \Om_F$, $\Psi_{\pm n}(a(r,s)b) = a(r,s)\Psi_{\pm n} (b)$, that is, $\Psi_{\pm n}$ are $\Om_F$-module homomorphisms.
\item
$\Psi_{+n} \circ \Psi_{-n} = \Psi_{-n} \circ \Psi_{+n} =\mathrm{id}_{\Om_F}$. In particular, $\Psi_{\pm n}$ are linear isomorphisms.
\end{itemize}
Define $f_{n}:F(n) \rightarrow F(0)$ by
\begin{align}
\begin{split}
f_{n}=\Psi_{-n}\otimes \mathrm{id}_{G_{H,p}}&:F(n)= \bigoplus_{\al \in H} \Om_F^\al \otimes M_{H,p}(\al-n(J+\bJ))\\ &\longrightarrow 
\bigoplus_{\al\in H} \Om_F^{\al -n(J+\bJ)} \otimes M_{H,p}(\al-n(J+\bJ)) = F(0)
\end{split}
\label{eq_period_isom}
\end{align}
and $f_{-n}:F(-n) \rightarrow F(0)$ in the same way.
Then, it is clear that $f_{\pm n}$ is a $F(0)$-module isomorphisms
since $\id_{G_{H,p}}$ is an isomorphism of generalized full vertex algebra. Hence, (1) holds.
(1) clearly implies (2), and $\Psi_{\pm n}$ more generally leads to isomorphisms of $F(\la)$ and $F(\la \pm n)$ for any $\la \in \R$.
Since $(\Om_F)_{\frac{(J,J)_l}{2},\frac{(\bJ,\bJ)_r}{2}}^{J+\bJ} \cong (\Om_F)_{0,0}^{0} =\C \va$, the assertion holds.
\end{proof}

\begin{thm}\label{thm_unitary_period}
Let $F$ be a unitary $N=(2,2)$ full vertex operator superalgebra.
Then, the following conditions are equivalent:
\begin{enumerate}
\item
For any $n \in \Z$, $F(n)$ is isomorphic to $F(0)$ as $F$-modules;
\item
There is a non-zero vector in $H^{\top}(F,d_B)$ (Definition \ref{def_full_vol}).
\end{enumerate}
\end{thm}
\begin{proof}
Assume (1) holds.
Then, there is a non-zero vector $v \in F_{\frac{c}{6},\frac{\bc}{6}}^{J+\bJ}$ by Proposition \ref{thm_chiral_period}.
By Proposition \ref{prop_cohomology_primary}, $F_{\frac{c}{6},\frac{\bc}{6}}^{J+\bJ} = H^\top(F)$. Hence, (2) holds.

Assume (2) holds. Then, by Proposition \ref{prop_cohomology_primary} and Lemma \ref{lem_primary_inequality}, there is a non-zero vector 
$v_+ \in \Om_F^{J+\bJ}$ such that $L(-1)v_+= J(-1)v_+$ and $\Ld(-1)v_+ = \bJ(-1)v_+$.
Set $v_-=\phi(v_+)$, where $\phi$ is the anti-linear involution. Then, by definition, $v_- \in \Om_{F}^{-J-\bJ}$ and $L(-1)v_-= -J(-1)v_-$ and $\Ld(-1)v_- = -\bJ(-1)v_-$.
Since
\begin{align}
\begin{split}
\langle v_+,v_+ \rangle &= (\phi(v_+),v_+)=(v_-, v_+(-1,-1)\va)\\
&=(-1)^{s+2s^2} (v_+((J,J)_l-1,(\bJ,\bJ)_r-1)v_-,\va)
\end{split}
\label{eq_star_normalize}
\end{align}
where $s = \frac{c}{6}-\frac{\bc}{6}$,
$v_+((J,J)_l-1,(\bJ,\bJ)_r-1)v_- \neq 0$.
Since $v_+((J,J)_l-1,(\bJ,\bJ)_r-1)v_- \in (\Om_F)_{0,0}^{0} \subset F_{0,0} =\C\va$, $v_+((J,J)_l-1,(\bJ,\bJ)_r-1)v_- = A\va$ with some non-zero $A\in \C$.
Hence, by Theorem \ref{thm_chiral_period}, (1) holds.
\end{proof}

\begin{cor}
Let $F$ be a unitary $N=(2,2)$ full vertex operator superalgebra.
Assume that any irreducible $F$-module is isomorphic to $F$ as an $F$-module.
Then, $H^\top(F) \neq 0$.
\end{cor}

We end this section by looking at periodicity for $F(\la,\mu)$. Physically, there is no guarantee in general that $F(\la,\mu)$ is periodic in both chiral and anti-chiral directions. In fact, the Landau-Ginzburg model has no such periodicity. However, it is known that $F$ has such periodicity in the case of the supersymmetric sigma model associated with a Calabi-Yau manifold. 
We characterize such periodicity by the cohomology ring $H(F,d_B)$.

%$F(\la,\mu)$についての周期性をみて我々はこの章を終える。物理的には一般の$N=(2,2)$理論がカイラル・反カイラル方向のフェルミオンを独立にtwist できる保証は全くない。実際、Landau-Ginzburg 模型はこのような周期性をもたない。しかし、$F$が supersymmetric sigma model の場合は、このような周期性を持つことが知られている。
%我々はこうした周期性を$H(F,d_B)$の元で特徴づける。
\begin{lem}
Let $F$ be a unitary $N=(2,2)$ full vertex operator algebra and $\eta \in C^{J,0}(F)$ a non-zero vector.
Then, there is $\eta_- \in \Om_F^{-J,0}$ such that $L(-1)\eta_- = J(-1)\eta_-$, $\Ld(-1)\eta_-=0$ and $\eta_-((J,J)_l-1,-1)\eta =\va$. Moreover, $J(0)$ has only integer eigenvalues.
\end{lem}
\begin{proof}
By Lemma \ref{lem_chiral_volume} and Lemma \ref{lem_primary_inequality},
$\eta$ satisfies $L(-1) \eta = J(-1)\eta$ and $\Ld(-1)\eta =0$.
By $\eta \in F_{\frac{(J,J)}{2},0}$ and (FO1), we have $(J,J)_l = \frac{c}{3} \in \Z$.
Since
\begin{align*}
0&\neq \langle \eta,\eta \rangle = (\phi(\eta),\eta) = (-1)^{s(\eta)+2s(\eta)^2}(\va, \phi(\eta)((J,J)_l-1,-1)\eta),
\end{align*}
by setting $\eta_- = \frac{(-1)^{s(\eta)+2s(\eta)^2}}{ \langle \eta,\eta \rangle} \phi(\eta)$, $\eta_-((J,J)_l-1,-1)\eta =\va$ holds (see the proof of Theorem \ref{thm_chiral_period}), and
$\exp(2\pi i J(0)) =\id_F$ follows from Proposition \ref{lem_vacuum_like} and Proposition \ref{prop_app_invertible}.
\end{proof}
Similarly to Theorem \ref{thm_chiral_period}, we have:\begin{prop}\label{prop_period_cc}
Let $F$ be a unitary $N=(2,2)$ full vertex operator superalgebra. The following conditions are equivalent:
\begin{enumerate}
\item
For any $n\in \Z$, $\exp(2\pi i n J(0)) =\id_F$ and $F(n,0)$ is isomorphic to $F(0,0)$ as $F$-modules;
\item
$H^{J,0}(F,d)$ is non-zero.
\end{enumerate}
Moreover, if one of the above conditions are satisfied, then $c \in 3\Z$, $\dim H^{J,0}(F) = \dim (\Om_{F})_{\frac{(J,J)_l}{2},0}^{J,0} =1$ and $F(\la,\rho)$ is isomorphic to $F(\la+n,\rho)$ as $\exp(2\pi i(\la J(0)-\rho \bJ(0)))$-twisted modules
for any $\la,\rho \in \R$ and $n\in\Z$
 and
\begin{align*}
\dim F_{h,\h}^\al = \dim F_{h-n(\al,J)+\frac{n^2 c}{6},\h}^{\al-nJ}.
\end{align*}
for any $n\in\Z$.
\end{prop}
Hence, we have:
\begin{thm}\label{thm_CY}
Let $F$ be a unitary $N=(2,2)$ full vertex operator superalgebra. The following conditions are equivalent:
\begin{enumerate}
\item
All of 
$H^{J,0}(F,d)$, $H^{0,\bJ}(F,d)$ and $H^{J,\bJ}(F,d)$ are non-zero.
\item
Two of 
$H^{J,0}(F,d)$, $H^{0,\bJ}(F,d)$ and $H^{J,\bJ}(F,d)$ are non-zero.
\item
For any $n, m \in \Z$, $\exp(2\pi i (n J(0)+m\bJ(0))) =\id_F$ and $F(n,m)$ is isomorphic to $F(0,0)$ as $F$-modules.
\end{enumerate}
Moreover, if one of the above conditions are satisfied, then
\begin{align*}
\dim F_{h,\h}^\al = \dim F_{h-n(\al,J)+\frac{n^2 c}{6},\h-m(\al,\bJ)+\frac{m^2 c}{6}}^{\al-nJ-m\bJ}.
\end{align*}
for any $n,m \in \Z$.
\end{thm}
\begin{proof}
Assume (2). In the case of $H^{J,0}(F,d)\neq 0$ and $H^{J,\bJ}(F,d)\neq 0$, by Proposition \ref{prop_period_cc},
\begin{align*}
\dim F_{0,\frac{(\bJ,\bJ)_r}{2}}^{0,\bJ} = \dim F_{(J,\bJ)+c/6,\frac{(\bJ,\bJ)_r}{2}}^{J,\bJ}
=\dim F_{\frac{(J,J)_l}{2},\frac{(\bJ,\bJ)_r}{2}}^{J,\bJ} = \dim H^{J,\bJ}(F,d).
\end{align*}
Hence, (1) holds. The other cases can similarly be obtained.
By Proposition \ref{prop_period_cc}, (1) and (3) are equivalent.
\end{proof}

\subsection{Relation with Lie theoretical spectral flow}
\label{sec_Lie_flow}
Let $F$ be an $N=(2,2)$ full vertex operator superalgebra.
By Proposition \ref{prop_chiral_twist}, $F(\la,\mu)$ is an $\exp(2\pi i ( \la J(0)- \mu \bJ(0)))$-twisted modules,
and thus, by Proposition \ref{prop_twist_susy}, the twisted $N=(2,2)$ superconformal algebras $\g_\la^{N=2} \oplus \g_{\mu}^{N=2}$ acts on $F(\la,\mu)$.
It is important to note that this (vertex algebra theoretical) twisted action is completely different from the
one which obtained by the pull-back of the Lie algebra isomorphism \ref{def_spectral_Lie}
\begin{align*}
U_\eta: \ga_\la^{N=2} \rightarrow \g_{\la+\eta}^{N=2}.
\end{align*}
In this section, we will study the relation between them.

Let $D_{\la,\mu}: F(0,0) \rightarrow F(\la,\mu)$ be a linear map defined by
\begin{align}
\begin{split}
D_{\la,\mu}&(v_\al \otimes J(-i_1-1)\cdots J(-i_k-1) \bJ(-j_1-1)\cdots \bJ(-j_l-1)e_{\al})\\
&= v_\al \otimes J(-i_1-1)\cdots J(-i_k-1) \bJ(-j_1-1)\cdots \bJ(-j_l-1)e_{\al-\la J-\mu\bJ)}
\end{split}
\label{eq_Dla}
\end{align}
for $\al\in H$, $v_\al \in \Om_F^\al$ and $i_1,\dots,i_k,j_1,\dots,j_l \geq 0$.
Note that this map is just a linear map that does not reflect the module structures.

%\item
%For any $k \in \Z$, $V(\frac{6k}{c}J)$ in Proposition \ref{prop_chiral_twist} is isomorphic to $V(0)$ as $V(0)$-modules.
%\end{enumerate}
%\end{dfn}
%Set $h = \sqrt{\frac{3}{c}}J$ so that $h(1)h=\va$.

Let $\la,\mu \in \R$.
For any $\al \in H$ and $h,\h \in \R$, let $C_{h,\h}^\al(F(\la,\mu))$ be the subspace of $F(\la,\mu)$ consisting of vectors $v$ such that:
\begin{align}
\begin{split}
\begin{cases}
U_{\la}(L_n) v=0=U_{\mu}(\Ld_n)v & (n\geq 1),\\
U_{\la}(J_n) v=0=U_{\mu}(\bJ_n)v & (n\geq 1),\\
U_\la(G_r^\pm) v=0=U_{\mu}(\bG_r^\pm)v & (r \geq \ft \text{ and }r \in \ft+\Z),\\
U_\la(G_{-\ft}^+) v=0=U_{\mu}(\bG_{-\ft}^+)v,\\
U_\la(L(0)) v = hv, \quad U_{\mu}(\Ld(0)) v = \h v,\\
U_\la(J(0))v = (\al,J)_l v,\quad U_{\mu}(\bJ(0))v = (\al,\bJ)_r v.
\end{cases}
\end{split}
\label{eq_twist_primary}
\end{align}
Here $U_\la(G_r^\pm)$ etc. are thought to be elements of $\g_{\la}^{N=2}\oplus \g_{\mu}^{N=2}$, and its action on $F(\la,\mu)$ is given by Proposition \ref{prop_twist_susy}.
%ここで$U_\la(G_r^\pm)$などは$\g_{\la}^{N=2}$の元と思っており、その$F(\la)$への作用は命題1で与えれている。
Note that $C_{h,\h}^\al(F(0,0))$ coincides with the space of c-c primary vectors (Definition \ref{def_full_primary}).

To see what this definition means, we will examine the case of $\la=\ft=\mu$.
In this case, $F(\ft)=F(\ft,\ft)$ is a module of the $N=(2,2)$ Ramond algebras $\g_\ft^{N=2} \oplus \g_\ft^{N=2}$.
Then, $C_{h,\h}^\al(F(\ft))$ coincides with the vector space called the {\it Ramond vacuum} defined for the module of $N=(2,2)$ Ramond algebra.
By definition of $U_\ft$, 
$v \in F(\ft)$ is in $C_{h+\ft(J,\al)_l+\frac{c}{24},\h+\ft(\bJ,\al)_r+\frac{\bc}{24}}^{\al+\ft J+\ft \bJ}(F(\ft))$ if and only if it satisfies the following conditions:
\begin{align}
\begin{split}
\begin{cases}
L_n v=0=\Ld_n v & (n\geq 1),\\
J_n v=0=\bJ_n v & (n\geq 1),\\
G_r^\pm v=0=\bG_r^\pm v & (r \geq 0 \text{ and }r \in \Z),\\
%U_\la(G_{-\ft}^+) v=0=U_{\la}(\bG_{-\ft}^+)v,\\
L(0) v = hv, \quad \Ld(0) v = \h v,\\
J_0 v = (\al,J)_l v,\quad \bJ_0 v = (\al,\bJ)_r v.
\end{cases}
\end{split}
\label{eq_Ramon_primary}
\end{align}
Set
\begin{align*}
R_{h,\h}^\al =C_{h+\ft(J,\al)_l+\frac{c}{24},\h+\ft(\bJ,\al)_r+\frac{\bc}{24}}^{\al+\ft J+\ft \bJ}\left(F\left(\ft\right)\right),  
\end{align*}
whose element is called a {\it Ramond vacuum} in physics (see \cite{LVW}).

%By Proposition \ref{prop_chiral_twist}, $V(\frac{3}{c}J)$ is an $\exp(\pi J(0)) = (-1)^F$-twisted module,
%and thus, by Proposition \ref{prop_twist_susy}, the Ramon sector of the $N=2$ superconformal algebra acts on $V(\frac{3}{c}J)$.
%For $\al\in H$, let $R_h^\al(V)$ be the subspace of $V(\ft J)$ consisting of vectors $v$ such that:
%\begin{align*}
%\begin{cases}
%L_n v=0 & (n\geq 1),\\
%J_n v=0 & (n\geq 1),\\
%G_r^\pm v=0 & (r \geq 0 \text{ and }r \in \Z),\\
%L(0) v = hv,\\
%J(0)v = (\al,J)v,
%\end{cases}
%\end{align*}
%A vector in $C_h^\al(V)$ is called a {\it chiral primary vector} and a vector in $R_h^\al(V)$ is called a {\it Ramon vacuum} (see \cite{}).

As we have already emphasized, $D_{\la,\mu}:F(0,0) \rightarrow F(\la,\mu)$ in \eqref{eq_Dla} is just a linear map, 
but we see that it induces a linear isomorphism between $C_{h,\h}^\al(F(\la,\mu))$. In particular, the Ramond vacuum and the c-c primary vectors are isomorphic as linear spaces.
%
%Let $\al \in H$ and $v_\al  \in C_{h,\h}^\al(F(0))$.
%Then, we can regard $v_\al$ as an element in $\Om_F^\al$ and set
%\begin{align*}
%D_\la(v)=  v \otimes e(\al - \la (J+\bJ)) \in \Om_V^\al \otimes M(\al - \ft J) \subset V(\ft J).
%\end{align*}
\begin{lem}\label{lem_S_Ramon}
Let $\al \in H$ and $v_\al  \in C_{h,\h}^\al(F(0,0))$. Then, $D_{\la,\mu}(v_\al)$ is in $C_{h,\h}^{\al}(F(\la,\mu))$.
\end{lem}
\begin{proof}
We regard $v_\al$ as an element in $\Om_F^\al$.
Then, $D_{\la,\mu}(v_\al \otimes e_\al) = v_\al \otimes e_{\al - \la J-\mu \bJ}$.
It is clear that $L_n D_{\la,\mu}(v_\al)=0=\Ld_n= D_{\la,\mu}(v_\al)$ and $J_n D_{\la,\mu}(v_\al)=0 =\bJ_n D_{\la,\mu}(v_\al)$ for any $n\geq 1$ and $J_0 D_{\la,\mu}(v_\al) = ((\al,J)_l-\la (J,J)_l)D_{\la,\mu}(v_\al)$, $\bJ_0 D_{\la,\mu}(v_\al) = ((\al,\bJ)_r-\mu (\bJ,\bJ)_r)D_{\la,\mu}(v_\al)$ and
\begin{align*}
L(0) D_{\la,\mu}(v_\al) &=(L_\Om(0) v_\al) \otimes e(\al-\la J-\mu\bJ)) +  v_\al \otimes L_{H}(0) e(\al-\la J-\mu\bJ))\\
&= ((h-\frac{(p\al,p\al)_l}{2})+\frac{(p\al-\la J,p\al-\la J)_l}{2}) D_{\la,\mu}(v_\al)\\
&=(h-\la(\al,J)_l+ \frac{c\la^2}{6}) D_{\la,\mu}(v_\al)
\end{align*}
and similarly $\Ld(0)D_{\la,\mu}(v_\al) = (\h-\mu(\al,\bJ)_r+ \frac{\bc \mu^2}{6}) D_{\la,\mu}(v_\al)$.
Hence,
\begin{align*}
U_\la(J_0)D_{\la,\mu}(v_\al) = (J_0 +\la \frac{c}{3}) D_{\la,\mu}(v_\al) =(\al,J)_l v_\al
\end{align*}
and
\begin{align*}
U_\la(L_0)D_{\la,\mu}(v_\al) &= (L_0 +\la J_0+ \frac{c \la^2}{6}) D_{\la,\mu}(v_\al)\\
&=((h-\la(\al,J)_l+ \frac{c\la^2}{6})+\la( (\al,J)_l - \la \frac{c}{3}) + \frac{c \la^2}{6})D_{\la,\mu}(v_\al) = h D_{\la,\mu}(v).
\end{align*}
By $\tau^\pm \in \Om_F^{\pm \frac{3}{c}J}$ and $\btau^\pm \in \Om_F^{\pm \frac{3}{\bc}\bJ}$, we have
\begin{align*}
\tY_\sft(\tau^\pm,\uz) D_{\la,\mu}(v_\al)= \hY_\Om(\tau^\pm,\uz)v_\al \otimes \hY_\std(e(\pm \frac{3}{c}J),z) e(\al-\la J-\mu\bJ)).
\end{align*}
Hence, we have
\begin{align*}
\tY_\sft(\tau^+,\uz) D_{\la,\mu}(v_\al) &= \hY_\Om(\tau^+,z)v \otimes z^{(\frac{3}{c}J,\al-\la J)_l} E^-(\frac{3}{c}J,z) e(\al+\frac{3}{c}J-\la J-\mu\bJ))\\
&= z^{-(\frac{3}{c}J,\la J)_l}D_{\la,\mu}\Bigl(Y_{F}(\tau^+,\uz)v_\al \Bigr)\\
&= z^{-\la} D_{\la,\mu}\Bigl(\sum_{r \leq -\frac{3}{2}} (G_r^{+,0} v_\al) z^{-\frac{3}{2}-r} \Bigr)\\
&=D_{\la,\mu}(G_{-\frac{3}{2}}^{+,0} v_\al) z^{-\la}+ D_{\la,\mu}(G_{-\frac{5}{2}}^{+,0} v_\al) z^{1-\la}+ \cdots,
\end{align*}
where $Y_F(\bullet,\uz)$ is the original vertex operator on $F(0,0)$
and $G_r^{+,0}$ is the original action of $\g_0^{N=2}$ on $F(0,0)$.
Since the action of $U_{\la}(G_r^+) = G_{r+\la}^+$ on $F(\la,\mu)$ is given by the coefficient of $z^{-\frac{3}{2}-r-\la}$ in $\tY_\sft(\tau^+,\uz)$ (see Proposition \ref{prop_twist_susy}), $U_\la(G_r^+) D_{\la,\mu}(v)=0$ for any $r \geq -\ft$.
Similarly, we have
\begin{align*}
Y_\sft(\tau^-,\uz) D_{\la,\mu}(v_\al) &= \hY_\Om(\tau^-,\uz)v_\al \otimes z^{(-\frac{3}{c}J,\al-\la J)_l} E^-(-\frac{3}{c}J,z) e(\al-\la J-\mu\bJ))\\
&= z^{\la} D_{\la,\mu}\Bigl(Y_{F}(\tau^-,\uz)v_\al \Bigr)\\
&= z^{\la} D_{\la,\mu}\Bigl(\sum_{r \leq -\frac{1}{2}} (G_r^{-,0} v_\al) z^{-\frac{3}{2}-r} \Bigr)=
D_{\la,\mu}(G_{-\frac{1}{2}}^{-,0} v_\al) z^{\la-1}+ D_{\la,\mu}(G_{-\frac{3}{2}}^{-,0} v_\al) z^{\la}+\cdots.
\end{align*}
Hence, $U_\la(G_r^-)D_{\la,\mu}(v_\al)=0$ for any $r \geq \ft$.
\end{proof}

Conversely, let $\al \in H$ and $v \in C_{h,\h}^{\al}(F(\la,\mu))$.
Then, there exists unique $T_{\la,\mu}(v) \in \Om_F^\al$ such that $v = T_{\la,\mu}(v) \otimes e_{\al-\la J -\mu \bJ)} \in \Om_F^\al \otimes M_H(\al-\la J-\mu \bJ)$.
\begin{lem}\label{lem_T_Ramon}
The vector $T_{\la,\mu} (v)$ is in $C_{h,\h}^{\al}(F)$.
\end{lem}
\begin{proof}
The proof is the same as in Lemma \ref{lem_S_Ramon}. 
\begin{align*}
Y_{F}(\tau^\pm,\uz) T_{\la,\mu}(v) &= \hY_\Om(\tau^+,z) T_{\la,\mu}(v) \otimes Y_\std(e_{\pm \frac{3}{c}J},\uz)e_{\al}\\
&=z^{\pm (\frac{3}{c}J, \la J)} D_{\la,\mu}^{-1}(\hY_\sft(\tau^\pm,\uz) v)\\
&=
\begin{cases}
z^\la D_{\la,\mu}^{-1}\Bigl(\sum_{r \leq \la-\frac{3}{2}} G_r^{+}v z^{-\frac{3}{2}-r}\Bigr),\\
z^{-\la} D_{\la,\mu}^{-1}\Bigl(\sum_{r \leq \la-\ft} G_r^{-}v z^{-\frac{3}{2}-r}\Bigr),
\end{cases}
\\
&= \begin{cases}
D_{\la,\mu}^{-1}(G_{\la-\frac{3}{2}}^+ v) + D_{\la,\mu}^{-1}(G_{\la-\frac{5}{2}}^+ v) z+\dots,\\
D_{\la,\mu}^{-1}(G_{\la-\frac{1}{2}}^- v) z^{-1}+ D_{\la,\mu}^{-1}(G_{\la-\frac{3}{2}}^+ v) +\dots.
\end{cases}
\end{align*}
Hence, the assertion holds.
\end{proof}

By Lemma \ref{lem_S_Ramon} and \ref{lem_T_Ramon}, we have:
\begin{prop}\label{prop_chiral_bijection}
For any $\la,\mu \in \R$, $\al \in H$ and $h,\h \in \R$, the linear maps $D_{\la,\mu}: C_{h,\h}^\al(F) \rightarrow C_{h,\h}^{\al}(F(\la,\mu))$ and $C_{h,\h}^{\al}(F(\la,\mu)) \rightarrow C_{h,\h}^\al(F)$ are mutually inverse.
\end{prop}

%よって頂点代数の意味でのspectral flow $F(\la)$を考えると、$N=2$ superconformal algebra の作用は一般に頂点作用素$E(\al,\uz)$ \eqref{}による変更を受けて複雑に変化するが、chiral primary 上に制限すると、この変化は小さく Lie algebra の意味での spectral flow と整合性がある。
By the spectral flow, the action of $N=(2,2)$ superconformal algebra generally undergoes a complicated change due to the vertex operator $E(\al,\uz)$ \eqref{eq_lattice_vertex1}. However, when restricted on primary vectors, this change is consistent with the spectral flow in the sense of Lie algebra.

An important consequence of Proposition \ref{prop_chiral_bijection} is the following:
\begin{cor}\label{cor_Ramon_weight}
Let $F$ be a unitary $N=(2,2)$ full vertex operator superalgebra.
% Set
%\begin{align*}
%F_{h,\h}^\al\left(\ft,\ft\right) = \{v\in F\left(\ft,\ft\right)\mid L(0)v = hv,\Ld(0)v =\h v, J_0v=(\al,J)_lv, \bJ_0v =(\al,\bJ)_r v\}.
%\end{align*}
If $F_{h,\h}^\al(\ft,\ft)\neq 0$, then $h \geq \frac{c}{24}$ and $\h \geq \frac{\bc}{24}$.
Moreover, $v\in R_{h,\h}^\al(F)$ if and only if $(h,\h)=(\frac{c}{24},\frac{\bc}{24})$.
% and assume $R_{h,\h}^\al(F) \neq 0$ for some $h,\h\in \R$ and $\al\in H$. Then, $(h,h)=(\frac{c}{24},\frac{\bc}{24})$.
\end{cor}
\begin{proof}
Note that  $D_{\ft,\ft}(a) \in F_{h-\frac{(\al,J)_l}{2}+\frac{c}{24},\h-\frac{(\al,\bJ)_r}{2}+\frac{\bc}{24}}^{\al-\ft(J+\bJ)}(\ft,\ft)$ for any $a \in F_{h,\h}^\al$.
Since $F$ is unitary, $F_{h,\h}^\al \neq 0$ implies $h \geq \frac{(\al,J)_l}{2}$ and $\h \geq \frac{(\al,\bJ)_r}{2}$
by Lemma \ref{lem_NS_inequality}.
The assertion follows from Proposition \ref{prop_chiral_bijection} and \eqref{eq_Ramon_primary}.
\end{proof}

%\begin{thm}\label{thm_}
%Uniqueness of the spin structure for holomorphic vertex operator superalgebra
%\end{thm}
%

%
%We will show the following theorem:
%\begin{thm}\label{thm_chiral_spectral}
%For any $\la \in \R$, $V(\la) = V_+(\la) \oplus V_-(\la)$ is a $\exp(\pi i \la J_0)$-twisted module of $V$,
%and $V(\la)$ is isomorphic to $V(\la+1)$ as $\exp(\pi i \la J_0)$-twisted $V$-modules.
%Moreover, $\{V_\pm(\la)\}_{\la \in S^1}$ naturally inherits a spectral vertex operator algebra structure.
%\end{thm}
%
%Set 
%\begin{align*}
%V_\pm^{\NS} = V_\pm \text{ and } V_\pm^{\Ra} = V_\pm(\frac{1}{2}).
%\end{align*}
%By Lemma \ref{}, we have:
%\begin{cor}\label{cor_spin}
%The $\Z_2 \times \Z_2$-graded vector space $(V_\pm^{\NS},V_\pm^{\Ra})$ inherits a spin vertex operator algebra structure as a subalgebra of the spectral VOA $\{V_\pm(\la)\}_{\la \in S^1}$.
%\end{cor}
%The following lemma is important:

Assume that $F$ is a unitary $N=(2,2)$ full VOA
and $\ep \in C^{J+\bJ}(F)$ is a volume form with $\langle \ep,\ep \rangle=1$.
Let $f_+:F(1)\rightarrow F(0)$ be the $F$-module isomorphism.
Let $a \in C_{\frac{(\al,J)_l}{2},\frac{(\al,\bJ)_r}{2}}^\al(F)_{cc}$ be a c-c primary vector.
By Lemma \ref{lem_S_Ramon}, $D_{1,1}a \in C_{\frac{(\al,J)_l}{2},\frac{(\al,\bJ)_r}{2}}^\al(F(1,1))$, which satisfies the following conditions by \eqref{eq_twist_primary}:
\begin{align}
\begin{split}
\begin{cases}
L_n D_{1,1}a=0=\Ld_n D_{1,1}a & (n\geq 1),\\
J_n D_{1,1}a=0=\bJ_n D_{1,1}a & (n\geq 1),\\
G_r^\pm D_{1,1}a=0=\bG_r^\pm D_{1,1}a & (r \geq \ft \text{ and }r \in \ft + \Z),\\
G_{-\ft}^- D_{1,1}a=0=\bG_{-\ft}^-D_{1,1}a,\\
L(0) D_{1,1}a = \frac{(J-\al,J)_l}{2}D_{1,1}a, \quad \Ld(0) D_{1,1}a =
\frac{(\bJ-\al,\bJ)_r}{2}D_{1,1}a,\\
J_0 D_{1,1}a = (\al-J-\bJ,J)_l D_{1,1}a,\quad \bJ_0 D_{1,1}a = (\al-J-\bJ,\bJ)_r D_{1,1}a.
\end{cases}
\end{split}
\label{eq_shift_primary}
\end{align}
Since $f_+:F(1,1) \rightarrow F(0,0)$ is an $F$-module homomorphism, the image of $D_{1,1}v$ by $f_+$ is a-a primary vectors. Thus, we have:
\begin{align*}
f_+ \circ D_{1,1}: C_{h,\h}^\al(F)_{cc} \rightarrow C_{h,\h}^{\al-J-\bJ}(F)_{aa}.
\end{align*}
This map can be written more explicitly as follows:

Since we normalize $\ep$ by $\langle \ep, \ep \rangle =1$, by \eqref{eq_star_normalize},
$\ep(\frac{c}{3}-1,\frac{\bc}{3}-1)\phi(\ep) = (-1)^{s_F+2s_F^2}$ with $s_F = \frac{c}{6}-\frac{\bc}{6} \in \ft\Z$.
Hence, by the proof of Theorem \ref{thm_chiral_period}, we have
\begin{align*}
f_+(D_{1,1} a) &= (\Psi_{-1} \otimes \id ) (a \otimes e_{\al-J-\bJ})\\
&=(-1)^{s_F+2s_F^2} (a (-1,-1)_\Om \phi(\ep)) \otimes e_{\al-J-\bJ} \in F(0),
\end{align*}
where $a (-1,-1)_\Om \phi(\ep)$ is the product of the generalized full vertex algebra $\Om_F$.
We will rewrite $(a (-1,-1)_\Om \phi(\ep)) \otimes e_{\al-J-\bJ}$ by the product of full vertex algebra.
Namely, by
\begin{align}
\begin{split}
Y_F(a,\uz) \phi(\ep) &= Y_F(a \otimes e_\al,\uz) \phi(\ep)\otimes e_{-J-\bJ}\\
& = (\hY_\Om(a,\uz) \phi(\ep)) \otimes (\hY_\std(e_\al,\uz)e_{-J-\bJ})\\
&=z^{-(J,\al)_l}\z^{-(\bJ,\al)_r} (\hY_\Om(a,\uz) \phi(\ep)) \otimes (1+p\al z+\p\al\z+\dots) e_{\al-J-\bJ},
\end{split}
\label{eq_relate_full_gen}
\end{align}
we have $f_+(D_{1,1} a) = (-1)^{s_F+2s_F^2} a((J,\al)_l-1,(J,\al)_r-1)\phi(\ep)$,
which is the product of full vertex algebra $F$.
Hence, we have:
\begin{prop}\label{prop_shift_primary}
Let $F$ be a unitary $N=(2,2)$ full VOA
and $\ep \in C^{J+\bJ}(F)$ a volume form with $\langle \ep,\ep \rangle=1$.
Then, we have the linear isomorphism:
\begin{align*}
C_{h,\h}^\al(F)_{cc} \rightarrow C_{h,\h}^{\al-J-\bJ}(F)_{aa},\quad a \mapsto a((J,\al)_l-1,(J,\al)_r-1)\phi(\ep).
\end{align*}
\end{prop}

Similarly, we have:
\begin{prop}\label{prop_shift_serre}
Let $F$ be a unitary $N=(2,2)$ full VOA
and $\eta \in C^{J}(F)$ a holomorphic volume form with $\langle \eta,\eta \rangle=1$.
Then, we have the linear isomorphism:
\begin{align*}
C_{h,\h}^\al(F)_{cc} \rightarrow C_{h,\h}^{\al-J}(F)_{ac},\quad a \mapsto a((J,\al)_l-1,-1)\phi(\eta).
\end{align*}
\end{prop}
\begin{proof}
Since
\begin{align}
\begin{split}
Y_F(a,\uz) \phi(\eta) &= Y_F(a \otimes e_\al,\uz) (\phi(\eta)\otimes e_{-J})\\
& = (\hY_\Om(a,\uz) \phi(\eta)) \otimes (\hY_\std(e_\al,\uz)e_{-J})\\
&=z^{-(J,\al)_l}(\hY_\Om(a,\uz) \phi(\eta)) \otimes (1+p\al z+\p\al\z+\dots) e_{\al-J},
\end{split}
\label{eq_relate_full_eta}
\end{align}
$a((J,\al)_l-1,-1)\phi(\eta) = (a(-1,-1)_\Om \phi(\eta)) \otimes e_{\al-J}$. Hence, the assertion follows from Proposition \ref{prop_period_cc}.
\end{proof}

\subsection{Construction of Poincar\'e duality and T-duality}
\label{sec_Hodge_Serre}
Let $F$ be an $N=(2,2)$ unitary full vertex operator superalgebra with the anti-linear involution $\phi$.
The purpose of this section is to show the following propositions:
\begin{prop}\label{prop_star_primary}
\begin{enumerate}
\item
Assume that there is a non-zero vector $\ep \in F_{\frac{c}{6},\frac{\bc}{6}}^{J+\bJ}$ with $\langle \ep,\ep \rangle =1$. 
Let $*:F \rightarrow F$ be the anti-linear map  (\eqref{eq_full_star_def}) given by
\begin{align*}
*a = (-1)^{s(\al)+2s(\al)^2}  \phi(a)((\al,J)_l-1,(\al,\bJ)_r-1)\ep,
\end{align*}
where $s(\al) = \frac{(\al,J)_l}{2}-  \frac{(\al,\bJ)_r}{2} \in \ft\Z$.
Then, for any $a\in C^\al(F)_{cc}$, $*a$ is in $C^{J+\bJ-\al}(F)_{cc}$ and $**a = (-1)^{|a|+|a||\ep|} a$.
In particular, this gives a linear isomorphism:
\begin{align*}
H^{p,q}(F,d_B) \rightarrow H^{d-p,d-q}(F,d_B)
\end{align*}
with respect to the grading in Definition \ref{def_hodge_number}.
\item
Assume that there is a non-zero vector $\eta \in F_{\frac{c}{6},0}^{J}$ with $\langle \eta,\eta \rangle =1$. 
Let $*_c:F \rightarrow F$ be the anti-linear map  (\eqref{eq_full_star_def}) given by
\begin{align*}
*_c a =\phi(a)((\al,J)_l-1,-1)\eta.
\end{align*}
Then, for any $a\in C^\al(F)_{cc}$, $*_c a$ is in $C^{J-\al}(F)_{ca}$ and
for any $b\in C^\be(F)_{ca}$, $*_c b$ is in $C^{J-\be}(F)_{cc}$. Moreover, $*_c*_c a = (-1)^{t_F+2t_F^2} a$ and $*_c *_c b= (-1)^{t_F+2t_F^2}b$ with $t_F = c/6\in\ft\Z$. In particular, this gives a linear isomorphism:
\begin{align*}
H^{p,q}(F,d_B) \rightarrow H^{d-p,q}(F,d_A)
\end{align*}
with respect to the grading in Definition \ref{def_hodge_number}.
\end{enumerate}
\end{prop}

The following lemma follows from the definition:
\begin{lem}\label{lem_anti_flip}
The anti-linear involution $\phi$ gives an isomorphism from $C_{\frac{(\al,J)_l}{2},\frac{(\al,J)_r}{2}}^\al(F)_{cc}$ onto ${C_{\frac{(\al,J)_l}{2},\frac{(\al,J)_r}{2}}^{-\al}}(F)_{aa}$
and $C_{\frac{-(\al,J)_l}{2},\frac{(\al,J)_r}{2}}^\al(F)_{ac}$ onto ${C_{\frac{-(\al,J)_l}{2},\frac{(\al,J)_r}{2}}^{-\al}}(F)_{ca}$, respectively.
\end{lem}

Now, the proof of Proposition \ref{prop_star_primary} (1) can be seen from the fact that $*$-operator is the composition of the following isomorphisms:
% in Proposition \ref{prop_chiral_bijection}, Proposition \ref{prop_shift_primary} and Lemma \ref{lem_anti_flip}:
%さて命題\ref{prop_star_primary}の証明は$*$-operator が次の同型の合成でかけることを見る:
\begin{align}
\begin{split}
C^\al&(F)_{cc}\\
&\downarrow {}_{f_+ \circ D_{1,1}\text{ in Proposition \ref{prop_shift_primary}}}\\
%\text{twisted c-c primary in } &F(1,1) \text{ with weight $\al$}\\
%&\downarrow {}_{=}\\
C^{\al-J-\bJ}&(F)_{aa}\\
%\text{a-a primary in } &F(1,1) \text{ with weight $\al-J-\bJ$}\\
&\downarrow {}_{\phi \text{ in Lemma \ref{lem_anti_flip}}}\\
%\text{a-a primary in }&F(0,0)\text{ with weight $\al-J-\bJ$}\\
%&\downarrow {}_\phi\\
C^{J+\bJ-\al}&(F)_{cc}.
%\text{c-c in } &F(0,0) \text{ with weight $J+\bJ-\al$}.
\end{split}
\label{eq_comp_star}
\end{align}
Hence, it suffices to show that $**a=(-1)^{|a|+|a||\ep|}a$.

Since $s(J+\bJ-\al)=\frac{(J-\al,J)_l}{2} - \frac{(\bJ-\al,\bJ)_r}{2} =s_F - s(\al)$ with $s_F = \frac{c}{6}-\frac{\bc}{6}$,
\begin{align*}
(-1)^{s(\al)+2s(\al)^2}(-1)^{s(J+\bJ-\al)+2s(J+\bJ-\al)^2}&= (-1)^{s_F+2s_F^2+4s(\al)^2-4s(\al)s_F}=(-1)^{s_F+2s_F^2} (-1)^{|a|+|a||\ep|},
\end{align*}
where $|a|$ and $\ep$ are the parities.

\begin{align}
\begin{split}
**a &=(-1)^{s(\al)+2s(\al)^2}*(\phi(a)((\al,J)_l-1,(\al,\bJ)_r-1)\ep)\\
&=(-1)^{s_F+2s_F^2} (-1)^{|a|+|a||\ep|}\phi\left(\phi(a)((\al,J)_l-1,(\al,\bJ)_r-1)\ep)\right) ((J-\al,J)_l-1,(\bJ-\al,\bJ)_r-1)\ep)\\
&=(-1)^{s_F+2s_F^2} (-1)^{|a|+|a||\ep|}\left(a((\al,J)_l-1,(\al,\bJ)_r-1)\phi(\ep))\right) ((J-\al,J)_l-1,(\bJ-\al,\bJ)_r-1)\ep).
\end{split}
\label{eq_star_star}
\end{align}

\begin{lem}\label{lem_star_calc}
Let $a \in \Om_F^\al$ and $b \in \Om_F^\be$ such that $L(-1)v= p\be(-1)v$ and $\Ld(-1)v=\p\be(-1)v$.
Then, $a(-(\al,\be)_l-1,-(\al,\be)_r-1)v \in \Om_F^{\al+\be}$ and
\begin{align*}
a(-(\al,\be)_l-1,-(\al,\be)_r-1)v = (a(-1,-1)_{\Om}v) \otimes e_{\al+\be}
\end{align*}
as an element of the tensor product generalized full vertex algebra $\Om_F \otimes G_{H,p}$.
\end{lem}
\begin{proof}
Since
\begin{align*}
Y(a,\uz)v &= (-1)^{|a||v|}\exp(L(-1)z+\Ld(-1)\z)Y(v,-\uz)a\\
&=(-1)^{|a||v|}\exp(L(-1)z+\Ld(-1)\z)(Y_\Om(v,-\uz)\otimes Y_\std(e_\be,-\uz))(a \otimes e_\al)\\
&= (-1)^{|a||v|}(-1)^{(\al,\be)_c} z^{(\al,\be)_l}\z^{(\al,\be)_r}
\exp(L(-1)z+\Ld(-1)\z) (v\cdot_{\Om}a) \otimes E^-(\be,-\uz)e_\al,
\end{align*}
by Lemma \ref{lem_vacuum_like}, we have
\begin{align*}
a(-(\al,\be)_l-1,-(\al,\be)_r-1)v 
&= (-1)^{|a||v|}(-1)^{(\al,\be)_c} (v\cdot_{\Om}a) \otimes e_{\al+\be}\\
&= (a(-1,-1)_{\Om}v) \otimes e_{\al+\be}.
\end{align*}
\end{proof}

Hence, by Lemma \ref{lem_star_calc} and Proposition \ref{prop_vac_right},
\begin{align*}
\eqref{eq_star_star} &=(-1)^{s_F+2s_F^2} (-1)^{|a|+|a||\ep|}
\left(\left(a(-1,-1)_\Om \phi(\ep)\right)(-1,-1)\ep \right) \otimes e_{\al}\\
&=(-1)^{s_F+2s_F^2} (-1)^{|a|+|a||\ep|} a(-1,-1)_\Om \left( \phi(\ep)(-1,-1)\ep \right) \otimes e_{\al}\\
&=(-1)^{|a|+|a||\ep|}(a(-1,-1)_\Om \vac) \otimes e_{\al}=(-1)^{|a|+|a||\ep|}a.
\end{align*}
Hence, Proposition \ref{prop_star_primary} (1) holds.
We will show (2). Similarly to \eqref{eq_comp_star}, we have:
%\begin{align}
%\begin{split}
%C^\al&(F)_{cc}\\
%&\downarrow {}_{f_{1,0} \circ D_{1,0}\text{ in Proposition \ref{prop_shift_serre}}}\\
%C^{\al-J}&(F)_{ac}\\
%&\downarrow {}_{\phi \text{ in Lemma \ref{lem_anti_flip}}}\\
%C^{J-\al}&(F)_{ca}.
%\end{split}
%\label{eq_comp_star2}
%\end{align}
\begin{align}
\begin{split}
\begin{array}{ll}
C^\al(F)_{cc}& C^\be(F)_{ca} \\
\downarrow {}_{f_{1,0} \circ D_{1,0}\text{ in Proposition \ref{prop_shift_serre}}} & \downarrow {}_{f_{1,0} \circ D_{1,0}}\\
C^{\al-J}(F)_{ac} & C^{\be-J}(F)_{aa}\\
\downarrow {}_{\phi \text{ in Lemma \ref{lem_anti_flip}}} & \downarrow {}_{\phi} \\
C^{J-\al}(F)_{ca} & C^{J-\be}(F)_{cc}.
\end{array}
\end{split}
\label{eq_comp_star2}
\end{align}
Let $a \in C^\al(F)_{cc}$. 
Since $1= \langle \eta, \eta \rangle = (-1)^{t_F+2t_F^2} \langle \phi(\eta)(\frac{c}{3}-1,1)\eta  ,\vac \rangle$ with $t_F = \frac{c}{6}$,
by Lemma \ref{lem_star_calc}, similarly to \eqref{eq_star_star}, we have:
\begin{align*}
*_c*_c a &= \left((a(\al,J)_l-1,-1)\phi(\eta)\right)((J-\al,J)_l-1,-1)\eta \\
&= (a(-1,-1)_\Om (\phi(\eta)(-1,-1)_\Om \eta))\otimes e_{\al}\\
&= (-1)^{t_F+2t_F^2}a.
\end{align*}
Hence, (2) holds.

Now, we prove Theorem \ref{thm_full_frob} and Theorem \ref{thm_chiral_star}.
\begin{proof}[proof of Theorem \ref{thm_full_frob}]
All except (6) immediately follow from previous discussions. We will show (6). Let $a \in C^\al(F)$ and $b\in C^{J+\bJ-\al}(F)$.
By \eqref{eq_unitary_inv}, we have:
\begin{align*}
(a,b)_\ep &= \langle \ep, a (-1,-1) b \rangle
= (-1)^{(h_a-\h_a)+2(h_a-\h_a)^2}
\langle \phi(a)((\al,J)_l-1,(\al,\bJ)_r-1)\ep, b \rangle\\
&=\langle *a , b \rangle.
\end{align*}
Hence, the assertion holds.
\end{proof}

\begin{proof}[proof of Theorem \ref{thm_chiral_star}]
(1) follows from Proposition \ref{prop_period_cc}
an (2) from Proposition \ref{prop_star_primary}.
We claim that for any $r\in\R$
\begin{align*}
C_{\frac{cr}{6},0}^{rJ}(F)_{cc} = C_{\frac{cr}{6},0}^{rJ}(F)_{ca},
\end{align*}
which follows from Lemma \ref{lem_chiral_volume}.
Hence, $*_c$ gives an isomorphism:
\begin{align*}
H^{p,0}(F,d_B) = C_{\frac{p}{2},0}^{\frac{3p}{c}J}(F)_{cc} \cong  C_{\frac{c}{6}-\frac{p}{2},0}^{\frac{c-3p}{c}J}(F)_{ca}= C_{\frac{c}{6}-\frac{p}{2},0}^{\frac{c-3p}{c}J}(F)_{cc} = H^{d-p,0}(F,d_B).
\end{align*}
\end{proof}

\subsection{Holomorphic twist and Witten genus}
\label{sec_genus}
Let $F$ be a unitary $N=(2,2)$ full vertex operator superalgebra.
In this section we introduce an invariant called the Witten genus of $N=(2,2)$ SCFT.
%この章では$F$の楕円種数と呼ばれる不変量を定義する。
Recall that the spectral flow $F(\ft)$ of $F$ is a module of the Ramond algebras $\g_\ft^{N=2}\oplus \g_\ft^{N=2}$, which is called the {\bf Ramond sector} in physics.
A good mathematical feature of the Ramond sector is the existence of the conformal weight $(0,0)$ operators $G_0^\pm$ and $\bG_0^\pm$.
Since $(G_0^+)^2 =0$, we have the following chain complex:
\begin{align}
\cdots \overset{\bG_0^+}{\rightarrow} F_{h,\h}^{\al-\bJ}\left(\ft\right) \overset{\bG_0^+}{\rightarrow} F_{h,\h}^\al\left(\ft\right) \overset{\bG_0^+}{\rightarrow} F_{h,\h}^{\al+\bJ}\left(\ft\right) \overset{\bG_0^+}{\rightarrow}\cdots.
\label{eq_chain_complex_G}
\end{align}
Note that by Corollary \ref{cor_Ramon_weight}, $F_{h,\h}^\al\left(\ft\right)=0$ unless $h \geq \frac{c}{24}$ and $\h \geq \frac{\bc}{24}$.
Since 
\begin{align*}
\bG_0^+ \bG_0^- v = [\bG_0^+, \bG_0^-]_+ v = (L(0) -\frac{\bc}{24})v
\end{align*}
for any $v \in \ker \bG_0^+$, we have:
\begin{lem}\label{lem_exact_G}
If $\h > \frac{\bc}{24}$, then \eqref{eq_chain_complex_G} is exact.
\end{lem}

Recall that $F= \bigoplus_{\al \in H} \Om_F^\al \otimes M_{H,p}(\al)$ and
\begin{align*}
F^0 &= \bigoplus_{\substack{\al \in H \\ (\al,J)_l-(\al,\bJ)_r \in 2\Z}} \Om_F^\al \otimes M_{H,p}(\al)\\
F^1 &= \bigoplus_{\substack{\al \in H \\ (\al,J)_l-(\al,\bJ)_r \in 1+2\Z}} \Om_F^\al \otimes M_{H,p}(\al).
\end{align*}
Then, define $\Z_2$-grading on $F(\ft)$ as follows:
\begin{align*}
F\left(\ft\right)^0 &= \bigoplus_{\substack{\al \in H \\ (\al,J)_l-(\al,\bJ)_r \in 2\Z}} \Om_F^\al \otimes M_{H,p}(\al-\ft(J+\bJ))\\
F\left(\ft\right)^1 &= \bigoplus_{\substack{\al \in H \\ (\al,J)_l-(\al,\bJ)_r \in 1+2\Z}} \Om_F^\al \otimes M_{H,p}(\al-\ft(J+\bJ))
\end{align*}
and the linear map $(-1)^F:F\left(\ft\right) \rightarrow F\left(\ft\right)$ by the $\Z_2$-grading.
With respect to this $\Z_2$-grading, $\bG_0^+$ is an odd differential.
By Proposition \ref{prop_flow_dim}, we have:
\begin{align}
\dim F_{h+\frac{c}{24},\h+\frac{\bc}{24}}^{\al-\ft(J+\bJ)}\left(\ft\right)= \dim F_{h+\ft(\al,J)_l,\h+\ft(\al,\bJ)_r}^{\al}.
\label{eq_Ramon_dim2}
\end{align}

In physics, the Witten index is defined by 
\begin{align}
\begin{split}
&\ind(F)(y,q,\bar{q})\\
&=
\mathrm{tr}|_{F(\ft)} ((-1)^F y^{J(0)}q^{L(0)-\frac{c}{24}}\bar{q}^{\Ld(0)-\frac{\bc}{24}})\\
&=\sum_{h,\h\in\R,\al\in H} \dim (F^0)_{h,\h}^\al\left(\ft\right)y^{(\al,J)_l}q^{h-\frac{c}{24}}\bar{q}^{\h-\frac{\bc}{24}} - \sum_{h,\h\in\R,\al\in H} \dim (F^1)_{h,\h}^\al\left(\ft\right)y^{(\al,J)_l}q^{h-\frac{c}{24}}\bar{q}^{\h-\frac{\bc}{24}}.
\end{split}
\label{eq_ind_phys}
\end{align}
In order to see that \eqref{eq_ind_phys} is well-defined, i.e., each coefficient of $y$ and $q$ is a finite sum, it suffices to show that:
\begin{align*}
\sum_{s \in \R}
 \dim F_{h,\h}^{\al + s\bJ}\left(\ft\right) <\infty,
\end{align*}
which follows from the following lemma and \eqref{eq_Ramon_dim2}:
\begin{lem}\label{lem_ind_well}
Let $\al \in H$, $a,b \in \R_{\geq 0}$ and $v \in F_{a+\frac{(\al,J)_l}{2},b+ \frac{(\al,\bJ)_r}{2}}^\al$ be a non-zero vector. Then,
\begin{align*}
\frac{1}{a+1}(2a+\frac{c}{3}((a+2)^2-\frac{1}{4})) &\geq (\al,J)_l\\
\frac{1}{b+1}(2b+\frac{\bc}{3}((b+2)^2-\frac{1}{4})) &\geq (\al,\bJ)_r.
\end{align*}
In particular, for any $a,b \in \R_{\geq 0}$,
\begin{align*}
\sum_{\al \in H} \dim F_{a+\frac{(\al,J)_l}{2},b+\frac{(\al,\bJ)_r}{2}}^\al < \infty.
\end{align*}
\end{lem}
\begin{proof}
$v \in F_{a+\frac{(\al,J)_l}{2},b+ \frac{(\al,\bJ)_r}{2}}^\al$ with $a,b \in \R_{\geq 0}$. Let $r \in \ft+\Z$ with $r > a+1$.
Then, $G_{r}^- v \in F_{a-r+\frac{(\al,J)_l}{2},b+ \frac{(\al,\bJ)_r}{2}}^{\al-\frac{3}{c}J}=0$.
Hence, we have:
\begin{align*}
0 &\leq \langle G_{-r}^+ v, G_{-r}^+ v \rangle=
\langle v, G_r^-G_{-r}^+ v \rangle \\
&= \langle v, [G_{-r}^+,G_r^-]_+ v \rangle =
 \langle v, (L(0)-rJ(0)+\frac{c}{6}(r^2-\frac{1}{4})) v \rangle\\
 &=(a+\frac{(\al,J)_l}{2} - \frac{r(\al,J)}{2}+\frac{c}{6}(r^2-\frac{1}{4})) \langle v,v\rangle
\end{align*}
and $\frac{1}{r-1}(2a+\frac{c}{3}(r^2-\frac{1}{4})) \geq (\al,J)_l$. Hence, the assertion holds.
\end{proof}
Hence, \eqref{eq_ind_phys} is well-defined and by Lemma \ref{lem_exact_G}, the terms in \eqref{eq_ind_phys} should all cancel each other except for $\h=\frac{\bc}{24}$. Thus we obtain the following definition:
\begin{dfn}\label{def_index}
The Witten index of a unitary $N=(2,2)$ full vertex operator superalgebra $F$ is the formal power series in $y,q$ given by
\begin{align*}
\ind_F(y,q)
&=\sum_{h\in\R,\al\in H} \dim (F^0)_{h,\frac{\bc}{24}}^\al\left(\ft\right)y^{(\al,J)_l}q^{h-\frac{c}{24}} - \sum_{h \in\R,\al\in H} \dim (F^1)_{h,\frac{\bc}{24}}^\al\left(\ft\right)y^{(\al,J)_l}q^{h-\frac{c}{24}}.
\end{align*} 
\end{dfn}
Note that by \eqref{eq_Ramon_dim2}, we have:
\begin{align}
\begin{split}
\ind_F(y,q)&= \sum_{h\in \R,\al \in H} (-1)^{(\al,J)_l-(\al,\bJ)_r} \dim F_{h+\frac{c}{24},\frac{\bc}{24}}^{\al-\ft(J+\bJ)}\left(\ft\right) y^{(J,\al)_l-\frac{c}{6}} q^{h}\\
&= y^{-\frac{c}{6}} \sum_{h\geq 0, \al \in H} (-1)^{(\al,J)_l-(\al,\bJ)_r} \dim F_{h+\frac{(\al,J)_l}{2},\frac{(\al,\bJ)_r}{2}}^\al y^{(J,\al)_l} q^{h}.
\end{split} \label{eq_ind_q_only}
\end{align}
By Lemma \ref{lem_primary_inequality}, $\al$ in the sum \eqref{eq_ind_q_only} runs through $\frac{\bc}{3} \geq (\al,\bJ)_r \geq 0$. 
%Hence, This definition is well-defined (each coefficient is finite) by (FO2).
By (FO1) and Definition \ref{def_N22} (3), $h$ runs through $\Z_{\geq 0}$ in \eqref{eq_ind_q_only}.
Thus, we have
\begin{align*}
\ind_F(y,q) \in \Z[[q]][y^\R].
\end{align*}
%Note that \eqref{eq_ind_q_only} only involves positive real powers of $q$.
Denote the constant term for $q$ in the formal power series $\ind(F)(y,q)$ by $\ind(F)(y,0)$.

%By definition of the spectral flow $F\left(\ft\right) = \bigoplus_{\al \in H} \Om_F^\al \otimes M_{H,p}(\al-\ft (J+\bJ))$
%and Corollary \ref{cor_Ramon_weight},
%\begin{align*}
%\ind(F)(y,q) = (\Pi_{n \geq 1} (1-q^n)^{-1}) \sum_{\al, h}(-1)^{(\al,J)_l-(\al,\bJ)_r} \dim (\Om_F)_{h,\frac{(\bJ-\al,\al)_r}{2}}^\al y^{(J,\al)_l-\frac{c}{6}}
%q^{h+(\al-\ft J,\al-\ft J)/2-c/24}
%\end{align*}
The following proposition is obvious:
\begin{prop}\label{prop_Witten_Hirz}
The $q \to 0$ limit of the Witten index satisfies
\begin{align*}
\ind_F(y,0) =y^{-\frac{c}{6}}\sum_{\al \in H} (-1)^{(\al,J)_l-(\al,\bJ)_r} y^{(J,\al)} C^{\al}(F)_{cc}=y^{-\frac{c}{6}}\chi_y^B(F),
\end{align*}
where $\chi_y^B(F)$ is the Hirzebruch genus of $F$ (Definition \ref{def_hodge_number}).
%and the following properties hold:
%\begin{description}
%\item[Euler number)]
%\begin{align*}
%\ind(F)(1,0) = \chi(F) = \sum_{\al \in H} (-1)^{(\al,J)_l-(\al,\bJ)_r}\dim H^{\al}(F)
%\end{align*}
%\item[Todd number]
%\begin{align*}
%\ind(F)(0,0) = \mathrm{td}(F) = \sum_{\al \in H} (-1)^{(\al,J)_l-(\al,\bJ)_r}\dim H^{\al}(F)
%\end{align*}
%\item[Signature]
%\begin{align*}
%\ind(F)(-1,0) = \sigma(F) = \sum_{\al \in H} (-1)^{(\al,J)_l-(\al,\bJ)_r}\dim H^{\al}(F)
%\end{align*}
%\end{description}
\end{prop}
%\begin{proof}
%By Corollary \ref{cor_Ramon_weight}, 
%\begin{align*}
%\ind(F)(y,0) &=  \sum_{\al \in H} (-1)^{(\al,J)_l-(\al,\bJ)_r} \dim F_{\frac{c}{24},\frac{\bc}{24}}^{\al-\ft(J+\bJ)}\left(\ft\right) y^{(J,\al)_l-\frac{c}{6}},\\
%&= y^{-\frac{c}{6}}\sum_{\al \in H} (-1)^{(\al,J)_l-(\al,\bJ)_r} \dim C^{\al}(F)y^{(J,\al)_l},\\
%&=y^{-\frac{c}{6}}\chi_y(F).
%\end{align*}
%\end{proof}

For any $n \in \ft\Z_{\geq 0}$, set
\begin{align*}
V(F)_n^\al &= F_{n,\frac{(\al,\bJ)_r}{2}}^\al\\
V(F)_n &= \bigoplus_{\al \in H} V(F)_n^\al.
\end{align*}
We will show that $V(F)=\bigoplus_{n \in \ft\Z_{\geq 0}}V(F)_n$ is an $N=2$ vertex operator algebra which is isomorphic to the \textbf{holomorphic twist} of $F$ (also called \textbf{half-twist} by \cite{Kap05}).
Recall that $\bG_{-\ft}^+=\btau^+(0)$ is a superderivation \eqref{eq_super_derivation_d}, which is a holomorphic twist in the sense of \cite{ES} (see Section \ref{sec_class_twist}).

%Set
%\begin{align*}
%d_r = \bar{\tau}^+(0) =\bG_{-\ft}.
%\end{align*}
Since $\frac{d}{d\z}Y(a,\uz)b \in \im \btau^+(0)$ holds by \eqref{eq_ker_dz} for any $a,b\in \ker \btau^+(0)$, we have an induced linear map:
\begin{align*}
\bar{Y}(\bullet,z):H(F,\btau^+(0)) \rightarrow \End H(F,\btau^+(0))[[z^\pm]],\quad a \mapsto \bar{Y}(a,z) =\sum_{n\in \Z}a(n,-1)z^{-n-1},
\end{align*}
where $H(F,\btau^+(0)) = \ker \btau^+(0)/\im \btau^+(0)$.
Since $\om,\tau,\pm,J \in \ker \btau^+(0)$, as in Proposition \ref{prop_constant_chiral} and the proof of \cite[Proposition 3.12]{M1}, we have:
\begin{lem}
\label{lem_hol_twist}
$(H(F,\btau^+(0)),\bar{Y}(\bullet,z))$ is an $N=2$ vertex operator superalgebra.
\end{lem}
Since $V(F) \subset \ker \btau^+(0)$ by Lemma \ref{lem_NS_inequality}, we have the induced linear map:
\begin{align*}
\Psi: V(F) \rightarrow H(F,\btau^+(0)),
\end{align*}
which is injective.
By the proof of Lemma \ref{lem_chiral_concentrate}, $\Psi$ is surjective. Hence, $\Psi$ is a linear isomorphism.
For $a\in V(F)_n^\al$ and $b \in V(F)_m^\be$,  define products for $k \in \Z$ by
\begin{align}
a(k)b = a(k,-1)b \in F_{n+m-k-1,\frac{(\al+\be,\bJ)_r}{2}}^{\al+\be} = V(F)_{n+m-k-1}^{\al+\be}.
\label{eq_vertex_pro_twist}
\end{align}
Then, $\Psi$ preserves the product, and thus, $V(F)$ inherits a $\Z$-graded vertex algebra structure by \eqref{eq_vertex_pro_twist}.
%In order to get the same $\ft\Z_{\geq 0}$-grading as $V(F)$ into $H(F,d_r)$, we have to choose 
%$\om-\ft L(-1)J$ instead of $\om$ as the conformal vector. This again becomes a conformal vector (see  \cite{M10}).
%$V(F)$の次数付けを$H(F,d_r)$に入れるためには、共形ベクトルとして$\om$ではなく$\om-\ft L(-1)J$を選ばなくてはならない。これは再び共形ベクトルになる \cite{M?}
Hence, we have:
\begin{prop}\label{prop_holomorphic_twist}
$\Psi: V(F) \rightarrow H(F,\btau^+(0))$ is an isomorphism of $N=2$ vertex operator superalgebras, and
\begin{align*}
\ind_F(y,q) 
&= y^{-\frac{c}{6}}\sum_{h\geq 0,\al \in H}(-1)^{(\al,J)_l-(\al,J)_r} \dim V(F)_{h+\frac{(\al,J)_l}{2}}^\al y^{(\al,J)_l}q^{h}\\
&=\mathrm{tr}|_{V(F)} (-1)^F y^{J(0)-\frac{c}{6}}q^{L(0)-\ft J(0)}.
\end{align*}
\end{prop}

Since $H(F,\btau^+(0))$ is an $N=2$ VOA, we can consider the  twist of it by $G_{-\ft}^+=\tau^+(0)$,
which is a unital supercommutative algebra by $(-1)$-th product by Proposition \ref{prop_chiral_ring}.
The following proposition easily follows from the results in Section \ref{sec_inequality}:
\begin{prop}\label{prop_double}
Let $F$ be a unitary $N=(2,2)$ full VOA. Then,
\begin{align*}
H(F,d_B) &\cong H(H(F,\btau^+(0)),\tau^+(0)),\\
H(F,d_A) &\cong H(H(F,\btau^-(0)),\tau^+(0))
\end{align*}
as $\R^2$-graded supercommutative algebras.
\end{prop}

\begin{rem}
There are two construction of an $N=2$ VOA associated with a Calabi-Yau manifold $X$:
\begin{enumerate}
\item
The chiral de Rham complex $\Om_X^{\mathrm{ch}}$;
\item
The holomorphic-twist of the supersymmetric sigma model $H(F_X,\bG_{-\ft}^+)$.
\end{enumerate}
The chiral de Rham complex $\Om_X^{\mathrm{ch}}$ is a mathematically well-defined object \cite{MSV,BHS}, but the definition of the sigma model $F_X$ is not known (there are often cases where $F_X$ can be defined by hand, and such examples are treated in Section \ref {sec_example}).
The chiral de Rham complex $\Om_X^{\mathrm{ch}}$ depends only on the complex structure of $X$, while $F_X$ depends on the K\"{a}hler structure and a $B$-field on $X$. Kapustin predicted that the large volume limit of $H(F_X,\bG_{-\ft}^+)$ would be $\Om_X^{\mathrm{ch}}$ \cite{Kap05}.
%The chiral de Rham complex $\Om_X^{\mathrm{ch}}$は数学的に良く定義された対象であるが\cite{MSV}、$F_X$の定義は分かっていない ($F_X$を手で定義できる場合はしばしばあり、そのような例は Section \ref{}で扱う)。
%chiral de Rham complex $\Om_X^{\mathrm{ch}}$は$X$の複素構造のみに依存しており、一方で$F_X$は $X$の K\"{a}hler 構造や$B$-field に依存している。物理的な議論からKapustin は$H(F_X,\bG_{-\ft}^+)$の large volume limit が $\Om_X^{\mathrm{ch}}$ になると予想した.
%We will later see that $H(F_X,\bG_{-\ft}^+)$ really depends on the K\"{a}hler structure in the case that $X$ is torus.
\end{rem}

\begin{rem}\label{rem_non_simple}
Set
\begin{align*}
I = \bigoplus_{\substack{h \geq 0,\al \in H\\ (\al,\bJ)>0}} F_{h,\frac{(\al,\bJ)}{2}}^{\al}.
\end{align*}
Then, it is clear that $I$ is an ideal of $H(F,\btau^+(0)
)$. Hence, $H(F,\btau^+(0))$ is not a simple VOA in general.
\end{rem}

Hereafter, we assume that
\begin{itemize}
\item
$H^{J}(F,d_B) \neq 0$ and $H^{\bJ}(F,d_B) \neq 0$, and set $d=\frac{c}{3}, \bar{d}=\frac{\bc}{3} \in \Z$,
\end{itemize}
and study $\ind_F(y,q)$.
By Proposition \ref{prop_period_cc}, the eigenvalues of $J(0)$ and $\bJ(0)$ are integers. Hence, by \eqref{eq_ind_q_only}, we have:
\begin{align}
\ind_F(y,q)&= y^{-d/2} \sum_{\substack{h, k,l \in \Z \\ h\geq 0, \bar{d} \geq l \geq 0}} (-1)^{k+l}
\dim F_{h+\frac{k}{2},\frac{l}{2}}^{\frac{3k}{c}J+\frac{3l}{\bc}\bJ} y^{k} q^{h},
\label{eq_ind_discrete}
\end{align}

By Proposition \ref{prop_period_cc}, 
$\dim F_{h,\frac{(\al,\bJ)}{2}}^{\al} = \dim F_{h-n(\al,J)+\frac{n^2 c}{6},\frac{(\al,\bJ)}{2}}^{\al-nJ}$ for any $n\in \Z$.
Hence, by setting $\al'=\al-nJ$ and $h' = h-n(\al,J)+\frac{n^2c}{6}$, we have
\begin{align*}
\ind_F(y,q)&=y^{-\frac{c}{6}}\sum_{h,\al \in H} (-1)^{(\al,J)_l-(\al,\bJ)_r} \dim F_{h-n(\al,J)+\frac{n^2 c}{6},\frac{(\al,\bJ)}{2}}^{\al+nJ}y^{(J,\al)_l} q^{h-\frac{(\al,J)}{2}}\\
&=(-1)^{2n\frac{c}{6}}y^{-\frac{c}{6}}\sum_{h',\al' \in H} (-1)^{(\al',J)_l-(\al',\bJ)_r} \dim F_{h',\frac{(\al',\bJ)}{2}}^{\al'}
y^{(J,\al'+nJ)_l} q^{h'+n(\al'+nJ,J)-n^2 \frac{c}{6}- \frac{(\al'+nJ,J)}{2}}\\
&=(-1)^{n\frac{c}{3}}y^{n\frac{c}{3}}q^{tn^2} (q^ny)^{-\frac{c}{6}} \sum_{h',\al' \in H} (-1)^{(\al',J)_l-(\al',\bJ)_r} \dim F_{h',\frac{(\al',\bJ)}{2}}^{\al'}
(q^ny)^{(J,\al')_l}q^{h'-\frac{(\al',J)}{2}}\\
&=(-1)^{n\frac{c}{3}}y^{n\frac{c}{3}}q^{tn^2}\ind_F(q^ny,q).
\end{align*}
Hence, we have:
\begin{prop}
\label{prop_ind_periodicity}
Assume $H^{\frac{c}{3},0}(F,d_B) \neq 0$ and $H^{0,\frac{\bc}{3}}(F,d_B) \neq 0$. Then, the following properties hold:
\begin{enumerate}
\item
$\ind (F)(y,q) \in y^{-\frac{c}{6}}\C[[q]][y^\pm]$;
\item
$\ind(F)(yq^n,q) = (-1)^{2n \frac{c}{6}}q^{-\frac{c}{6} n^2}y^{-2n \frac{c}{6}} \ind(F)(y,q)$ holds for any $n\in \Z$.
\end{enumerate}
\end{prop}
By setting $q=\exp(2\pi i \tau),y=\exp(2\pi i z)$, we will see that 
Proposition \ref{prop_ind_periodicity} is nothing but the part of the modularity of weak Jacobi forms of index $\frac{c}{6}$ (see Definition \ref{def_weak_jacobi}).

\section{Relation with complex geometry and examples}\label{sec_example}

In this section we compare the structures obtained above with the corresponding
structures in complex geometry, and we examine several examples.
Section~4.1 reviews the basic geometric invariants of Calabi-Yau manifolds.
In Section~4.2 we formulate the cohomological properties that are expected to hold
for full VOAs arising from the sigma models associated with Calabi-Yau manifolds.
Sections~4.3 and~4.4 verify these expectations for abelian varieties and for a
special K3 surface, while Section~4.5 discusses the simplest Landau-Ginzburg
model, which lies outside the Calabi-Yau-type class.

\subsection{Invariants of Calabi-Yau manifolds in geometry}\label{sec_cpx_geom}
The properties of $N = (2,2)$ supersymmetric conformal field theories described in Theorem \ref{thm_chiral_star} are related to the properties of  Calabi-Yau manifolds. In this section, we briefly review  invariants of a Calabi-Yau manifold. The contents of this section are all standard
%定理\ref{thm_chiral_star}で述べられた$N = (2,2)$超対称共形場理論の性質は Calabi-Yau manifold の性質と関係している。この章ではCalabi-Yau manifold の不変量を簡単に振り返る。この章の内容はどれも標準的な内容である
 (see for example \cite{Huy,Gri,HBJ,EZ,Hori}). 
 
Let $X$ be a compact K\"{a}hler manifold.
Then, the deRham cohomology $H^\bullet(X,\C)$ admits the Hodge decomposition \cite{Huy}:
\begin{align*}
H^k(X,\C) = \bigoplus_{p+q=k} \mathrm{Harm}^{p,q}(X),
\end{align*}
where  $\mathrm{Harm}^{p,q}(X)$ is a vector space of Harmonic $(p,q)$-forms such that:
\begin{enumerate}
\item
$\mathrm{Harm}^{p,q}(X)$ is isomorphic to the sheaf cohomology $H^q(X,\Om_X^p)$, where $\Om_X^p$ is the sheaf of holomorphic $p$-forms;
\item
the pairing
\begin{align*}
\mathrm{Harm}^{p,q}(X) \times \mathrm{Harm}^{n-p,n-q}(X) \rightarrow \C,\quad(\al,\be) \mapsto \int_X \al \wedge \be
\end{align*}
is non-degenerate, and thus, the integral gives an $\R^2$-graded supercommutative Frobenius algebra structure on $H(X,\C)$.
%\item
%The complex conjugation on $H^*(X,\C) =H^*(X,\R)\otimes_\R \C$ yields $\overline{H^{p,q}(X)} = H^{q,p}(X)$.
%\item
%The Hodge star operator induces isomorphisms
%\begin{align*}
%*: \mathrm{Harm}^{p,q}(X,g) \rightarrow \mathrm{Harm}^{n-q,n-p}(X,g).
%\end{align*}
%\item
%\todo
%Frobenius algebra structure !!
\end{enumerate}

\begin{dfn}\label{def_Hirzebruch}
The Hodge numbers of $X$ is defined as
\begin{align*}
h_{p,q}(X) = \dim_\C \mathrm{Harm}^{p,q}(X)
\end{align*}
and the Hirzebruch $\chi_y$ genus of $X$ as
\begin{align*}
\chi_y(X)= \sum_{p,q}(-1)^{p+q} \dim_\C H^{p,q}(X) y^p.
\end{align*}
%The dimension $h_{p,q}(X) = \dim_\C H^{p,q}(X)$ are called the Hodge numbers.
% Hence, the Hirzebruch genus can be written by using the Chern roots by the index Theorem.
\end{dfn}

Note that $\chi_y(X) = \sum_{p\geq 0} \chi(X,\Om_X^p)y^p$, where $\chi(X,\Om_X^p)$ is the Euler number of the sheaf cohomology.
Hence, according to the Hirzebruch-Riemann-Roch theorem, we have (see for example \cite[Corollary 5.1.4]{Huy} or \cite{HBJ}):
\begin{prop}\label{prop_HRR}
Let $X$ be a compact complex manifold of dimension $d$.
Let $\{\ga_i\}_{i=1,\dots,d}$ denote the formal Chern roots of the holomorphic tangent bundle $T_X$. Then,
\begin{align*}
\chi_y(X) = \int_X \Pi_{i=1}^d(1+y e^{-\ga_i}) \frac{\ga_i}{1-e^{-\ga_i}}.
\end{align*}
\begin{enumerate}
\item
(Arithmetic genus)
\begin{align*}
\chi_{y=0}(X) = \sum_{q}(-1)^{q} h_{0,q}(X) =\chi(X,O_X) = \int_X \mathrm{td}(X).
\end{align*}
\item
(Euler number) If $X$ is a compact K\"{a}hler manifold, then
\begin{align*}
\chi_{y=1}(X) = \sum_{p,q}(-1)^{p+q} h_{p,q}(X) = \int_X c_d(X)=e(X).
\end{align*}
\item
(Signature) If $X$ is a complex even dimensional compact K\"{a}hler manifold, then
\begin{align*}
\chi_{y=-1}(X) = \sum_{p,q}(-1)^{q} h_{p,q}(X) = \int_X \Pi_{i=1}^d \ga_i \mathrm{coth}\left(\frac{\ga_i}{2}\right) =\si(X),
\end{align*}
which is the signature of the underlying real $2d$-dimensional manifold.
\end{enumerate}
\end{prop}

Finally, we review the elliptic genus of a compact complex manifold $X$ \cite{Gri}.
The elliptic genus is a generalization of the Hirzebruch genus with two formal variables $q,y$, and if the first Chern class $c_1(X)=0$ of $X$ is equal to $0$, then the formal series is a weak Jacobi form.

%最後に compact complex manifold $X$の楕円種数を復習する \cite{Gri}。楕円種数は二つの形式的変数$q,y$を持ったHirzebruch genusの一般化であり、$X$の first Chern class $c_1(X)=0$が０に等しい場合は、形式的級数はweak Jacobi form になる。

Let $X$ be a compact complex manifold of complex dimension $d$ and $T_X$ be the holomorphic tangent bundle of $X$. One defines formal power series
\begin{align*}
\mathbb{E}_{q,y} = \bigotimes_{n\geq 0} \extp_{-y^{-1}q^n}T_X \otimes \bigotimes_{n\geq 1}
\extp_{-yq^n}T_X^* \otimes \bigotimes_{n\geq 1} S_{q^n}(T_X \oplus T_X^*),
\end{align*}
where $\wedge^k$ and $S^k$ denote the $k$-th exterior product and $k$-th symmetric tensor product
respectively and $T_X^*$ denotes the dual bundle.
The elliptic genus of the complex manifold
is the holomorphic Euler characteristic of this formal power series with vector bundle coefficients. This definition is standard in physical literature (see, for example, \cite{W,EOTY,Grimm}):
\begin{dfn}\label{def_elliptic_genus}
The elliptic genus of $X$ is defined as follows
\begin{align*}
\mathrm{Ell}(X;q,y) = y^{\frac{d}{2}} \int_X \mathrm{ch}(\mathbb{E}_{q,y})\mathrm{td}(X)
\end{align*}
where $\mathrm{td}(X)$ is the Todd class of $X$ (see \cite{HBJ}).
% One applies the Chern character in $\mathrm{ch}(\mathbb{E}_{q,y})$ to all coefficients of $\mathbb{E}_{q,y}$ and the integral $\int_X$ denotes the evaluation of the top degree differential form on the fundamental cycle of the manifold.
\end{dfn}

\begin{dfn}\label{def_weak_jacobi}
A holomorphic function $\phi(\tau,z)$ on $\mathbb{H}\times \C$ is called a {\bf weak Jacobi form} of weight $k \in \ft\Z$ and index $t \in \ft \Z$ 
%with respect to the Jacobi group $\Ga^J$
%if the function $\tilde{\phi}(Z) = \phi(\tau,z)\exp(2\pi i t\om) on the Siegel upper half-plane 
%\begin{align*}
%\mathbb{H}_2= \{Z = \begin{pmatrix}
%a & b \\
%c & d \\
%\end{pmatrix}
%\in M_4(\C) \mid \mathrm{Im}(Z) >0\}
%\end{align*}
%is a $\Ga^J$-modular form of weight $k$ with character $v_J^{2t}$, i.e., 
if $\phi$ satisfies the equations
\begin{align*}
\phi\left(\frac{a\tau+b}{c\tau+d},\frac{z}{c\tau+d}\right)=(c\tau+d)^k e^{2\pi i t cz^2/(c\tau+d)}\phi(\tau,z)
\end{align*}
for any $\begin{pmatrix}
a & b \\
c & d \\
\end{pmatrix} \in \mathrm{SL}_2(\Z)$ and
\begin{align*}
\phi(\tau,z+\la\tau+\mu) = (-1)^{2t(\la+\mu)}e^{-2 \pi i t (\la^2\tau+2\la z)} \phi(\tau,z)
\end{align*}
for any $\la,\mu \in \Z$ and has Fourier expansions of the following type:
\begin{align}
\phi(\tau,z) = \sum_{n\geq 0, l \in t+\Z} f(n,l)\exp(2\pi i (n\tau+lz)).
\label{eq_Jacobi_fourier}
\end{align}
We denote by $J_{k,t}$ the space of all weak Jaboci forms of weight $k$ and index $t$ (see \cite{EZ}).
\end{dfn}
Note that a {\it Jacobi form} of weight $k$ and index $t$ is a function $\phi$ in $J_{k,t}$ which has Fourier expansions of the following type:
\begin{align*}
\phi(\tau,z) = \sum_{n\geq 0, l^2 \leq 4mn} f(n,l)\exp(2\pi i (n\tau+lz)).
\end{align*}

\begin{prop}\cite[Proposition 1.2]{Gri}
\label{prop_jacobi}
Let $X$ be a compact complex manifold of dimension $d$. Then, the following properties hold:
\begin{enumerate}
\item
If the first Chern class $c_1(X)$ is equal to zero over $\R$, then its elliptic genus $\mathrm{Ell}(X;q,y)$ is a weak Jacobi form of weight $0$ and index $\frac{d}{2}$.
\item
If $X$ is a K\"{a}hler manifold, then the constant term of $\mathrm{Ell}(X;q,y)$ is essentially the Hirzebruch $\chi_y$-genus, that is,
\begin{align*}
\lim_{q\to 0} \mathrm{Ell}(X;q,y) = y^{-\frac{d}{2}}\chi_y(X).
\end{align*}
\end{enumerate}
\end{prop}
Note that Proposition \ref{prop_Witten_Hirz} is a direct analogue of Proposition \ref{prop_jacobi} (3).

\begin{dfn}\label{def_CY_manifold}
A Calabi-Yau $d$-fold is a compact $d$-dimensional K\"{a}hler manifold $X$ such that the canonical bundle $K_X = \Om_X^d$ is trivial.
\end{dfn}
\begin{rem}\label{rem_CY}
Important consequences that follow immediately from the definition include, for example:
\begin{enumerate}
\item
The first Chern class $c_1(X)$ is zero;
\item
$H^q(X,O_X) \cong H^{d-q}(X,O_X)^*$ by the Serre duality with $K_X=1$;
\item
$h_{d,0} =1$, that is, the existence of the holomorphic volume form $\C = \Gamma(X,\Om_X^d)$.
\end{enumerate}
\end{rem}

In summary, we get the following:
\begin{prop}\label{prop_CY_geometry}
Let $X$ be a Calabi-Yau $d$-fold. Then, the following properties hold:
\begin{itemize}
\item
$h_{d,0}=h_{0,d} = h_{0,0} = h_{d,d}=1$;
\item
$h_{p,p} \geq 1$ for any $d\geq p \geq 0$ (by the existence of the K\"{a}hler form);
\item
$h_{p,q}=h_{d-p,d-q}$ for any $p,q\in \Z$;
\item
$h_{p,0}=h_{d-p,0}$ for any $p\in \Z$;
\item
The elliptic genus $\mathrm{Ell}(X;q,y)$ is a weak Jacobi form of weigh $0$ and index $\frac{d}{2}$.
\end{itemize}
\end{prop}

Finally, we conclude this section by presenting the Hodge numbers and elliptic genus of Calabi-Yau manifolds of dimension $1$ and $2$.
%最後に$X$が1次元と2次元のCalabi-Yau 多様体の場合に、楕円種数やhodge 数を述べてこの章を終える。

\begin{exa}[Torus]
A Calabi-Yau $1$-fold is a complex torus; the Hodge numbers of the torus are
%1次元のCalabi-Yau多様体は torus である。Torus の Hodge 数は
\begin{align}
\begin{array}{ccc}
& h^{0,0} \\
h^{1,0} && h^{0,1}\\
& h^{1,1}&
\end{array}
\quad =\quad
\begin{array}{ccc}
& 1 &\\ 1 && 1 \\ & 1
\end{array}
\end{align}
and $\chi_y(torus) = \Ell(torus;q,y) = 0$.
In particular, for complex $d$-dimensional torus, the elliptic genus is zero.
\end{exa}

\begin{exa}[K3 surface]\label{example_K3}
A Calabi-Yau $2$-fold is a 2-torus or K3 surface; the Hodge numbers of the K3 surface are
%2次元のCalabi-Yau多様体は torusの直積か K3 surface である。K3 surface の Hodge 数は
\begin{align}
\begin{array}{ccccc}
&& h^{0,0} \\
& h^{1,0} && h^{0,1}\\
h^{2,0} & & h^{1,1} & & h^{0,2}\\
& h^{2,1} & & h^{1,2} \\
&& h^{2,2}
\end{array}
\quad =\quad
\begin{array}{ccccc}
&& 1 \\ & 0 && 0 \\ 1 & & 20 & & 1 \\ & 0 & & 0 \\ && 1
\end{array}
\end{align}
and
\begin{align*}
\chi_y(K3) = 2+20y + 2y^2,\quad\quad \Ell(K3;q,y) = 2 \phi_{0,1}(\tau,z),
\end{align*}
where 
\begin{align}
\phi_{0,1}(\tau,z) =(y^{-1}+10+y) + q(10y^{-2}-64y^{-1}+108-64y+10y^2)+q^2(\dots)
\end{align}
is the unique, up to a constant, weak Jacobi form of weight 0 and index 1 \cite{Gri}, that is, 
\begin{align}
J_{0,1} = \C \phi_{0,1}(\tau,z).
\label{eq_Jacobi_unique}
\end{align}
\end{exa}

\subsection{CY type and Hodge-theoretic mirror symmetry}
%\subsection{Proposal on the Hodge mirror symmetry}
\label{sec_conj}
%Calabi-Yau多様体でコンパクト化されたターゲット時空を持つ超弦理論を考えることで、world-sheet上に$(2,2)$の超対称性を持つ共形場理論を得ることができる。
The class of unitary $N=(2,2)$ full VOAs is much broader than the class expected to
arise from Calabi--Yau sigma models.
Motivated by the characterization of periodicity in terms of top-degree classes,
by the resulting Poincar\'e duality and T-duality, and by the geometric definition
of Calabi--Yau manifolds, we introduce the following condition for
full VOAs of geometric origin.
%
%In this section we propose a conjecture about the existence  of unitary $N=(2,2)$ full VOAs satisfying certain properties based on physics \cite{Witten1,Witten87,Witten,Hori}.
%Importantly, not all unitary $N=(2,2)$ supersymmetric CFTs come from Calabi-Yau manifolds.
%For example, CFTs from Calabi-Yau $d$-folds have the central charge $({3d},{3d})$, but there are many supersymmetric conformal field theories whose central charges are not integers (for example, the supersymmetric minimal model and Landau-Ginzburg model \cite{Hori}).
%The question then arises: What conditions can be considered to characterize the $N=(2,2)$ full VOAs  obtained from (generalized) Calabi-Yau manifolds?
%Although it is probably not a sufficient condition for characterization, we expect that the full VOAs obtained as sigma models at least satisfies the following conditions:
%この章ではある性質を満たす$N=(2,2)$ supersymmetric full VOA の存在と性質に関する予想を述べる。
%重要なことに、全ての unitary $N=(2,2)$ supersymmetric CFTが Calabi-Yau多様体から来るわけでは全くない。
%たとえば $d$次元のCalabi-Yau多様体から来る CFT は中心電化 $(\frac{3d}{2},\frac{3d}{2})$を持つが、半整数でない超対称共形場理論はたくさん存在する(for example, supersymmetric minimal model, and Landau-Ginzburg model, etc.).
%するとCalabi-Yau多様体から得られる supersymmetric full VOA を代数的に特徴づけるにはどのような条件を考えればよいか？という問いが生じる。
%特徴づける条件としてはおそらく十分ではないが、我々は少なくともこうした sigma model として得られる full VOAは次の条件を満たすと予想する:
\begin{dfn}\label{def_CY_VOA}
We call a unitary $N=(2,2)$ full vertex operator superalgebra $F$ of \textbf{CY type} if the following conditions are satisfied:
\begin{enumerate}
\item
$c=\bar{c}=3d$;
%\item
%$F$ is isomorphic to the conjugate full vertex operator algebra $\overline{F}$ as $N=(2,2)$ full vertex operator superalgebras;
%\item
%$F$ admits a mirror automorphism $\phi \in \Aut F$;
\item
$H^{d,0}(F,d_A)$, $H^{0,d}(F,d_A)$ and $H^{d,d}(F,d_A)$ are non-zero vector spaces;
\item
$H^{d,0}(F,d_B)$, $H^{0,d}(F,d_B)$ and $H^{d,d}(F,d_B)$ are non-zero vector spaces.
%\item
%There is a vector $\om \in H^{1,1}(F,d_B)$ such that $\om^k \in H^{k,k}(F,d_B)$ is non-zero for $k=1,\dots,d$.
\end{enumerate}
\end{dfn}
\begin{rem}
In an earlier version, the definition of CY type included the existence of a
symplectic form in addition to the conditions above. For mirror-symmetric purposes,
however, this is too restrictive: 
the class defined in this way is not closed under
taking mirror algebras, since the existence of a symplectic form is not preserved by
the mirror involution in general.
Geometrically, rigid Calabi-Yau manifolds
already show the issue: one should not expect mirrors to remain
within the ordinary class of Calabi-Yau manifolds in such a restrictive sense
\cite{CDGP,Witten}. For this reason, in the present paper we do not include the
existence of a symplectic form in the definition of CY type.
\end{rem}

\begin{prop}\label{prop_equiv_CY}
Assume that $c=\bar{c}$ and set $d=c/3$. Then, the following three conditions are equivalent:
\begin{enumerate}
\item
$F$ is of CY type.
\item
Two of $H^{d,0}(F,d_A)$, $H^{0,d}(F,d_A)$ and $H^{d,d}(F,d_A)$ are non-zero vector spaces.
\item
Two of $H^{d,0}(F,d_B)$, $H^{0,d}(F,d_B)$ and $H^{d,d}(F,d_B)$ are non-zero vector spaces.
\end{enumerate}
Moreover, if these equivalent conditions hold, then $d$ is a positive integer.
\end{prop}
\begin{proof}
Assume that (3) holds. Then, by Theorem \ref{thm_CY}, all of $H^{d,0}(F,d_A)$, $H^{0,d}(F,d_A)$ and $H^{d,d}(F,d_A)$ are non-zero vector spaces. By Proposition \ref{prop_period_cc}, $d=\frac{c}{3} \in \Z_{\geq 0}$. Hence, by Theorem \ref{thm_chiral_star}, (1) holds.
\end{proof}

\begin{conj}\label{conj_sigma}
Let $X$ be a compact Calabi-Yau manifold of complex dimension $d$. Then, there is a unitary $N=(2,2)$ full vertex operator algebra $F_X$ of central charge $({3d},{3d})$ such that:
\begin{enumerate}
\item
$F_X$ is of CY type;
%\item
%$F_X$ admits a mirror automorphism $\phi \in \Aut F_X$;
%\item
%All of $H^{d,d}(F_X)$, $H^{d,0}(F_X)$ and $H^{0,d}(F_X)$ are non-zero vector spaces;
\item
The $A$-Hodge numbers of $F_X$ coincide with the Hodge numbers in K\"{a}hler geometry:
$h_{p,q}^A(F_X) = h_{p,q}(X)$;
%\item
%The cohomology ring $H(F_X,d_A)$ is isomorphic to the
%quantum cohomology ring $QH(X)$ (For the definition of quantum cohomology ring, see \cite{Hori});
% Dolbeault cohomology $H^{\bullet,\bullet}(X)$ as a $\Z^2$-graded algebra.
%\footnote{We also expect that the cohomology ring $H(F_X,d_B)$ is isomorphic to the
%Hochschild cohomology $HH(X)$.}
\item
The Witten index coincides with the elliptic genus
\begin{align*}
\mathrm{Ind}_{F_X}(q,y) = \Ell(X;q,y).
\end{align*}
\end{enumerate}
\end{conj}

\begin{rem}\label{rem_conj_ring}
Witten further conjectured that $H(F_X,d_A)$ should recover the small quantum
cohomology of $X$ as a ring \cite{Witten}. More precisely, the small quantum
cohomology ring is naturally defined over a Novikov-type coefficient ring, and the
physical ring associated with the sigma model is expected to arise from this formal
family after specialization to the K\"ahler parameters. From a mathematical point of
view, the convergence of such a specialization is a highly nontrivial issue; see, for
example, \cite{Iri}.
On the $B$-side, the corresponding cohomology ring is expected to be described in
terms of the Hochschild cohomology of the derived category of $X$, or more generally
of its $B$-field twist \cite{Kont,Costello}.
\end{rem}

We will prove the existence of full VOAs which satisfy the conjecture in Section \ref{sec_example_abelian} and Section \ref{sec_example_K3} for the case that $X$ is an abelian variety and a K3 surface.
By combining Theorem \ref{thm_full_frob}, Theorem \ref{thm_chiral_star} and Proposition \ref{prop_ind_periodicity}, we have the following result:
\begin{thm}\label{thm_CY_VOA}
Let $F$ be a unitary $N=(2,2)$ full vertex operator superalgebra of CY type.
Set $h_{p,q} = \dim H^{p,q}(F,d_B)$ for $p,q\in \R$.
Then, the following properties hold:
\begin{enumerate}
\item
The mirror algebra $\tilde{F}$ is also of CY type;
\item
$h_{p,q}=0$ unless $p,q \in \Z_{\geq 0}$;
\item
$h_{0,0}=h_{d,0}=h_{0,d}=h_{d,d}=1$;
%\item
%$h_{p,p} \geq 1$ for any $d\geq p\geq 0$;
\item
$h_{p,q} = h_{d-p,d-q}$ for any $p,q\in\Z$;
\item
$h_{p,0} = h_{d-p,0}$ and $h_{0,p} = h_{0,d-p}$ for any $p\in \Z$;
\item
$H^{p,q}(F,d_A) \cong H^{d-p,q}(F,d_B)$.
\end{enumerate}
Moreover, if we assume
\begin{align*}
\ind_F(y,q)&= y^{-d/2} \sum_{\substack{h, k,l \in \Z \\ h\geq 0,\bar{d} \geq l \geq 0}} (-1)^{k+l}
\dim F_{h+\frac{k}{2},\frac{l}{2}}^{\frac{3k}{c}J+\frac{3l}{\bc}\bJ} y^{k} q^{h}
\end{align*}
is absolutely convergent for any $|q|<1$ and $y\in \C$, then, by setting $q=\exp(2\pi i \tau)$ and $y=\exp(2\pi i z)$, $\phi_F(\tau,z) =\ind_F(\exp(2\pi i z),\exp(2\pi i \tau))$ is a holomorphic function on $\mathbb{H} \times \mathbb{C}$ such that:
\begin{align*}
\phi_F(\tau+1,z) &= \phi_F(\tau,z)\\
\phi_F(\tau,z+1) &= (-1)^d \phi_F(\tau,z)\\
\phi_F(\tau,z+\tau) &= (-1)^d \exp(-\pi i d (\tau+2z))\phi_F(\tau,z).
\end{align*}
If we further assume that
\begin{align*}
\phi_F\left(\frac{-1}{\tau},\frac{z}{\tau}\right) = \exp\left( \pi i d\frac{z^2}{\tau}\right)\phi_F(\tau,z),
\end{align*}
then $\phi_F(\tau,z)$ is a weak Jacobi form of weight $0$ and index $d/2$.
\end{thm}

Combining Conjecture \ref{conj_sigma} with T-duality, one recovers the expected
Hodge-theoretic mirror relation for Calabi-Yau sigma models. Hence, we
have:
\begin{cor}\label{cor_mirror_CY}
Let $X$ and $Y$ be Calabi-Yau $d$-folds. Assume that there are $N=(2,2)$ full vertex operator superalgebras $F_X$ and $F_{Y}$ which satisfy Conjecture \ref{conj_sigma}, respectively.
Assume that $F_X$ and $F_{Y}$ are the mirror algebras, that is, $\hat{F_X} \cong F_{Y}$ as unitary $N=(2,2)$ full VOAs. Then, $h_{p,q}(X) = h_{d-p,q}(Y)$ for any $p,q\in \Z$.
\end{cor}

\subsection{Supersymmetric sigma model on abelian varieties}\label{sec_example_abelian}

Let $d \in \Z_{>0}$.
The supersymmetric sigma model with complex $d$-dimensional torus $T_d$ (abelian varieties) as target is the tensor product of a Narain CFT of central charge $(2d,2d)$ and $d$ chiral and $d$ anti-chiral free fermions \cite{KO}.
%この章では天下り的に unitary $N=(2,2)$ full VOA の変形族を与えてその cohomology ring を計算する。代数の変形族はトーラスの複素構造とKahler 構造の変形に対応する。こうした幾何構造との関係は\cite{KO}を参照されたい。

Narain CFTs were constructed in \cite{M1} as a deformation family of full vertex operator algebras, and their unitarity was shown in \cite{AMT}.
However, \cite{M1,AMT} only treat the bosonic case,
and the sign must be treated carefully in order to extend these results to the fermionic case.
In this section, using the notion of a unitary generalized full vertex algebra formulated in Appendix A, we show that the full vertex operator superalgebra corresponding to abelian varieties has the structure of an $N=(2,2)$ full VOA of CY type, and compute the Witten genus and the cohomology rings.
Note that the deformation family of algebras correspond to the deformation of complex and Kahler structures of the torus. Readers interested in the relationship to these geometric structures are referred to \cite{KO}.

%
%complex $d$次元torusを target にする超対称シグマ模型は、$(S^1)^{2d}$をtarget にする Narain CFTと d個のchiral and anti-chiral free fermion のテンソル積となる。
%Narain CFT は full vertex operator algebra として \cite{M1} において構成され、\cite{AMT}においてそのunitary 性が示された。
%ただしこれらはbosonic な場合であり、こうした結果を fermionic な場合に拡張するためには符号を慎重に扱わなければならない。
%この章では、Appendix A において定式化したユニタリ GVA という概念を用いて、abelian varieties に対応する full vertex operator superalgebra が CY type の $N=(2,2)$ full VOAの構造を持つことを示し、その elliptic genus と cohomology ring を計算する。

Let $H = \R^{3d,3d}$ be the real vector space equipped with non-degenerate symmetric bilinear form $(-,-)$ with signature $(3d,3d)$.
Hereafter of this section, we denote a vector in $\R^{3d,3d}$ by $u=(a_1,\dots,a_{3d};b_1,\dots,b_{3d}) \in \R^{3d}$. The square norm of $u$ is given by
\begin{align*}
(u,u) = \sum_{i=1}^{3d} a_i^2 -  \sum_{i=1}^{3d} b_i^2.
\end{align*}
Let $p:\R^{3d,3d} \rightarrow \R^{3d}$ be the projection given by $u \mapsto (a_1,\dots,a_{3d})$. Then, $p$ is a projection onto a maximal positive definite subspace. Set $\bar{p}=1-p$.
Set 
\begin{align*}
\al_i &= (\underbrace{0,\dots,1,\dots,0}_{i\text{-th place}}; 0,\dots,0) \in \R^{3d,3d}\\
\ga_j &= (\underbrace{0,\dots,0}_{2d 0s}, \underbrace{0,\dots,1,\dots,0}_{j\text{-th place}}; 0,\dots,0)\in \R^{3d,3d}\\
\bal_i &= (\underbrace{0,\dots,0}_{3d 0s}; \underbrace{0,\dots,1,\dots,0}_{i\text{-th place}}, 0,\dots,0)\in \R^{3d,3d}\\
\bar{\ga}_j &= (\underbrace{0,\dots,0}_{3d 0s};\underbrace{0,\dots,0}_{2d 0s}, \underbrace{0,\dots,1,\dots,0}_{j\text{-th place}})\in \R^{3d,3d},
\end{align*}
for $i \in \{1,\dots,2d\}$ and $j \in \{1,\dots,d\}$.

Let $L\subset \R^{3d,3d}$ be a subgroup such that:
\begin{enumerate}
\item[L1)]
$L$ is free abelian group of rank $6d$;
\item[L2)]
$L$ is an {\bf integral lattice}, that is, $(\al,\be) \in \Z$ for any $\al,\be \in L$;
\item[L3)]
$L$ is {\bf unimodular}, that is, if $\ga \in \R^{3d,3d}$ satisfies $(\al,\ga)\in\Z$ for any $\al\in L$,
then $\ga \in L$.
\end{enumerate}

In Proposition \ref{prop_unitary_Narain} in appendix, we construct a full vertex operator algebra 
\begin{align*}
F_{L,p} = \bigoplus_{\al \in H}\C e_\al \otimes M_{H,p}(\al)\subset \C[\hat{L}]\otimes G_{H,p},
\end{align*}
which is unitary.
We will take a specific family of $L \hookrightarrow \R^{3d,3d}$ and introduce $N=(2,2)$ structure on $F_{L,p}$.

\begin{rem}
An integral unimodular lattice  $L$ is called {\bf odd} (resp. {\bf even}) if there exists an element $\al \in L$ such that $(\al,\al)$ is an odd integer (resp. otherwise).
It is well-known that an even or odd indefinite unimodular lattice is unique up to isomorphism.
More specifically, if $L$ is odd, then
\begin{align*}
L \cong I_{3d} \oplus -I_{3d}
\end{align*}
and if $L$ is even, then
\begin{align*}
L \cong \tw_{3d,3d}.
\end{align*}
(see for example \cite{Serre}). In our case, however, the embedding of $L$ into $\R^{3d,3d}$ is important, which is not unique at all and gives a family (of full vertex operator superalgebras) parametrized by the Grassmannian
\begin{align*}
\mathrm{Aut}\,L \backslash \mathrm{O}(3d,3d;\R) / \mathrm{O}(3d;\R)\times \mathrm{O}(3d;\R),
\end{align*}
where $\mathrm{O}(3d,3d;\R)$ is the real orthogonal group and $\mathrm{Aut}\,L$ is the automorphism group of the lattice (see \cite{M1}).
\end{rem}

Let $M \subset \R^{2d,2d}$ be a subgroup such that:
\begin{itemize}
\item
$M$ is an even unimodular lattice of rank $4d$.
\end{itemize}
Define $M\oplus I_{d,d} \subset \R^{3d,3d}$ be the subgroup generated by $L$ and $\{\ga_j,\bar{\ga}_j\}_{j=1,\dots,d}$. Then, $M\oplus I_{d,d}$ satisfies (L1), (L2) and (L3).
%Choose a $\Z$-basis of $M\oplus I_{d,d}$ by extending $\{\ga_j,\bar{\ga}_j\}_{j=1,\dots,d}$, and construct a 2-cocycle $\ep \in H^2(M\oplus I_{d,d},\C^\times)$ by \eqref{eq_def_cocycle2}. 
Then, $F_{M\oplus I_{d,d},p}$ is a unitary full vertex operator superalgebra with the anti-linear involution $\phi$ defined by Proposition \ref{prop_unitary_Narain}.
Note that $\phi$ satisfies
\begin{align*}
\phi(e_{\al})&=e_{-\al}
\end{align*}
for any $\al \in M\oplus I_{d,d}$.

\begin{prop}\label{prop_sigmaVOA_torus}
The unitary vertex operator superalgebra $(F_{M\oplus I_{d,d},i},\phi)$ inherits a unitary $N=(2,2)$ structure by
\begin{align*}
\om &= \frac{1}{2} \sum_{i=1}^{2d} \al_i(-1)\al_i + \ft\sum_{i=1}^d \ga_i(-1)\ga_i,&\omb &= \frac{1}{2} \sum_{i=1}^{2d} \bal_i(-1)\bal_i + \ft\sum_{i=1}^d \bar{\ga}_i(-1)\bar{\ga}_i\\
J &= \sum_{i=1}^d \ga_i, &\bJ &= \sum_{i=1}^d \bar{\ga}_i\\
\tau^+ &= \sum_{i=1}^d (\al_{2i-1}+i\al_{2i})(-1) e_{\ga_i}, &\tau^- &= \sum_{i=1}^d (-\al_{2i-1}+i\al_{2i})(-1) e_{-\ga_i}\\
\btau^+ &= \sum_{i=1}^d (i \bal_{2i-1}-\bal_{2i})(-1) e_{\bar{\ga}_i}, &\btau^- &= \sum_{i=1}^d (-i \bal_{2i-1}-\bal_{2i})(-1) e_{-\bar{\ga}_i}.
\end{align*}
\end{prop}
\begin{proof}
Since $\phi$ is an anti-linear map, $\phi(J)=-J$, $\phi(\bJ)=-\bJ$, $\phi(\tau^\pm)=\tau^\mp$ and $\phi(\btau^\pm)=-\btau^\mp$.
The difference of $\tau^\pm$ and $\btau^\pm$ comes from
\begin{align*}
e_{\ga_j} (0)e_{-\ga_j} &= (-1)^{\ft((\ga_j,\ga_j)+(\ga_j,\ga_j)^2)}\vac=(-1)^{\ft(1+1)}\vac =-\vac\\
e_{\bar{\ga}_j} (0)e_{-\bar{\ga}_j} &= (-1)^{\ft((\bar{\ga}_j,\bar{\ga}_j)+(\bar{\ga}_j,\bar{\ga}_j)^2)}\vac=(-1)^{\ft(-1+1)}\vac=\vac
\end{align*}
by construction (see \eqref{eq_minus_ep}).
%\begin{itemize}
%\item
%$\phi(\tau^\pm)$も計算するべき。
%\end{itemize}
%We will check the OPE of $\tau^+,\tau^-$ and $\btau^+,\btau^-$. 
We leave the rest of the calculation to the reader.
%\begin{align*}
%\tau^+(z)\tau^-=
%\end{align*}
\end{proof}

\begin{prop}\label{prop_}
$H^{p,q}(F_{M\oplus I_{d,d},p},d_B)$ is spanned by \begin{align*}
\{e_{\ga_{i_1}+\dots+\ga_{i_p}}\otimes e_{\bar{\ga}_{j_1}+\dots+\bar{\ga}_{j_q}}\}
\end{align*}
with $1 \leq i_1<\cdots <i_p\leq d$ and $1 \leq j_1<\cdots <j_q\leq d$  
and $H(F_{M\oplus I_{d,d},p},d_T)$ isomorphic to the
deRham cohomology of $2d$-dimensional complex torus $H(T_d,\C) = \bigoplus_{p,q}H^{p,q}(T_d,\C)$ with the Hodge decomposition for $T=A,B$.
Moreover, $\ind_{F_{M\oplus I_{d,d},p}}(y,q)=0$.
In particular, $F_{M\oplus I_{d,d},p}$ satisfies Conjecture \ref{conj_sigma} for the $d$-dimensional complex torus.
\end{prop}
\begin{proof}
Since $\ind_{\overline{V_\Z}}(y,q)=0$, $\ind_{F_{M\oplus I_{d,d},p}}(y,q)=0$. Since any holomorphic map from complex projective space $\CP$ to the complex one-dimensional torus $T_1$ is trivial, the quantum cohomology ring $QH(T_d)$ is isomorphic to the deRham cohomology ring $H(T_d,\C)$ as $\C$-algebras.
\end{proof}

\begin{rem}\label{rem_large_volume}
By \cite[Proposition 6.7]{M1}, the holomorphic twist $H(F_{M \oplus I_{d,d} ,p},\bG_{-\ft}^+)$ depends on the embedding 
\begin{align}
M \hookrightarrow \R^{2d,2d}. \label{eq_emb_family}
\end{align}
Thus, the holomorphic twist is not homotopy theoretical invariant of manifolds, but knows the complex / K\"{a}hler geometry \cite{KO}, while \eqref{eq_emb_family} generically produce the same holomorphic twist as in \cite[Proposition 6.7]{M1}, which should correspond to the large volume limit in \cite{Kap05}.
\end{rem}

\subsection{Supersymmetric sigma model on a Kummer surface}\label{sec_example_K3}
A complex $d$-dimensional abelian variety $A$ has the automorphism
\begin{align*}
-1:A \rightarrow A,\quad a \mapsto -a
\end{align*}
which is determined by the group structure. The quotient of $A$ by $-1$ has $2^{2d}$ singular points consisting of
\begin{align}
A[2]=\{a\in A\mid 2a=0\}.\label{eq_Ab_singular}
\end{align}
When $d=2$, this variety is called a Kummer surface and has $16$ singular points, and it admits a crepant resolution, which is a K3 surface (see Example \ref{example_K3}).
Sigma models defined by K3 surfaces are well studied in physics (see \cite{EOTY,W,GHV,Grimm} and references therein).
However, the orbifold constructions for irrational CFTs are not yet available as mathematics, so many of them are only conjectures in mathematics.
If we take a special Abelian variety, the orbifold construction with $-1$ again becomes a lattice full vertex operator algebra, which has a different $N=(2,2)$ structure than the one in Section \ref{sec_example_abelian}.

In this section, we use it to construct a unitary $N=(2,2)$ full VOA satisfying Conjecture \ref{conj_sigma} for the K3 surface. 
The construction of a unitary full VOA for a general K3 surface is left for future work.

%複素$d$次元のabelian variety $A$は群構造から定まる
%\begin{align*}
%-1:A \rightarrow A,\quad a \mapsto -a
%\end{align*}
%という自己同型を持つ。$A$の$-1$による商は$A$の二等分点
%\begin{align*}
%A[2]=\{a\in A\mid 2a=0\}
%\end{align*}
%からなる$2^{2d}$個の特異点にもつ。$d=2$のとき、この多様体は kummer surfaceと呼ばれ$16$個の特異点を持ち、その crepant resolution はK3 surfaceである。
%K3曲面をターゲットとするsigma 模型は物理では良く調べられている \cite{EOTY,W}。
%ただしirrational CFTに対するorbifold構成は数学として、まだ整備されていないため、数学としては多くは conjecture にとどまる。
%Abelian variety として特別なものをとると、虚数乗法を用いることで、$-1$による orbifold 構成が再び lattice full vertex operator algebra になる (ただし Section \ref{sec_example_abelian}とは違う $N=(2,2)$ structure を持つ)。
%この章では、それを用いて Conjecture \ref{conj_sigma}を K3 surface に対して満たす unitary $N=(2,2)$ full VOA を構成する。これは特別な Kummer surface の crepant resolution として得られるSCFTである。一般のK3 surface に対する unitary full VOA の構成 is left for future work.
%

Let $d \in \Z_{\geq 0}$ be an even integer.
Let $N_{d} \subset \R^{3d,3d}$ be a subgroup generated by
\begin{align*}
\{\pm \al_i \pm \al_j, \pm \bal_i \pm \bal_j, \pm \al_i \pm \bal_j, \ft \sum_{l=1}^{2d} (\al_l+\bal_l), \ga_k,\bga_k
\}_{\substack{i,j = 1,\dots,2d,\\ k = 1,\dots,d}},
\end{align*}
which satisfies (L1), (L2), (L3) and contains $(D_{2d} \oplus \Z^d) \oplus \overline{ (D_{2d}\oplus \Z^d)}$, where $D_{2d}$ is the root lattice of type $D_{2d}$.
Set
\begin{align*}
\delta &= \sum_{k=1}^d (\al_{2k-1}+\bal_{2k-1} +\ga_k +\bga_k) \\
&= (\underbrace{1,0,1,0,\dots,1,0}_{2d},\underbrace{1,\dots,1}_{d};
\underbrace{1,0,1,0,\dots,1,0}_{2d},\underbrace{1,\dots,1}_{d}) \in \R^{3d,3d}
\end{align*}
and
\begin{align*}
J =  \sum_{k=1}^d \ga_k\quad\text{ and }\quad \bJ= \sum_{k=1}^d \bga_k.
\end{align*}
Set
\begin{align*}
N_{d}^0 = \{v \in N_d \mid (v,\delta) \in 2\Z\},
\end{align*}
which is an index $2$ subgroup of $N_d$, and let $K_d$ be a subgroup of $\R^{3d,3d}$ generated by
\begin{align*}
N_d^0 \cup \left\{\ft \delta\right\}.
\end{align*}
Then, by construction, $K_d$ also satisfies (L1), (L2) and (L3), and $J,\bJ \in K_d$ since $d$ is even.

\begin{dfn}\label{def_characteristic}
A vector $r$ in an integral lattice $L$ is called a {\bf characteristic vector} if $(\al,\al) = (\al,r)$ (mod $2$) for any $\al\in L$.
\end{dfn}
\begin{lem}\label{lem_characteristic}
The vector $J+\bJ = \sum_{k=1}^d (\ga_k +\bar{\ga}_k)$ is a characteristic vector of $K_d$.
\end{lem}
\begin{proof}
It is clear that $w$ is a characteristic vector for $N_d$, and thus, for $N_d^0$.
For any $u \in N_d^0$, 
\begin{align*}
(u+\ft\delta,u+\ft\delta)&= (u,u)+2(u,\ft \delta)+(\ft\delta,\ft\delta)\equiv (u,u) \equiv (u,J+\bJ),\quad \text{mod (2)}.
\end{align*}
Since $(J+\bJ, \ft \delta) =0$, the assertion holds.
\end{proof}

Let $F_{K_d,p}$ be the unitary full vertex operator algebra associated with $K_d \hookrightarrow \R^{3d,3d}$ (see Proposition \ref{prop_unitary_Narain}), and $\phi$ be the unitary involution.
By Lemma \ref{lem_characteristic}, $J \in (F_{K_d,p})_{1,0}$ and $\bJ \in (F_{K_d,p})_{0,1}$
satisfies Definition \ref{def_N22} (3).

%\begin{itemize}
%\item
%$\phi$は 2-cocycle のとり方によらない。
%\item
%一方で積は 2-cocycle の取り方によっている。$\ga_i$たちがよい OPE を持つようにできるか？
%\end{itemize}
We note that
\begin{align*}
\pm \al_{2k-1}\pm \al_{2k}\pm \ga_k,\quad \pm \bal_{2k-1}\pm \bal_{2k}\pm \bga_k, \in N_d^0 \subset K_d
\end{align*}
and $e_{\pm \al_{2k-1}\pm \al_{2k}\pm \ga_k} \in (F_{K_d,p})_{\frac{3}{2},0}$, 
$e_{\pm \bal_{2k-1}\pm \bal_{2k}\pm \bga_k} \in (F_{K_d,p})_{0,\frac{3}{2}}$, and thus, they are  candidates for supersymmetric currents $\tau^\pm,\btau^\pm$. In fact, we have:
\begin{prop}\label{prop_K3_unitary}
The unitary full vertex operator algebra $F_{K_{d},p}$ admits the following unitary $N=(2,2)$ structure:
$\om = \frac{1}{2} \sum_{i=1}^{2d} h_i(-1)h_i + \ft\sum_{i=1}^d \ga_i(-1)\ga_i$, $\omb = \frac{1}{2} \sum_{i=1}^{2d} \h_i(-1)\h_i + \ft\sum_{i=1}^d \bar{\ga}_i(-1)\bar{\ga}_i$, $J = \sum_{i=1}^d \ga_i$ and $\bJ = \sum_{i=1}^d \bar{\ga}_i$  as in Proposition \ref{prop_sigmaVOA_torus} and 
\begin{align*}
\tau^+ &= -\sum_{k=1}^d (\frac{1+i}{2} (e_{\al_{2k-1}+\al_{2k}+\ga_k}+e_{-\al_{2k-1}-\al_{2k}-\ga_k})+
\frac{1-i}{2}
(e_{\al_{2k-1}-\al_{2k}+\ga_k}+e_{-\al_{2k-1}+\al_{2k}-\ga_k}))\\
\tau^- &= -\sum_{k=1}^d (\frac{1-i}{2} (e_{\al_{2k-1}+\al_{2k}+\ga_k}+e_{-\al_{2k-1}-\al_{2k}-\ga_k})+
\frac{1+i}{2}
(e_{\al_{2k-1}-\al_{2k}+\ga_k}+e_{-\al_{2k-1}+\al_{2k}-\ga_k}))\\
\btau^+ &= \sum_{k=1}^d (\frac{1-i}{2} (e_{\bal_{2k-1}+\bal_{2k}+\bga_k}+e_{-\bal_{2k-1}-\bal_{2k}-\bga_k})-
\frac{1+i}{2}
(e_{\bal_{2k-1}-\bal_{2k}+\bga_k}+e_{-\bal_{2k-1}+\bal_{2k}-\bga_k}))\\
\btau^- &= -\sum_{k=1}^d (\frac{1+i}{2} (e_{\bal_{2k-1}+\bal_{2k}+\bga_k}+e_{-\bal_{2k-1}-\bal_{2k}-\bga_k})-
\frac{1-i}{2}
(e_{\bal_{2k-1}-\bal_{2k}+\bga_k}+e_{-\bal_{2k-1}+\bal_{2k}-\bga_k})).
\end{align*}
\end{prop}
\begin{proof}
The assertion follows from an easy computation (see also Remark \ref{rem_sl2} and Remark \ref{rem_K3_sigma}).
\end{proof}

%我々は$F_{K_{d},p}$が $d=2$の場合に、K3 surface について Conjecture \ref{conj_sigma}を満たすことを示す。$F_{K_d,p}$の定義の物理的動機が知りたい場合、Remark \ref{rem_K3_sigma}をみよ。
We will show that $F_{K_{d},p}$ satisfies Conjecture \ref{conj_sigma} for the K3 surface when $d=2$ (see Remark \ref{rem_K3_sigma} for the physics motivation for the definition of $F_{K_d,p}$).
% If you want to know the physical motivation for the definition of $F_{K_d,p}$, see Remark .
\begin{rem}\label{rem_K3_sigma}
Let $F_{M\oplus I_{d,-d},p}$ be the unitary $N=(2,2)$ full vertex operator algebra in Section \ref{sec_example_abelian}.
Let $\si_1$ be the order $2$ automorphism of $F_{M,p}$ such that $e_\al \mapsto e_{-\al}$ for any $\al \in M$ as in 
\eqref{eq_grouplattice_theta}, which can be extended onto $F_{M\oplus I_{d,-d},p}$.
Then,
\begin{align*}
\si_1 \otimes (-1)^F: F_{M\oplus I_{d,-d},p} \rightarrow F_{M\oplus I_{d,-d},p}
\end{align*}
is an order 2 automorphism such that it preserves $\om,\omb,J,\bJ,\tau^\pm,\btau^\pm$.
In physics, in the case of $d=2$, the $\Z_2$-orbifold of $F_{M\oplus I_{d,-d},p}$ associated with $\si_1 \otimes (-1)^F$ is expected to be the sigma model of the crepant resolution of the Kummer surface.

In the case of $M=N_d$, the inner automorphism group of $F_{N_d \oplus I_{d,-d},p}$ is $\mathrm{SO}(4d)\times \mathrm{SO}(4d)$.
In particular, $\si_1 \in \mathrm{SU}(2)^d \times \mathrm{SU}(2)^d$ is an inner automorphism.
Hence, by an easy calculation, we can confirm that $F_{K_d,p}$ is isomorphic to the $\Z_2$-orbifold of $F_{N_d,p}$ associated with $\si_1 \otimes (-1)^F$,
whose $N=(2,2)$ structure is determine by Remark \ref{rem_sl2}.
\end{rem}

\begin{rem}\label{rem_sl2}
Let $\sll_2=\{e,f,h\}$ be the $\mathrm{sl}_2$-triple and $\si_1$ the automorphism of $\mathrm{sl}_2$ given by
\begin{align*}
\si_1 :\{e,f,h\} \mapsto \{f,e,-h\}.
\end{align*}
It is easy to see that 
$\si_1 = U^{-1}\exp(\ft \pi i \ad(h)) U$ with $U=
\begin{pmatrix}
\frac{\sqrt{2}}{2} & \frac{\sqrt{2}}{2}\\
-\frac{\sqrt{2}}{2} & \frac{\sqrt{2}}{2}
\end{pmatrix}
= \exp(\frac{\pi i}{4} (\ad (e-f)))$.
The supersymmetric currents in Proposition \ref{prop_K3_unitary} is determined by Proposition \ref{prop_sigmaVOA_torus} and $U(h)=-(e+f)$.
\end{rem}

Although $\tau^\pm,\btau^\pm$ are complicated, in order to compute the cohomology ring of $F_{K_d,p}$,
it suffices to classify vectors satisfying the following conditions by Proposition \ref{prop_isom_ring}:
\begin{align*}
C(K_d)=\{\be \in K_d \mid (p\be,p\be) = (p\be,J)\quad\text{and}\quad (\p\be,\p\be) = (\p\be,\bJ)\}.
\end{align*}
\begin{lem}\label{lem_Kd_property}
The subset $C(K_d) \subset K_d$ consisting of the following vectors:
\begin{align}
&\{\ga_{\mu_1}+\dots+\ga_{\mu_n}+\bga_{\nu_1}+\dots+\bga_{\nu_m} \} \label{eq_primaryK3_1} \\
\begin{split}
&\ft \sum_{k=1}^d ({\ep_k} \al_{2k-1} + \ga_k + {\bar{\ep}_{k}} \bal_{2k-1}+ \bga_k)\\
&\ft \sum_{k=1}^d ({\ep_k} \al_{2k} + \ga_k + {\bar{\ep}_{k}} \bal_{2k}+ \bga_k)
\end{split}
\label{eq_primaryK3_2}
\end{align}
where $0 \leq n,m \leq d$ with $n+m \in 2\Z$, $1 \leq \mu_1 < \mu_2< \cdots <\mu_n \leq d, 1 \leq \nu_1 < \nu_2< \cdots <\nu_n \leq d$
and $\ep_k,\bar{\ep}_k \in \{1,-1\}$ with $\sum_k (\ep_k+\bar{\ep}_k) \in 2\Z$.
\end{lem}
Note that there are $2^{2d}$ vectors of the form \eqref{eq_primaryK3_2} which coincides with the number of singular points \eqref{eq_Ab_singular}.
The following proposition follows from Lemma \ref{lem_Kd_property}:
\begin{prop}\label{prop_coho_K3}
Let $d \in 2\Z_{\geq 0}$. Then, the basis of $H(F_{K_d,p},d_B)$ is given by Lemma \ref{lem_Kd_property} and
\begin{align*}
\dim H^{p,q}(F_{K_d,p}) = 
\begin{cases}
0 & p+q \text{ is odd}\\
\binom{d}{p}\binom{d}{q} & p+q \text{ is even and }(p,q)\neq (\frac{d}{2},\frac{d}{2})\\
\binom{d}{d/2}\binom{d}{d/2} +2^{2d} & (p,q)=(d/2,d/2)
\end{cases}
\end{align*}
and
\begin{align*}
\chi_y(F_{K_d}) &=(1+y)^{d}2^{d-1}+2^{2d}y^{d/2}
\end{align*}
In particular, $F_{K_d,p}$ is of CY type.
\end{prop}

By the above proposition, in the case of $d=2$, the Hodge numbers and the algebra structure of $H(F_{K_2,p},d_A)$ coincides with those of K3 surfaces \footnote{The quantum cohomology ring of a K3 surface $Z$ coincides with the deRham cohomology since (1) $Z$ are simply-connected and (2) the dimension of the moduli space of genus zero stable maps to $Z$ is 2.} (see Example \ref{example_K3}).

Thus, to see that $F_{K_2,p}$ satisfies Conjecture \ref{conj_sigma} for K3 surface, it is sufficient to show that the Witten index is a weak Jacobi form of weight 0 index $1$ (see Proposition \ref{prop_Witten_Hirz} and Example \ref{example_K3} \eqref{eq_Jacobi_unique}).

Recall that the theta function is defined by
\begin{align*}
\Theta_{ab}(z,\tau)=\Theta_{ab}(y,q) =\sum_{n \in \Z} e^{\pi i b (n+ \ft a)}q^{\ft(n+\ft a)^2}y^{n+\ft a}
\end{align*}
for any $a,b \in \{0,1\}$ and the eta function by
\begin{align*}
\eta(\tau) =\eta(q)= q^{\frac{1}{24}}\prod_{n=1}^\infty(1-q^n).
\end{align*}
Set $\Theta_{ab}(q) = \Theta_{ab}(1,q)$.
Then, we have:
\begin{thm}\label{thm_K3}
The Witten index of $F_{K_2,p}$ is given by
\begin{align}
\ind_{F_{K_2},p}(y,q)
&= 2\eta(q)^{-6}\left(
\Theta_{00}^2(q)\Theta_{01}^2(q)\Theta_{10}^2(y,q)
+\Theta_{00}^2(q)\Theta_{10}^2(q)\Theta_{01}^2(y,q)
+\Theta_{10}^2(q)\Theta_{01}^2(q)\Theta_{00}^2(y,q)
\right),
\label{eq_index_K3_formula}
\end{align}
which is a weak Jacobi form of weight $0$ and index $1$. In particular, the unitary $N=(2,2)$ full vertex operator superalgebra $F_{K_2,p}$ satisfies Conjecture \ref{conj_sigma} for the K3 surface.
\end{thm}

In the remainder of this section we prove this theorem.
If \eqref{eq_index_K3_formula} is shown, then 
by the transformation formula of the theta functions with respect to $(z,\tau) \mapsto (\frac{z}{\tau},\frac{-1}{\tau})$ \cite{Mum}:
\begin{align*}
\eta\left(\frac{-1}{\tau}\right)^6 &= (-i)^3\tau^3 \eta(\tau)^6\\
\Theta_{00}^2\left(\frac{-1}{\tau},\frac{z}{\tau}\right) &= (-i) \tau \exp\left(2\pi i \frac{z^2}{\tau}\right) \Theta_{00}^2(\tau,z)\\
\Theta_{10}^2\left(\frac{-1}{\tau},\frac{z}{\tau}\right) &= (-i) \tau \exp\left(2\pi i \frac{z^2}{\tau}\right) \Theta_{01}^2(\tau,z)\\
\Theta_{01}^2\left(\frac{-1}{\tau},\frac{z}{\tau}\right) &= (-i) \tau \exp\left(2\pi i \frac{z^2}{\tau}\right) \Theta_{10}^2(\tau,z),
\end{align*}
we have
\begin{align*}
\phi_{F_{K_2},p}\left(\frac{-1}{\tau},\frac{z}{\tau}\right) = \exp\left(2\pi i \frac{z^2}{\tau}\right) \phi_{F_{K_2},p}.
\end{align*}
Thus, by Theorem \ref{thm_CY_VOA}, $\phi_{F_{K_2},p}$ is a weak Jacobi form of weight $0$ index $1$, which coincides with the elliptic genus of K3 surfaces by \eqref{eq_Jacobi_unique}. Thus, it suffices to show that \eqref{eq_index_K3_formula}.
Set
\begin{align*}
P(K_d) &=  \{\be \in K_d \mid (\p\be,\p\be) = (\p\be,\p \bJ)\}.
\end{align*}
Then, by \eqref{eq_ind_discrete}, we have:
\begin{align}
\begin{split}
\phi_{F_{K_d},p}(\tau,z)
&= \eta(\tau)^{-3d}\sum_{\be \in P(K_d)}(-1)^{(\be,J+\bJ)}
q^{\ft((\be,\be)-(J,\be))+\frac{3d}{24}}y^{(J,\be)-d/2}\\
&= \eta(\tau)^{-3d}\sum_{\be \in P(K_d)}(-1)^{(\be,J+\bJ)}
q^{\ft(\be- \ft J ,\be-\ft J)}y^{(J,\be-\ft J)}.
\end{split}
\label{eq_ind_Kd}
\end{align}
Set $L_d = \ker \p \cap K_d =
K_d \cap (\R^{3d}\oplus 0)$,
the intersection with the maximal positive-definite subspace. The following lemma is obvious:
\begin{lem}
The following properties hold for $P(K_d)$:
\begin{enumerate}
\item
If $\al,\be \in P(K_d)$ satisfy $\p(\al)=\p(\be)$.
Then, $\al -\be \in L_d$.
\item
The image of $P(K_d)$ by the map $\p: \R^{3d,3d} \rightarrow \R^{3d}$ is
\begin{align}
\{\bga_{\mu_1}+\dots+ \bga_{\mu_l} \mid 0\leq l \leq d, 1\leq  \mu_1< \cdots <\mu_l \leq d \}\label{eq_finite_image1}
\\
\sum_{k=1}^d (\ep_k \ft \bal_{2k-1}+ \ft \bga_k)\quad \text{ or }\quad
\sum_{k=1}^d (\ep_k \ft \bal_{2k}+ \ft \bga_k),
\label{eq_finite_image2}
\end{align}
with $\ep_k \in \{1,-1\}$ ($k=1,\dots,d$). In particular, the image is a finite set consisting of $2^d+2^{d+1}$ elements.
\item
The fiber of \eqref{eq_finite_image1} is  
\begin{align*}
p^{-1}(\bga_{\mu_1}+\dots+ \bga_{\mu_l}) =
\begin{cases}
L_d & ($l$ \text{ is even}),\\
\ga_1 +L_d & ($l$ \text{ is odd})
\end{cases}
\end{align*}
and \eqref{eq_finite_image2} is
\begin{align*}
p^{-1}(\sum_{k=1}^d (\ep_k \ft \bal_{2k-1}+ \ft \bga_k))&= \sum_{k=1}^d (\ep_k \ft \al_{2k-1}+ \ft \ga_k)+L_d\\
p^{-1}(\sum_{k=1}^d (\ep_k \ft \bal_{2k}+ \ft \bga_k))&= \sum_{k=1}^d (\ep_k \ft \al_{2k}+ \ft \ga_k)+L_d.
\end{align*}
\end{enumerate}
\end{lem}
Set
\begin{align*}
L_d^0 = \{\be \in L_d\mid (\be,\be) \in 2\Z \}.
\end{align*}
Hereafter, we consider the case $d=2$. In this case, the lattice has larger symmetry and the computation is simpler. It is clear that
\begin{align*}
L_2^0 &= \{\be \in \Z^{6} \mid 
(\be,(101000)),(\be,(010100)),(\be,(000011))\in 2\Z\},
\end{align*}
which is isomorphic to $A_1^6 = (\sqrt{2}\Z)^6$ and $2\Z^6 \subset L_2^0$. 
Let $P(K_d)^0$ (resp. $P(K_d)^1$) be the subset of $P(K_d)$ consisting of all the elements whose image of $\p$ is of the form \eqref{eq_finite_image1} (resp. \eqref{eq_finite_image2}).
Then, by
\begin{align*}
\sum_{n\in \Z} q^{\ft(2n)^2}y^{2n}&=
\ft(\Theta_{00}(y,q)+\Theta_{01}(y,q))\\
\sum_{n\in \Z} q^{\ft(2n+1)^2}y^{2n+1}&= \ft(\Theta_{00}(y,q)-\Theta_{01}(y,q))\\
\sum_{n\in \Z} q^{\ft(2n\pm \ft)^2}y^{2n \pm \ft}&=
 \ft(\Theta_{10}(y,q)\pm \frac{1}{i}\Theta_{11}(y,q)),
\end{align*}
We have
\begin{align*}
&\sum_{\be \in P(K_d)^1}(-1)^{(\be,J+\bJ)}
q^{\ft(\be- \ft J ,\be-\ft J)}y^{(J,\be-\ft J)}\\
&= \Theta_{10}^2(q)\left(\Theta_{00}^2(q)+\Theta_{01}^2(q)\right)\left(\Theta_{00}^2(y,q)+\Theta_{01}^2(y,q)\right)
- 
\Theta_{10}^2(q)\left(\Theta_{00}^2(q)-\Theta_{01}^2(q)\right)\left(\Theta_{00}^2(y,q)-\Theta_{01}^2(y,q)\right)\\
&=2\Theta_{10}^2(q)\Theta_{01}^2(q)\Theta_{00}^2(y,q)+2\Theta_{10}^2(q)\Theta_{00}^2(q)\Theta_{01}^2(y,q)
\end{align*}
and similarly,
\begin{align*}
&\sum_{\be \in L_d}(-1)^{(\be,J+\bJ)}
q^{\ft(\be- \ft J ,\be-\ft J)}y^{(J,\be-\ft J)}\\
&= \frac{1}{8} (\Theta_{00}^2(q)+\Theta_{01}^2(q))^2 (\Theta_{10}^2(y,q)-\Theta_{11}^2(y,q)) -
\frac{1}{8}(\Theta_{00}^2(q)-\Theta_{01}^2(q))^2 (\Theta_{10}^2(y,q)+\Theta_{11}^2(y,q))
\end{align*}
and
\begin{align*}
&\sum_{\be \in L_d+\ga_1}(-1)^{(\be,J+\bJ)}
q^{\ft(\be- \ft J ,\be-\ft J)}y^{(J,\be-\ft J)}\\
&= \frac{1}{8} (\Theta_{00}^2(q)+\Theta_{01}^2(q))^2 (\Theta_{10}^2(y,q)+\Theta_{11}^2(y,q))
-
\frac{1}{8}(\Theta_{00}^2(q)-\Theta_{01}^2(q))^2 (\Theta_{10}^2(y,q)-\Theta_{11}^2(y,q)).
\end{align*}
Hence,
\begin{align*}
\ind_{F_{K_2,p}}(y,q) &=2\Theta_{00}^2(q)\Theta_{01}^2(q)\Theta_{10}^2(y,q)
+2\Theta_{00}^2(q)\Theta_{10}^2(q)\Theta_{01}^2(y,q)
+2\Theta_{10}^2(q)\Theta_{01}^2(q)\Theta_{00}^2(y,q).
\end{align*}
%where we used the Jacobi identity
%\begin{align*}
%\Theta_{01}^4(q)+\Theta_{10}^4(q)=\Theta_{00}^4(q).
%\end{align*}
Hence, we have \eqref{eq_index_K3_formula}.
Note that since
\begin{align*}
\Theta_{00}(y,q) &= 1+q\Z[[q]][y^\pm],\\
\Theta_{10}(y,q) &= q^{\frac{1}{8}}((y^\ft +y^{-\ft}) +y^\ft q\Z[[q]][y^\pm]),\\
\Theta_{01}(y,q) &= 1+q\Z[[q]][y^\pm],
%\Theta_{11}(y,q) &= q^{\frac{1}{8}}(iy^{\ft}-iy^{-\ft}+ y^\ft q\Z[[q]][y^\pm] )
\end{align*}
we have
\begin{align*}
\lim_{q \to 0} \ind_{F_{K_2,p}}(y,q) =  2(y^\ft+y^{-\ft})^2+8+8 
=2y+20+2y^{-1},
\end{align*}
which is consistent with Proposition \ref{prop_coho_K3} and Example \ref{example_K3}.

\subsection{Landau-Ginzburg model of type $A_2$}
\label{sec_LG}
In this section, we will construct a unitary $N=(2,2)$ full vertex operator algebra associated with one of the simplest non-trivial singularity:
\begin{align*}
x^2+y^2+z^{3}=0,
\end{align*}
which is called the Landau-Ginzburg model \cite{VW}.
The important point is that this full VOA is not of CY type.

Let $V_{\sqrt{3}\Z}$ be the lattice vertex operator algebra associated with the rank one lattice $\Z\al$ with $(\al,\al)=3$. It is well-known that $V_{\sqrt{3}\Z}$ is an $N=2$ vertex operator algebra.
Let
\begin{align*}
M_1^{N=2}= \bigoplus_{k=0,1,2} V_{\sqrt{3}\Z+\frac{k}{\sqrt{3}}}\otimes \overline{V_{\sqrt{3}\Z+\frac{k}{\sqrt{3}}}}
\end{align*}
be the unitary full vertex operator superalgebra constructed in Proposition \ref{prop_unitary_Narain},
which is unitary $N=(2,2)$ as well.
Then, it is easy to show that
all the c-c primary vectors are exhausted by
\begin{align*}
e_{(\frac{1}{\sqrt{3}},\frac{1}{\sqrt{3}})}\in M_1^{N=2}
\end{align*}
and the vacuum $\va$,
while all c-a primary vectors are exhausted just by $\va$. Hence, we have:
\begin{prop}\label{prop_LG2}
$M_1^{N=2}$ is a unitary $N=(2,2)$ full VOA, and
\begin{align*}
H(M_1^{N=2},d_B)&= \C \va \oplus \C e_{(\frac{1}{\sqrt{3}},\frac{1}{\sqrt{3}})} \cong \C[z]/z^2 \\
H(M_1^{N=2},d_A)&\cong \C\va
\end{align*}
with $H^{\frac{1}{3},\frac{1}{3}}(M_1^{N=2},d_B) =\C  e_{(\frac{1}{\sqrt{3}},\frac{1}{\sqrt{3}})}$.
\end{prop}

\begin{rem}
\label{rem_Jacobi_ring}
Let $M_k^{N=2}$ be the diagonal model of the unitary $N=2$ VOA of central charge $\frac{3k}{k+2}$.
In physics, it is well-known that 
$M_k^{N=2}$ is the Landau-Ginzburg model of the Kleinian singularity of type $A_{k+1}$ \cite{VW}:
\begin{align*}
x^2+y^2+z^{k+2}=0,
\end{align*}
that is, set $W_k(x,y,z) = x^2+y^2+z^{k+2}$.
Then,
\begin{align*}
H(M_k^{N=2},d_B) &\cong \C[x,y,z]/ \langle \pa_x(W_k),\pa_y(W_k),\pa_z(W_k)\rangle\\
&= \C[x,y,z]/\langle x,y,z^{k+1}\rangle.
\end{align*}
The latter $\C$-algebra is called a Jacobi ring of singularities. 
Since
\begin{align*}
W(\la^\ft x,\la^\ft y, \la^{\frac{1}{k+2}} z) = \la W(x,y,z),
\end{align*}
it is natural to define the grading of $z$ by $\frac{1}{k+2}$ from the singularity theory point of view.
This is the origin of the non-trivial $\R^2$-grading in Proposition \ref{prop_LG2} \cite{VW}.
\end{rem}

%\begin{itemize}
%\item
%spectral VOA の情報を落としたバージョンとして, $\Z_r \mathrm{Spin}$ VOA が定義できる. Majorana fermion はこの例になっており、Riemann 面に乗せた分配関数を Shkotty coordinate で計算すると、ちゃんと (higher genus)の theta 関数が出てくる.
%\end{itemize}

\appendix

\section{Unitarity of spectral flow}
In appendix, we define unitarity of a generalized full vertex operator superalgebra (GFV), and show that 
\begin{itemize}
\item
$\Om_F$ in Theorem \ref{thm_vacuum} is unitary if
a full $\mathcal{H}$ vertex operator algebra $(F,H)$ is unitary;
\item
$G_{H,p}$ in Proposition \ref{standard} is unitary GFV if $(H_l,(-,-)_l)$ and $(H_l,(-,-)_l)$ are positive-definite;
\item
A twisted group algebra $\C[\hat{L}]$ associated with an (possible indefinite) integral lattice $L$ is unitary GFV;
\item
A tensor product of unitay GFV is unitary GFV,
\end{itemize}
which imply that the full vertex operator superalgebras in Section \ref{sec_example} is unitary.

\begin{dfn}\label{def_unitary_chiral_susy}
A \textbf{unitary} full $\mathcal{H}$ vertex operator superalgebra is a full $\mathcal{H}$ vertex operator superalgebra $(F,\om,\omb,H_l,H_r)$ which is unitary as a full vertex operator superalgebra by an invariant bilinear form $(\bullet,\bullet)$ and an anti-linear involution $\phi:V \rightarrow V$ such that:
\begin{align}
\phi(h)=-h,\quad \phi(\h) = -\h
\label{eq_H_unitary}
\end{align}
for any $h\in H_l$ and $\h\in H_r$.
\end{dfn}
Note that a unitary $N=(2,2)$ full vertex operator superalgebra is an example of a unitary $\mathcal{H}$-vertex operator superalgebra by $H_l=\R J$ and $H_r=\R \bJ$.

Let $F$ be a unitary full $\mathcal{H}$ vertex operator superalgebra. Hereafter, we will show that $\Om_F$ in Theorem \ref{thm_vacuum} also inherits a unitary structure. It is clear that the restriction of $(-,-)$ on $\Om_F$ is non-degenerate. For any $v \in \Om_F^\al$ and $h\in H_l$, by \eqref{eq_H_unitary}
\begin{align*}
h(n)\phi(v) = \phi(-h(n)v) =
\begin{cases}
0 & n \geq 1\\
-(h,\al) v & n=0.
\end{cases}
\end{align*}
Hence, $\phi$ maps $\Om_F^\al$ onto $\Om_F^{-\al}$. Moreover, for any $a\in \Om_F^\al$ and $b\in \Om_F^\be$,
\begin{align}
\begin{split}
\phi(\hY_\Om(a,\uz)b) &= \phi(E^-(-\al,\uz)Y(a,\uz) E^+(-\al,\uz)z^{-(\al,\be)_l}\z^{-(\al,\be)_r} b)\\
& =E^-(\al,\uz)Y(\phi(a),\uz) E^+(\al,\uz)z^{-(\al,\be)_l}\z^{-(\al,\be)_r} \phi(b)\\
&=\hY_\Om(\phi(a),\uz) z^{-p\al(0)}\z^{-\p\al(0)}z^{-(\al,\be)_l}\z^{-(\al,\be)_r} \phi(b)\\
&=\hY_\Om(\phi(a),\uz)\phi(b).
\end{split}
\label{eq_generalized_auto}
\end{align}
Thus, $\phi$ is an anti-linear automorphism of the generalized full vertex operator algebra $\Om_F$.
Let $a_i \in (\Om_F)_{t_i,\td_i}^{\al_i}$. Recall that $t_i=h_i - \frac{(\al_i,\al_i)_l}{2}$ and $\td_i=\h_i - \frac{(\al_i,\al_i)_r}{2}$, where $(h_i,\h_i)$ is the conformal weight in $F$.
Then, we have
\begin{align}
\begin{split}
&(a_1,\hY_\Om(a_2,\uz)a_3)\\ &= 
(a_1, E^-(-\al_2,\uz) Y(a_2,\uz)E^+(-\al_2,\uz)z^{-p\al_2(0)}\z^{-\p\al_2(0)} a_3)\\
& =z^{-(\al_2,\al_3)_l}\z^{-(\al_2,\al_3)_r} (a_1, Y(a_2,\uz)a_3)\\
&=z^{-(\al_2,\al_3)_l}\z^{-(\al_2,\al_3)_r}
(Y(\exp(L(1)z+\Ld(1)\z) (-1)^{(L(0)-\Ld(0))+2(L(0)-\Ld(0))^2} z^{-2L(0)}\z^{-2\Ld(0)} a_2, \uz^{-1}) a_1, a_3)\\
&=z^{-(\al_2,\al_3)_l}\z^{-(\al_2,\al_3)_r}
z^{-2h_2}\z^{-2\h_2} (-1)^{(h_2-\h_2)+2(h_2-\h_2)^2}\\
&(E^-(\al_2,\uz^{-1}) \hY_\Om(\exp(L(1)z+\Ld(1)\z)a_2, \uz^{-1})E^+(\al_2,\uz^{-1})z^{-p\al_2(0)} \z^{-\p\al_2(0)} a_1, a_3)\\
&=z^{-(\al_2,\al_3+\al_1)_l}\z^{-(\al_2,\al_3+\al_1)_r}
z^{-2t_2-(\al_2,\al_2)_l}\z^{-2\td_2-(\al_2,\al_2)_r}
 (-1)^{(h_2-\h_2)+2(h_2-\h_2)^2}(\hY_\Om(\exp(L(1)z+\Ld(1)\z)a_2, \uz^{-1}) a_1, a_3)\\
&=z^{-2t_2}\z^{-2\td_2} (-1)^{(h_2-\h_2)+2(h_2-\h_2)^2}(\hY_\Om(\exp(L_\Om(1)z+\Ld_\Om(1) \z)a_2, \uz^{-1}) a_1, a_3),
\end{split}
\label{eq_inv_generalized}
\end{align}
where in the last line, we used $L(1) = L_H(1)+L_\Om(1)$ and \eqref{eq_inv_generalized} is zero unless $\al_1+\al_2+\al_3=0$.

\begin{dfn}\label{def_unitary_generalized_full}
Let $(\Om,\hY,H,(-,-)_c)$ be a generalized full vertex operator superalgebra.
\begin{itemize}
\item
A symmetric bilinear form $(\bullet,\bullet):\Om \otimes \Om \rightarrow \C$ is called {\bf invariant} if
\begin{enumerate}
\item
For any $u\in \Om^\al$ and $v\in \Om^\be$, $(u,v)=0$ unless $\al +\be =0$;
\item
For any $a,u,v \in \Om$,
\begin{align}
(u,\hY(a,\uz)v) =(-1)^{n(a)+2n(a)^2} 
 (\hY(\exp(L_\Om(1)z+\Ld_\Om(1)\z) z^{-2L_\Om(0)}\z^{-2\Ld_\Om(0)}a, \uz^{-1})u,v),
\label{eq_full_generalized_inv_def}
\end{align}
where $n_a = t_a-\td_a + \frac{(\al,\al)_c}{2} \in \ft\Z$ with $a\in \Om_{t_a,\td_a}^\al$.
\end{enumerate}
\item
An {\bf anti-linear automorphism} $\phi$ of $\Om$ is an anti-linear map $\phi:\Om \rightarrow \Om$ such that
 $\phi(\va) = \va, \phi(\om)=\om, \phi(\omb)=\omb$ and $\phi(a(r,s)b) = \phi(a)(r,s)\phi(b)$ for any $a,b\in \Om$ and $r,s \in \R$
 and $\phi$ maps $\Om_{t,\td}^\al$ onto $\Om_{t,\td}^{-\al}$ for any $t,\td \in \R$ and $\al \in H$.
\item
Let $\Om$ be a generalized vertex operator superalgebra with an invariant bilinear form $(\bullet,\bullet)$
and $\phi:\Om \rightarrow \Om$ be an anti-linear involution, i.e.\! an anti-linear automorphism of order 2. 
The pair $(\Om,\phi)$ is called \textbf{unitary} if the sesquilinear form $\langle \bullet,\bullet \rangle = (\phi(\bullet),\bullet)$ is positive-definite.
\end{itemize}
\end{dfn}

Then, we have:
\begin{prop}
\label{prop_om_full_unitary}
Let $F$ be a unitary $N=(2,2)$ full vertex operator superalgebra. Then, the associated generalized full vertex operator algebra $\Om_F$ is unitary by the restrictions of the bilinear form and the involution.
\end{prop}
Then, similarly to the proof of \cite[Propostion 2.16]{M9}, we have:
\begin{prop}\label{prop_unitary_tensor}
Let $\Om_1$ and $\Om_2$ be generalized full vertex operator algebras equipped with invariant bilinear forms $(-,-)_1$ and $(-,-)_2$. Then, the bilinear form $(-,-)$ on $\Om_1\otimes \Om_2 $ given by
\begin{align*}
(a\otimes b,c\otimes d) = (a,c)_1(b,d)_2,\quad \quad
\text{for $a,c \in \Om_1$ and $b,d\in \Om_2$}
\end{align*}
 is invariant. Moreover, if $\Om_1$ and $\Om_2$ are unitary with the anit-linear involution $\phi_1$ and $\phi_2$, then $\Om_1\otimes \Om_2$ is unitary by $\phi_1 \otimes \phi_2$.
\end{prop}

\begin{comment}
\begin{proof}
It suffices to consider \eqref{eq_inv_generalized} in the full case.
Let $a_i \in (\Om_F)_{t_i,\td_i}^{\al_i}$. Recall that $t_i=h_i - \frac{(p\al_i,p\al_i)_l}{2}$ and $\td_i=\h_i - \frac{(\p\al_i,\p\al_i)_r}{2}$, where $(h_i,\h_i)$ is the conformal weight in $F=\bigoplus_{h,\h\in \R}F_{h,\h}$.
Recall that $(-,-)_c: H \otimes H \rightarrow \C$ is given by $(\al,\be)_c= (p\al,p\be)_l-(\p\al,\p\be)_r$.
Set $n_2= t_2-\td_2+\frac{(\al_2,\al_2)_c}{2}$.
Then, we have
\begin{align*}
&(a_1,\hY_\Om(a_2,\uz)a_3)\\ &= 
(a_1, E^-(-\al_2,\uz) Y(a_2,\uz)E^+(-\al_2,\uz)z^{-p\al_2(0)}\z^{-\p\al_2(0)} a_3)\\
& =z^{-(p\al_2,p\al_3)_l}\z^{-(\p\al_2,\p\al_3)_r} (a_1, Y(a_2,z)a_3)\\
&=z^{-(p\al_2,p\al_3)_l}\z^{-(\p\al_2,\p\al_3)_r} (Y(\exp(L(1)z+\Ld(1)\z) (-1)^{(L(0)-\Ld(0))+2(L(0)-\Ld(0))^2} z^{-2L(0)}\z^{-2\Ld(0)}a_2, \uz^{-1}) a_1, a_3)\\
&=z^{-(p\al_2,p\al_3)_l}\z^{-(\p\al_2,\p\al_3)_r} (-1)^{n_2-2n_2^2}z^{-2h_2}\z^{-2\h_2}
 (Y(\exp(L(1)z+\Ld(1)\z)a_2, \uz^{-1}) a_1, a_3)\\
&=z^{-(p\al_2,p\al_3+p\al_2)_l}\z^{-(\p\al_2,\p\al_3+\p\al_2)_r} (-1)^{n_2-2n_2^2}z^{-2t_2}\z^{-2\td_2}
z^{-(p\al_2,p\al_1)_l}\z^{-(\p\al_2,\p\al_1)_r} (Y_\Om(\exp(L(1)z+\Ld(1)\z)a_2, \uz^{-1}) a_1, a_3)\\
&=(-1)^{n_2-2n_2^2}z^{-2t_2}\z^{-2\td_2}(Y_\Om(\exp(L_\Om(1)z+\Ld_\Om(1)\z)a_2, \uz^{-1}) a_1, a_3)\\
\end{align*}
where in the last line, we used $L(1) = L_J(1)+L_\Om(1)$, $\Ld(1) = \Ld_J(1)+\Ld_\Om(1)$ and the bilinear form is zero unless $\al_1+\al_2+\al_3=0$.
\end{proof}
\end{comment}

We will give two examples of unitary generalized full vertex operator superalgebras, which is a generalization of \cite{AMT} showing that the full VOA corresponding to Narain CFT is unitary.

First, we will show that $G_{H,p}$ is unitary (Definition \ref{def_unitary_generalized_full}).
Let $H_l$ and $H_r$ be finite dimensional real vector spaces equipped with non-degenerate symmetric bilinear forms $(-,-)_l$ and $(-,-)_r$. 
In this section, we will assume:
\begin{itemize}
\item
$(-,-)_l$ and $(-,-)_r$ are positive-definite.
\end{itemize}

Recall that $H=H_l\oplus H_r$ and 
\begin{align*}
(-,-)_c: H \otimes H \rightarrow \R, \quad (\al,\be)_c=(p\al,p\be)_l - (\p\al,\p\be)_r, \text{ for }\al,\be \in H,
\end{align*}
where $p:H \rightarrow H_l$ and $\p:H \rightarrow H_r$ are projections,
\begin{align*}
G_{H,p} =\mathrm{Ind}\,\R[H] = \bigoplus_{\al \in H} M_{H,p}(\al),
\end{align*}
the induced representation of affine Heisenberg Lie algebra from the group algebra.
Note that $(G_{H,p},\hY,H,-(-,-)_c)$ is constructed, here, over $\R$.
%Then, 
%\begin{itemize}
%\item
%Remark on 係数体
%\end{itemize}

Note that the conformal weight of $e_\al \in \R[H] \subset G_{H,p}$ is $(\frac{(p\al,p\al)_l}{2},\frac{(\p\al,\p\al)_r}{2})$ and 
\begin{align}
n(e_\al) = (\frac{(p\al,p\al)_l}{2}-\frac{(\p\al,\p\al)_r}{2})+(- \frac{(\al,\al)_c}{2})=0.
\label{eq_null_generalized}
\end{align}

By embedding $G_{H,p}$ into a full vertex algebra
and 
%This means that when we cancel all the monodromy of the generalized vertex operator algebra $G_{H,p}$ and make a full vertex algebra, we use the affine Heisenberg Lie algebra associated with the bilinear form $-(\al,\be)_c$ of the opposite sign of $H,(-,-)_c$,
\begin{align}
F_{G_{H,p}} =\bigoplus_{\al \in H} M_{H,p}(\al) \otimes M_{-H,p}(\al)
\label{eq_cancel_zero}
\end{align}
and by the proof of the existence of invariant bilinear form \cite{M9},
we can easily obtain:
%In this case, the conformal weight of the lowest weight vector in \eqref{eq_cancel_zero} is always 
%$(0,0)= (\frac{(p\al,p\al)_l}{2}-\frac{(p\al,p\al)_l}{2},\frac{(\p\al,\p\al)_r}{2}-\frac{(\p\al,\p\al)_r}{2})$.
%Since the full vertex algebra $F_{G_{H,p}}$ inherits a symmetric bilinear form, by restriction, we have:
\begin{prop}
\label{prop_exist_inv_GHp}
The generalized full vertex operator algebra $G_{H,p}$ has a unique non-degenerate symmetric invariant bilinear form with $(\va,\va)=1$.
\end{prop}

Since $\R[H] \rightarrow \R[H],\quad e_\al \mapsto e_{-\al}$ is an $\R$-algebra involution,
it can be extended to an automorphism of the generalized full vertex algebra $\phi:G_{H,p} \rightarrow G_{H,p}$ by
\begin{align*}
\phi: &h_l^1(-n_1-1)\dots h_l^l(-n_l-1)h_r^1(-m_1-1)\dots h_r^k(-m_k-1)e_\al\\
&\mapsto (-1)^{l+k} h_l^1(-n_1-1)\dots h_l^l(-n_l-1)h_r^1(-m_1-1)\dots h_r^k(-m_k-1)e_{-\al}
\end{align*}
Set $G_{H,p}^\C = G_{H,p} \otimes_{\R}\C$, which is naturally a generalized full vertex algebra over $\C$.
For $i=0,1$, set 
$G_i =\{a\in G_{H,p} \mid \phi(a)=(-1)^i a\}$,
which are regarded as subspaces of $G_{H,p}^\C$,
 and
\begin{align*}
\tilde{G} = G_0 \oplus i G_1 \subset G_{H,p}^\C.
\end{align*}
Then, $\tilde{G}$ is a real subalgebra of $G_{H,p}^\C$.
It is clear that $\om, \omb \in G_0 \subset \tilde{G}$
and by the restriction of the invariant bilinear form on $G_{H,p}^\C$, 
$\tilde{G}$ has a invariant symmetric bilinear form:
\begin{align*}
(-,-):\tilde{G} \otimes \tilde{G} \rightarrow \C.
\end{align*}
\begin{prop}\label{prop_positive_definite}
The bilinear form on $\tilde{G}$ is real valued and positive-definite.
\end{prop}
\begin{proof}
We first show that the bilinear form is positive-definite on
\begin{align*}
\R (e_\al+e_{-\al}) \oplus \R i(e_\al-e_{-\al}) \subset G_{H,p}
\end{align*}
for any $\al \in H$.
By \eqref{eq_full_generalized_inv_def} and \eqref{eq_null_generalized}, we have:
\begin{align*}
&(e_\al+e_{-\al}, e_\al+e_{-\al})\\
&=\lim_{z\to 0} (\hY_\std(e_\al+e_{-\al},\uz)\va, e_\al+e_{-\al})\\
&=\lim_{z\to 0} z^{-(p\al,p\al)_l}\z^{-(\p\al,\p\al)_r} ( \va, Y_\std(e_\al+e_{-\al},\uz^{-1})(e_\al+e_{-\al}))\\
&= (\va, e_\al\cdot e_{-\al}+e_{-\al}\cdot e_{\al}))= 2.
\end{align*}
Similarly,
$( i(e_\al-e_{-\al}), i(e_\al-e_{-\al})) =2$ and $( (e_\al+e_{-\al}),i(e_\al-e_{-\al}))=0.$
It is clear that $\tilde{G}$ is spanned as a real vector space by
\begin{align*}
G_\al^+ &= \mathrm{Span}_\R \{ i^{l+k} h_l^1(-n_1-1)\dots h_l^l(-n_l-1)h_r^1(-m_1-1)\dots h_r^k(-m_k-1)(e_\al+e_{-\al})\}\\
G_\al^- &= \mathrm{Span}_\R \{ i^{l+k+1} h_l^1(-n_1-1)\dots h_l^l(-n_l-1)h_r^1(-m_1-1)\dots h_r^k(-m_k-1)(e_\al-e_{-\al})\}.
\end{align*}
Then, $G_\al^{\eta_1}$ and $G_\be^{\eta_2}$ are orthogonal if 
$\eta_1 \neq \eta_2$ or $\al \neq \pm \be$.
Since $(-,-)_l$ and $(-,-)_r$ are positive-definite,
by the commutator relation of affine Heisenberg Lie algebra,
$G_\al^+$ and $G_\al^-$ are both positive-definite.
\end{proof}
There is a unique anti-linear involution $\phi_\std:G_{H,p}^\C \rightarrow G_{H,p}^\C$ whose fixed point real subalgebra is $\tilde{G}$.
Moreover,
\begin{align*}
\phi_\std(h) &= \phi_\std(-i\cdot ih) = i\phi_\std(ih)= -h\\
\phi_\std(e_{\pm\al}) &= \ft \phi_\std((e_\al+e_{-\al})\mp i \cdot i(e_\al-e_{-\al}))= \ft ((e_\al+e_{-\al})\mp (e_\al-e_{-\al})) = e_{\mp \al}
\end{align*}
 Hence, we have:
\begin{prop}\label{prop_GHp_unitary}
Let $H_l$ and $H_r$ be finite dimensional real vector spaces equipped with symmetric bilinear forms $(-,-)_l$ and $(-,-)_r$. Assume that $H_l$ and $H_r$ are positive-definite.
Then, the generalized full vertex algebra $(G_{H,p},\hY,H,-(-,-)_c)$ is unitary with the anti-linear involution $\phi_\std$.
\end{prop}

Let $H = \R^{n,m}$ be the real vector space equipped with non-degenerate symmetric bilinear form $(-,-)_\lat$ with signature $(n,m)$.
%
%$p:\R^{3d,3d} \rightarrow \R^{3d}$ を $u \mapsto (a_1,\dots,a_{3d})$によって定める。$p$はmaximal positive definite subspace への射影である。Set $\bar{p}=1-p$.
Let $L\subset \R^{n,m}$ be a subgroup such that:
\begin{enumerate}
\item[L1)]
$L$ is free abelian group of rank $n+m$;
\item[L2)]
$L$ is an integral lattice, that is, $(\al,\be)_\lat \in \Z$ for any $\al,\be \in L$.
%\item[L3)]
%$L$ is unimodular, that is, if $\ga \in \R^{3d,3d}$ satisfies $(\al,\ga)\in\Z$ for any $\al\in L$,
%then $\ga \in L$.
\end{enumerate}

Let $\{\al_i\}_{i=1,\dots,n+m}$ be a basis of $L$ and $\ep:L\times L \rightarrow \Z_2 = \R^\times$ be a (non-symmetric) bilinear form defined by
\begin{align}
\begin{split}
\ep(\al_i,\al_i)&= (-1)^{((\al_i,\al_i)_\lat +(\al_i,\al_i)_\lat^2)/2}\quad\quad\quad\text{ for all $i$}\\
\ep(\al_i,\al_j)&= (-1)^{(\al_i,\al_j)_\lat + (\al_i,\al_i)_\lat (\al_j,\al_j)_\lat}\quad\quad\text{ if $i>j$}\\
\ep(\al_i,\al_j)&= 1 \quad\quad\quad\quad\quad\quad\quad\quad\quad\quad\quad\text{ if $j > i$}.
\end{split}\label{eq_def_cocycle2}
\end{align}
Then, $\ep(-,-)$ is a 2-cocycle $Z^2(L,\R^\times)$.
Let $\R[\hat{L}]= \bigoplus_{\al \in L} \R e_\al$ be an $\R$-algebra with the multiplication defined by
\begin{align*}
e_\al \cdot e_\be = \ep(\al,\be)e_{\al+\be}
\end{align*}
for $\al,\be\in L$. This algebra is called a \textit{twisted group algebra} introduced in \cite{FLM} (see also \cite{Kac} for the case of odd lattices).
It is easy to show that
\begin{align}
%\ep(\al,0)=\ep(0,\al)=1,\label{eq_cocycle1}\\
\ep(\al,\be)=\ep(\al,-\be)=\ep(-\al,\be) 
\quad\text{ and }\quad \ep(\al,\al)=(-1)^{((\al,\al)_\lat+(\al,\al)_\lat^2)/2}
\label{eq_minus_ep}\\
\ep(\al,\be)\ep(\be,\al)^{-1} = (-1)^{(\al,\be)+(\al,\al)(\be,\be)}
\label{eq_cocycle}
\end{align}
holds for any $\al,\be \in L$.

The $\R$-algebra $\R[\hat{L}]$ can be viewed as a generalized full vertex operator superalgebra with charge lattice $(H, (-,-)_\lat)$ as follows.
The structure of $H$-graded vector space is naturally determined by embedding $i:L \hookrightarrow H$. 
The $\R\times \R$-grading is trivial, that is, the weight of any vector is $(0,0)$. The vertex operator is given by the left multiplication:
\begin{align*}
Y(e_\al,\uz)e_\be = e_\al \cdot e_\be =\ep(\al,\be)e_{\al+\be}
\end{align*}
By $\ep(\al,0)=\ep(0,\al)=1$, $\R[\hat{L}]$ is a full generalized vertex operator algebra over $\R$ with $e_0$ as the vacuum vector, and the conformal vectors are defined to be $(\om,\omb)=(0,0)$ (see \cite[Section 5.3]{M1} for more detail).

%$\R[\hat{L}]$は、次のように generalized full vertex operator superalgebra with charge lattice $(H, (-,-)_\lat)$ と見ることができる。
%$H$-graded vector space の構造は embedding $i:L \hookrightarrow H$により自然に定まる。$\R\times \R$-grading は全てのベクトルを weight $(0,0)$として入れる。
%頂点作用素は
%\begin{align*}
%Y(e_\al,\uz)e_\be = e_\al \cdot e_\be =\ep(\al,\be)e_{\al+\be}
%\end{align*}
%によって定める。\eqref{eq_cocycle}と$\ep(\al,0)=\ep(0,\al)=1$より、$\R[\hat{L}]$は$e_0$を vacuum vector とする full generalized vertex operator algebra over $\R$ となる (共形ベクトルは$(\om,\omb)=(0,0)$と定める) 
Note that $e_\al$ is odd if and only if $(\al,\al)_\lat$ is an odd integer.
Then, we have:
%We introduce the notion of an invariant bilinear form on a generalized full vertex operator algebra in Appendix (see Definition \ref{}).
\begin{lem}
The symmetric bilinear form on $\R[\hat{L}]$ defined by
\begin{align*}
(e_\al,e_\be) = \delta_{\al+\be,0}\quad\quad\text{ for }\al,\be\in L
\end{align*}
is an  invariant bilinear form on the generalized full vertex operator superalgebra.
\end{lem}
\begin{proof}
For any $\al,\be \in L$,
$(e_{-\al-\be},e_\al \cdot e_\be)=\ep(\al,\be)$ and
by \eqref{eq_minus_ep} and \eqref{eq_cocycle}
\begin{align*}
(-1)^{((\al,\al)_\lat+ (\al,\al)_\lat^2)/2} (e_\al \cdot e_{-\al-\be},e_\be)
&=(-1)^{((\al,\al)_\lat+ (\al,\al)_\lat^2)/2} \ep(\al,-\al-\be)\\
&=(-1)^{((\al,\al)_\lat+ (\al,\al)_\lat^2)/2} \ep(\al,\al+\be)\\
&= \ep(\al,\be).
\end{align*}
Hence, the assertion holds.
\end{proof}
Define a linear map $g:\R[\hat{L}] \rightarrow \R[\hat{L}]$ by
\begin{align}
g(e_\al)= e_{-\al},\label{eq_grouplattice_theta}
\end{align}
which is an $\R$-algebra automorphism by \eqref{eq_minus_ep},
and thus, an automoprphism of the generalized full vertex operator superalgebra.
One can extend $g$ to the generalized full vertex operator superalgebra $\C[\hat{L}]$ over $\C$,
and set 
\begin{align*}
\R[\hat{L}]_\pm = \{a \in \R[\hat{L}]\mid g(a)=\pm a \}.
\end{align*}
Then, the subspace
\begin{align*}
\C[\hat{L}]_{\Re}=
\mathrm{Span}_\R\{e_\al+e_{-\al},i (e_\al-e_{-\al})\}_{\al \in L} = \R[\hat{L}]_+ \oplus i \R[\hat{L}]_- \subset \C[\hat{L}]
\end{align*}
is a real subalgebra.
Since $\C[\hat{L}]_{\Re} \otimes_\R \C$ is naturally isomorphic to $\C[\hat{L}]$,
there is a unique anti-linear map $\phi_\lat: \C[\hat{L}] \rightarrow \C[\hat{L}]$ such that $\phi_\lat|_{\C[\hat{L}]_{\Re}}$ is the identity map. It is clear that the restriction of the induced bilinear form $(-,-):\C[\hat{L}] \otimes \C[\hat{L}] \rightarrow \C$ on $\C[\hat{L}]_{\Re}$ is positive definite. Hence, we have:
\begin{lem}\label{prop_unitary_groupalgebra}
$\C[\hat{L}]$ is a unitary generalized full vertex operator superalgebra with the symmetric invariant bilinear form $(-,-)$
and the anti-linear involution $\phi_\lat$.
\end{lem}

Let $p:\R^{n,m}\rightarrow \R^{n,m}$ be the orthogonal projection such that:
\begin{itemize}
\item
$\ker p$ is positive-definite and $\ker (1-p)$ is negative-definite.
\end{itemize}
Set $H_l =\ker p$ and $H_r =\ker \p$. Then, $H_l,(-,-)_\lat$ and $H_r,-(-,-)_\lat$ are positive-definite.
Set
\begin{align*}
F_{L,i,p} = \bigoplus_{\al \in H}\C e_\al \otimes M_{H,p}(\al),
\end{align*}
which is a subalgebra of the tensor product of generalized full vertex operator superalgebras $\C[\hat{L}]$ and $G_{H,p}$ (see \cite{M1}). 
Then, by Proposition \ref{prop_unitary_tensor}, Proposition \ref{prop_GHp_unitary} and Lemma \ref{prop_unitary_groupalgebra}, we have:
\begin{prop}\label{prop_unitary_Narain}
Let $L \subset \R^{n,m}$ be an integral lattice of rank $n+m$ and $p:\R^{n,m}\rightarrow \R^{n,m}$ be the orthogonal projection such that $\ker(p)$ is the maximal positive-definite subspace of $\R^{n,m}$.
Then, $F_{L,i,p}$ is a unitary full vertex operator superalgebra with the anti-linear involution $\phi=\phi_{\std} \otimes \phi_\lat$.
\end{prop}

\noindent
\begin{center}
{\bf Acknowledgements}
\end{center}

I wish to express my gratitude to Kentaro Hori for valuable discussions and
insightful comments on mirror symmetry, and Mayuko Yamashita and Yuji Tachikawa on supersymmetric QFTs, and
Tomohiro Karube and Atsushi Kanazawa on rigid Calabi-Yau manifolds.
I also wish to express my gratitude to Tomoyuki Arakawa, Maxim Kontsevich, Todor Milanov, Yuki Koto, Shintaro Yanagida, Thomas Creutzig,
Masashi Kawahira, Takumi Maegawa, Masahiro Futaki, Christoph Schweigert, Masahito Yamazaki for sharing their knowledge and insights with me. 
I would also like to thank all the members of RIKEN iTHEMS, where I spent three years, especially our assistants and the director Tetsuo Hatsuda for providing an excellent research environment.
This work is supported by Grant-in Aid for Early-Career Scientists (24K16911) and FY2023 Incentive Research Projects (Riken).

%\section{$2d$ spin topological field theory}

\end{document}